\documentclass[12pt]{scrartcl}

\usepackage{array}
\usepackage{supertabular}
\usepackage{amsmath,amsfonts,amssymb,amsthm} 
\usepackage{bm}
\usepackage{graphicx}
\usepackage{subcaption}

\newcommand{\wai}{\sum_{i=1}^p}

\newcommand{\cald}{\mathcal{D}}
\newcommand{\argmin}{\mathop{\textrm arg~min}\limits}

\newcommand{\xx}{\bm{X}}

\newcommand{\bareu}[1]{{\bar{e}}_{#1}}
\newcommand{\bareo}[1]{{\bar{e}}^{#1}}
\newcommand{\barteta}[1]{{\bar{\theta}}^{#1}}
\newcommand{\aaa}[2]{A_{#1}^{#2}}
\newcommand{\bee}[2]{B_{#1}^{#2}}
\newcommand{\bbar}[2]{\bar{B}_{#1}^{#2}}
\newcommand{\cee}[2]{C_{#1}^{#2}}
\newcommand{\cbar}[2]{\bar{C}_{#1}^{#2}}
\newcommand{\dee}[2]{D_{#1}^{#2}}

\allowdisplaybreaks[4]

\newtheorem{theo}{Theorem}

\newtheorem{coro}{Corollary}

\title {MLE convergence speed to information projection of exponential family: Criterion for model dimension and sample size -- complete proof  version-- }
\author{Yo Sheena\thanks{Faculty of Data Science, Shiga University, Japan; Visiting Professor of the Institute of Statistical Mathematics, Japan}}

\date{May 2021}

\begin{document}
	\maketitle
	
	\begin{abstract}
		For a parametric model of distributions, the closest distribution in the model to the true distribution located outside the model is considered. Measuring the closeness between two distributions with the Kullback-Leibler (K-L) divergence, the closest distribution is called the ``information projection.'' The estimation risk of the maximum likelihood estimator (MLE) is defined as the expectation of K-L divergence between the information projection and the predictive distribution with plugged-in MLE. Here, the asymptotic expansion of the risk is derived up to $n^{-2}$-order, and the sufficient condition on the risk for the Bayes error rate between the true distribution and the information projection to be lower than a specified value is investigated. Combining these results, the ``$p-n$ criterion'' is proposed, which determines whether the MLE is sufficiently close to the information projection for the given model and sample. In particular, the criterion for an exponential family model is relatively simple and can be used for a complex model with no explicit form of normalizing constant. This criterion can constitute a solution to the sample size or model acceptance problem.  Use of the  $p-n$ criteria is demonstrated for two practical datasets. The relationship between the results and information criteria is also studied.
	\end{abstract}
	\noindent
	MSC(2010) \textit{Subject Classification}: Primary 60F99; Secondary 62F12
	\section{Introduction}
	\label{section:int}
	\ \ Given a certain data set, an unknown probability distribution that generates the data as the independent, identically distributed (i.i.d.) sample can be assumed. Under this assumption, if a certain parametric distribution model is adopted to ``explain'' the data, the first task is to find the ``best'' distribution in the model. Because the true distribution is assumed to be outside the model (except for some rare cases),  the ``best'' means the ``closest'' to the true distribution.
	
	If the true distribution is successfully approximated by the ``best'' distribution, it has many possible applications. For example, regression or discrimination analysis can be performed based on the conditional distribution of one variable in the distribution (target variable) with respect to other variables (the explanatory variables). The conditional or unconditional distributions can also be used to complete missing values with multiple candidates (multiple imputation). Note that it is possible to decide whether an individual is an outlier based on a contour region of a certain probability. In essence, we can answer any type of question on the true distribution theoretically or, in most cases, numerically using the generated  random variables from the approximating distribution.
	
	The most important merit of the approximating distribution is that it naturally provides ``knowledge of the amount of uncertainty'' in the famous equation of C. R. Rao \cite{Rao}: 
	\begin{quote}
		``Uncertain knowledge'' + ``Knowledge of the amount of uncertainty in it'' = ``Usable knowledge''
	\end{quote}
	For example, for a prediction based on a regression analysis using the conditional distribution of the target variable with respect to the explanatory variables, the target variable value can be predicted and  its prediction interval can also be constituted. The multiple imputation method is preferred to the single approach, as it reflects the likelihood of each imputed value.
	
	During the true distribution approximation process, significant problems arise regarding the methods for the following:
	\begin{enumerate}
		\item systematic construction of a distribution model; 
		\item evaluation of the estimator closeness to the best distribution; 
		\item evaluation of the best distribution closeness to the true distribution.
	\end{enumerate}
	
	This study focuses on the second problem, with the aim of establishing a criterion to determine whether the maximum likelihood estimator (MLE) is sufficiently close to the best distribution. The result for a general distribution model is stated, but the main focus is on an exponential family model, for which a concise criterion is presented. We also address the third problem in relation to the information criteria.
	
\ \ For the finite-dimensional exponential family model, the first problem equates to the basis function selection. Portnoy \cite{Portnoy}, Stone \cite{Stone1}, and Barron and Sheu \cite{Barron&Sheu} have investigated cases involving a series of exponential families on a one-dimensional compact set with splines, polynomials, and trigonometric functions as the basis functions. Those researchers focused on the convergence rate of the predictive distribution to the true distribution as the basis-function dimension increased with the sample size. Further, Wainwright and Jordan \cite{Wainwright&Jordan} extensively studied basis function selection in the context of graphical models. Sundberg \cite{Sundberg} produced a comprehensive book on exponential family models and introduced many model types for various fields. In addition, Efron and Tibshirani \cite{Efron&Tibshirani} studied the hybrid construction of an exponential family using a nonparametric reference function and finite-dimensional basis functions. For recent developments regarding exponential families in association with the reproducing kernel Hilbert space, see \cite{Canu&Smola}, \cite{Fukumizu}, \cite{Sriperumbudur_et_al} and \cite{Arbel&Gretton}. 
	
	However, most of the asymptotic results reported in the above-mentioned papers pertain to investigations of the closeness between the predicted and the true distributions (not the best distribution). These results were related to the consistency or convergence order in accordance with the model inflation, along with the sample size. (Note that Barron and Sheu \cite{Barron&Sheu} studied the convergence of the predictive distribution to the best distribution in a theorem proof; however, their main concern was the distance to the true distribution.) 
	The approach considered in this work is characterized by separation of the second and third problems, focusing on the second problem.
	To address the second problem, we fix the model and derive the asymptotic expansion of the risk with respect to the sample size, $n$ (i.e., the expected distance between the predicted and best distributions). The approximated risk (up to the $n^{-1}$ or $n^{-2}$-order) yields a criterion for the second problem if it is combined  with a certain threshold $C$. 

	\subsection{Framework}
	The framework of this study is as follows. Consider the following parametric distribution model:
	\[
	\mathcal{M} = \{g(x;\theta)\,|\, \theta=(\theta^1,\ldots,\theta^p) \in \Theta\},
	\]
	where $g(x;\theta)$ is the probability density function (p.d.f.) with respect to a reference measure $d\mu$ on a measurable space.  The p.d.f. of the unknown true distribution with respect to $d\mu$ is denoted by $g(x)$.
	
	The Kullback--Leibler divergence (K-L divergence) is used to measure the closeness between two distributions, and the MLE is chosen as the estimator.  This pair, i.e., the K-L divergence and MLE, is the natural choice for the following reasons. 
	
	First, the divergence is a geometrical tool that is independent of the parameter  (i.e., the coordinate system with respect to the differential manifold $\mathcal{M}$). This allows extraction of purely geometrical results; in other words, result dependence on parameter choices can be avoided.  Second, K-L divergence is essentially the only ``decomposable,'' ``flat,'' and ``invariant'' divergence (see Theorem 4.1 of \cite{Amari4}). Invariance is especially important for comparing two distributions, because for a one-to-one transfer of the observed variable, the results should remain unchanged even with the transformed variable (see \cite{Vajda} for other important divergence properties).
	
	K-L divergence in the context of the ``best'' distribution in the model is also explained here. First, consider $\alpha$-divergence, i.e., a class of divergences with one parameter ($\alpha$) defined by
	\begin{equation}
		\label{alphadive}
		\overset{\tiny\alpha}{D}[g(x): g(x;\theta)]=
		\begin{cases}
			\frac{4}{1-\alpha^2}\Bigl\{ 1- \int \bigl(g(x)\bigr)^{(1-\alpha)/2}  \bigl(g(x;\theta)\bigr)^{(1+\alpha)/2} d\mu \Bigr\}, & \text{ if $\alpha \ne \pm 1,$}\\
			\int g(x;\theta) \log \Bigl(g(x;\theta)/g(x) \Bigr)d\mu, & \text{ if  $\alpha=1$,}\\
			\int g(x) \log \Bigl(g(x)/g(x;\theta) \Bigr) d\mu, & \text{ if $\alpha=-1$.}
		\end{cases}
	\end{equation}
	Note that $\alpha$-divergence is the only ``decomposable,'' ``flat,'' and ``invariant'' divergence when it is extended on the positive measure space (see Theorem 4.2 of \cite{Amari4}). Further, $\alpha$-divergence contains frequently used divergences such as the K-L divergence ($\alpha =-1$), Hellinger distance ($\alpha=0$), and $\chi^2$-divergence ($\alpha=3$).  
	
	The ``best'' approximating distribution in $\mathcal{M}$ is the closest distribution $g(x;\theta_*)$ to $g(x)$, where 
	\begin{equation}
		\label{center_model}
		\theta_* = \argmin_{\theta \in \Theta} \overset{\tiny\alpha}{D}[ g(x) \,|\,g(x;\theta)] .
	\end{equation}
	Csisz\'{a}r \cite{Csiszar1} called $g(x;\theta_*)$ the ``information projection.''  If the parametric model is regarded as a device for approximation of the true distribution, $g(x;\theta_*)$ is the best distribution in $\mathcal{M}$ for this purpose. 
	
	In fact, $g(x;\theta_*)$ is essentially given by the solution of the equations
	\begin{align}
		&\frac{\partial\ }{\partial \theta^i} \overset{\tiny \alpha}{D}[ g(x) \,|\,g(x;\theta)] = 0, \quad i=1,\ldots,p \label{solution_center}\\
		&\Longleftrightarrow 
		\begin{cases}
			\int \bigl(g(x;\theta)\bigr)^{\frac{\alpha-1}{2}} \bigl(g(x)\bigr)^{\frac{1-\alpha}{2}}\frac{\partial g(x;\theta)  }{\partial \theta^i} d\mu =0, \quad i=1,\ldots,p, &\text{if $\alpha \ne 1$,}
			\\
			\\ 
			\int \frac{\partial g(x;\theta)  }{\partial \theta^i} \log\bigl(g(x;\theta)/g(x)\bigr) d\mu =0, \quad i=1,\ldots,p, &\text{if $\alpha$=1.}
		\end{cases}
		\label{solution_center_2}
	\end{align}
	Note that, if $\alpha=-1$, equation \eqref{solution_center} is equivalent to
	\begin{equation}
		\label{solution_center_3}
		E \left[ \frac{\partial \ }{\partial \theta^i}\log g(X;\theta) \right] =0,\quad i=1,\ldots,p,
	\end{equation}
	and its solution $\theta^*$ is estimated via the MLE, which is the solution of
	\[
	\sum_{t=1}^n  \frac{\partial \ }{\partial \theta^i}\log g(X_t;\theta) = 0,\quad i=1,\ldots,p
	\]
	where $X_t,\ t=1,\ldots,n$ is the i.i.d. sample from $g(x)$.
	If $\alpha \ne -1$, equation \eqref{solution_center} does not have the form  
	\[
	E[ h(X, \theta) ] = 0
	\]
	with some known function $h(x, \theta)$, which is an M-estimator formulation.
	
	An overview of the contents of each section is now provided. Throughout this paper, the focus is on the ``distance''  between $g(x;\theta_*)$ and the predictive distribution $g(x;\hat\theta)$ with MLE $\hat{\theta}$, i.e.,
	\[
	D[g(x;\theta_*) \,| \, g(x; \hat{\theta})]
	\]
	as well as the ``estimation risk'' 
	\[
	R[g(x;\theta_*) \,| \, g(x; \hat{\theta})] \triangleq
	E\Bigl[D[g(x;\theta_*) \,| \, g(x; \hat{\theta})] \Bigr].
	\]
	It is known that $\hat{\theta}$ converges in probability to $\theta_*$. The estimation risk convergence to zero with increasing $n$ (and a fixed model) is investigated. Further, the asymptotic expansion of the risk with respect to $n$ (Section \ref{Risk}) is derived for both a general model (Section \ref{Risk_g}) and an exponential family model (Section \ref{Risk_e}).

	In Section \ref{Criteria}, the criterion to determine whether the MLE is sufficiently close to the information projection of the model is considered; namely, the ``$p-n$ criterion.'' In other words, this criterion indicates whether $n$ is sufficiently large for the given model, or if the model dimension is sufficiently small for the given $n$. In Section \ref{p-n_criteria}, estimation of the unknown elements that appear in the asymptotic expansion of the estimation risk in Section \ref{Risk} is shown. A method for setting a $C$ for the estimation risk in relation to the Bayes error rate is also proposed. In Section \ref{est_pro}, the algorithm for calculating the $p-n$ criterion in the case of an exponential family is described. In Section \ref{example_sec}, use of the $p-n$ criterion is demonstrated for two practical examples. 
	
	Finally, Section \ref{aic_tic} treats the ``total risk'' 
	\[
	E[D[g(x) \,|\,g(x;\hat{\theta}) ]];
	\]
	that is, the expected distance between $g(x)$ and $g(x;\hat{\theta})$.  By decomposing the total risk into the estimation risk and  ``approximation risk,''
	\[
	D[g(x)\, |\, g(x;\theta_*)],
	\]
	the relationships between the obtained result and Takeuchi's information criteria (TIC) and Akaike's information criteria (AIC) are explained.
	
	Throughout this study, the expectation over $\xx$ under $g(x)$ is denoted by $E[\cdot]$, while the expectation under $g(x;\theta_*)$ is denoted by $E_{\theta_*}[\cdot]$.
	\section{Estimation Risk for General Case and Exponential Family}
	\label{Risk}
	\ \ The convergence speed of $g(x;\hat\theta)$ to $g(x;\theta_*)$ is considered. The most relevant work to the present result is that of Barron and Sheu \cite{Barron&Sheu}, which considers the convergence with respect to the K-L divergence for an exponential family on a compact set. In particular, those researchers considered an exponential family with polynomials, splines, and trigonometric functions as the basis functions, and obtained the convergence order of the divergence itself (rather than the risk) as $p$  and $n$ increased simultaneously under the condition $p^2/n \to 0$. (For K-L divergence and nonparametric density estimation in general, see \cite{Hall}.)
	
	In this section, the concrete terms of the estimation-risk asymptotic expansion are derived as preparation for the criterion based on the relationship between $p$ and $n$, which is discussed in the next section.
	\subsection{Estimation Risk for General Case}
	\label{Risk_g}
	\ \ The asymptotic expansion of 
	\begin{equation}
		R[g(x;\theta_*)\,|\,g(x;\hat{\theta}) ] = E[D[g(x;\theta_*)\,|\,g(x;\hat{\theta}) ] 
	\end{equation}
	is derived up to the second-order term with respect to $n$. 
	
	Taylor expansion of  
	\[
	D[g(x;\theta_*)\,|\,g(x;\hat{\theta}) ]
	= \int g(x;\theta_*) \log (g(x;\theta_*) /g(x;\hat{\theta}) ) d\mu 
	\]
	as a function of $\hat{\theta}$ around $\theta_*$ is considered:
	\begin{align}
		&D[g(x;\theta_*)\,|\,g(x;\hat{\theta}) ] \nonumber\\
		&= -\sum_i \int  \frac{\partial \ }{\partial \theta^i} g(x;\theta)\bigg|_{\theta=\theta_*} d\mu\  (\hat{\theta}^i - \theta_*^i) \nonumber\\
		&\quad + \frac{1}{2} \sum_{i,j} \int g(x;\theta_*)\Bigl( \frac{\partial \ }{\partial \theta^i}\log g(x;\theta) \bigg|_{\theta=\theta_*}\Bigr) 
		\Bigl( \frac{\partial \ }{\partial \theta^j} \log g(x;\theta) \bigg|_{\theta=\theta_*}\Bigr) d\mu \nonumber\\
		&\qquad \times  (\hat{\theta}^i - \theta_*^i) (\hat{\theta}^j - \theta_*^j) \label{talyor_D}\\
		&\quad - \frac{1}{2}\sum_{i,j} \int \frac{\partial^2 \ \ \  }{\partial \theta^i \partial \theta^j} g(x, \theta)\bigg|_{\theta=\theta_*} d\mu \  (\hat{\theta}^i - \theta_*^i) (\hat{\theta}^j - \theta_*^j)\nonumber\\
		&\quad - \sum_{t=3}^{\infty} \frac{1}{t!}\sum_{i_1,\ldots,i_t}\int g(x;\theta_*) \frac{\partial\hfill}{\partial \theta^{i_1}\cdots \partial \theta^{i_t}} \log g(x, \theta) \bigg |_{\theta=\theta_*} d\mu \nonumber\\
		&\qquad \times (\hat{\theta}^{i_1} - \theta_*^{i_1})\cdots (\hat{\theta}^{i_t} - \theta_*^{i_t}).
		\nonumber
	\end{align}
	The equation
	\[
	\int g(x, \theta) d\mu = 1,\quad \forall \theta \in \Theta,
	\]
	yields
	\begin{equation}
		\label{deri_0}
		\int \frac{\partial \ }{\partial \theta^i} g(x;\theta) d\mu =0,\qquad \int \frac{\partial^2 \ \ }{\partial \theta^i \partial \theta^j} g(x, \theta) d\mu=0,\qquad \forall \theta \in \Theta.
	\end{equation}
	Therefore, 
	\begin{align}
		R[g(x;\theta_*)\,|\,g(x;\hat{\theta}) ] 
		&= \frac{1}{2}\sum_{i,j} g^*_{ij}(\theta_*)  E[(\hat{\theta}^i - \theta_*^i) (\hat{\theta}^j - \theta_*^j)] \nonumber\\
		&\quad - \sum_{t=3}^{\infty} \frac{1}{t!}\sum_{i_1,\ldots,i_t} \tau_{i_1,\ldots,i_t}E[(\hat{\theta}^{i_1} - \theta_*^{i_1})\cdots (\hat{\theta}^{i_t} - \theta_*^{i_t})].
		\label{talyor_D_2}
	\end{align}
	Here, 
	\[
	\tau_{i_1,\ldots,i_t} \triangleq \int g(x;\theta_*) \frac{\partial\hfill}{\partial \theta^{i_1}\cdots \partial \theta^{i_t}} \log g(x, \theta) \bigg |_{\theta=\theta_*} d\mu
	\]
	and $g^*_{ij}$ indicates the components of the Fisher metric matrix on $\mathcal{M}$, given by
	\begin{equation}
		\label{def_g*}
		\begin{split}
			g^*_{ij}(\theta_*)\triangleq (G^*(\theta_*))_{ij} 
			& \triangleq E_{\theta_*}\Bigl[\Bigl( \frac{\partial \ }{\partial \theta^i}\log g(x;\theta) \Big|_{\theta=\theta_*}\Bigr) 
			\Bigl( \frac{\partial \ }{\partial \theta^j} \log g(x;\theta)\Big|_{\theta=\theta_*} \Bigr) \Bigr]\\
			& = \int g(x;\theta_*)\Bigl( \frac{\partial \ }{\partial \theta^i}\log g(x;\theta) \Big|_{\theta=\theta_*}\Bigr) 
			\Bigl( \frac{\partial \ }{\partial \theta^j} \log g(x;\theta)\Big|_{\theta=\theta_*} \Bigr) d\mu\\
			&=- \int g(x;\theta_*) \frac{\partial^2 \hfill }{\partial \theta^i \partial \theta^j}\log g(x;\theta) \Big|_{\theta=\theta_*}
			d\mu.
		\end{split}
	\end{equation}
	The second equation is derived from \eqref{deri_0}.
	
	As $\theta_*$ is the solution of equation \eqref{solution_center_3} and $\hat{\theta}$ is its empirical solution (i.e., the M-estimator), the following result holds (see, e.g., Theorem 5.21 of \cite{Vaart}). 
	\begin{equation}
		\label{asymp_normal}
		\sqrt{n}\bigl(\hat{\theta}-\theta_*\bigr) \stackrel{d}{\to} N_p (0, \tilde{G}^{-1}G\tilde{G}^{-1}),
	\end{equation}
	where 
	\begin{align}
		\label{def_g}
		g_{ij}(\theta_*) \triangleq (G(\theta_*))_{ij} & \triangleq E\Bigl[\Bigl( \frac{\partial \ }{\partial \theta^i}\log g(X;\theta)\Big|_{\theta=\theta_*} \frac{\partial \ }{\partial \theta^j}\log g(X;\theta)\Big|_{\theta=\theta_*}\Bigr)\Bigr], \\
		\label{def_tilde_g}
		\tilde{g}_{ij}(\theta_*) \triangleq (\tilde{G}(\theta_*))_{ij} &\triangleq -E\Bigl[\frac{\partial^2 \ \quad} {\partial \theta^j \partial \theta^i} \log g(X;\theta)\Big|_{\theta=\theta_*}\Bigr] .
	\end{align}
	%
	%
	% Theorem 1
	%
	%
	%
	The following notation is defined, for $1\leq i, j, k, l\leq p$:
	\begin{align*}
		&L_{(ijk)}\triangleq E\Bigl[\frac{\partial^3 \log g(x;\theta_*)}{\partial \theta^i \partial \theta^j \partial \theta^k} \Bigr] \\
		&L_{(ij)k}\triangleq E\Bigl[\frac{\partial^2 \log g(x;\theta_*)}{\partial \theta^i \partial \theta^j} \frac{\partial \log g(x;\theta_*)}{\partial \theta^k}\Bigr] \\
		& L_{ijk}\triangleq E\Bigl[\frac{\partial \log g(x;\theta_*)}{\partial \theta^i} \frac{\partial \log g(x;\theta_*)}{\partial \theta^j}\frac{\partial \log g(x;\theta_*)}{\partial \theta^k}\Bigr] \\
		&L_{(ijkl)}\triangleq E\Bigl[\frac{\partial^4 \log g(x;\theta_*)}{\partial \theta^i \partial \theta^j \partial \theta^k \partial \theta^l} \Bigr] \\
		&L_{(ij)(kl)}\triangleq E\Bigl[\frac{\partial^2 \log g(x;\theta_*)}{\partial \theta^i \partial \theta^j} \frac{\partial^2 \log g(x;\theta_*)}{\partial \theta^k \partial \theta^l} \Bigr] \\
		&L_{(ijk)l}\triangleq E\Bigl[\frac{\partial^3 \log g(x;\theta_*)}{\partial \theta^i \partial \theta^j \partial \theta^k} \frac{\partial \log g(x;\theta_*)}{\partial \theta^l}\Bigr] \\
		&L_{(ij)kl} \triangleq E\Bigl[\frac{\partial^2 \log g(x;\theta_*)}{\partial \theta^i \partial \theta^j}  \frac{\partial \log g(x;\theta_*)}{\partial \theta^k}\frac{\partial \log g(x;\theta_*)}{\partial \theta^l}\Bigr].
	\end{align*}
	The next theorem states the estimation-risk asymptotic expansion up to the term $n^{-2}$ for a general distribution model. For brevity, Einstein notation is used and the dependency on $\theta_*$ is omitted; e.g., $G$ for $G(\theta_*)$ and $\tilde{g}^{ij}$ for $\tilde{g}^{ij}(\theta_*)$.
	\begin{theo}
		\label{theo_general_risk}
		The MLE estimation risk with respect to K-L divergence is given by
		\begin{equation}
			\label{expan_est_risk}
			\begin{split}
				&R[g(x;\theta_*)\,|\,g(x;\hat{\theta}) ] \\
				&= (2n)^{-1} \mathrm{tr} \Bigl(\tilde{G}^{-1}G\tilde{G}^{-1}G^* \Bigr) \\
				&\quad+n^{-2} \Bigl[2^{-1}g^*_{ij}\Bigl(
				\tilde{g}^{sj}\tilde{g}^{it}\tilde{g}^{lm}(L_{(sl)tm}+\tilde{g}_{ls}g_{tm})
				+\tilde{g}^{si}\tilde{g}^{jt}\tilde{g}^{lm}(L_{(sl)tm}+\tilde{g}_{ls}g_{tm})\\
				&\qquad\qquad+ 2^{-1}\tilde{g}^{uj}\tilde{g}^{ik}\tilde{g}^{ls}\tilde{g}^{mt}L_{kst}L_{(lmu)}
				+2^{-1}\tilde{g}^{ui}\tilde{g}^{jk}\tilde{g}^{ls}\tilde{g}^{mt}L_{kst}L_{(lmu)}\\
				&\qquad\qquad+\tilde{g}^{jk}\tilde{g}^{lu}\tilde{g}^{is}\tilde{g}^{mt} 
				(L_{(kl)(um)}g_{st}-\tilde{g}_{kl}\tilde{g}_{um}g_{st}+L_{(kl)s}L_{(um)t}+L_{(kl)t}L_{(um)s})\\
				&\qquad\qquad+\tilde{g}^{ik}\tilde{g}^{lu}\tilde{g}^{js}\tilde{g}^{mt} 
				(L_{(kl)(um)}g_{st}-\tilde{g}_{kl}\tilde{g}_{um}g_{st}+L_{(kl)s}L_{(um)t}+L_{(kl)t}L_{(um)s})\\
				&\qquad\qquad+2^{-1}\tilde{g}^{jk}\tilde{g}^{it}\tilde{g}^{mu}\tilde{g}^{sv}\tilde{g}^{wl}L_{(msw)}(L_{(lk)t}g_{uv}+L_{(lk)u}g_{tv}+L_{(lk)v}g_{tu})\\
				&\qquad\qquad+2^{-1}\tilde{g}^{ik}\tilde{g}^{jt}\tilde{g}^{mu}\tilde{g}^{sv}\tilde{g}^{wl}L_{(msw)}(L_{(lk)t}g_{uv}+L_{(lk)u}g_{tv}+L_{(lk)v}g_{tu})\\
				&\qquad\qquad+2^{-1}\tilde{g}^{js}\tilde{g}^{it}\tilde{g}^{lu}\tilde{g}^{mv}
				(L_{(slm)t}g_{uv}+L_{(slm)u}g_{tv}+L_{(slm)v}g_{tu})\\
				&\qquad\qquad+2^{-1}\tilde{g}^{is}\tilde{g}^{jt}\tilde{g}^{lu}\tilde{g}^{mv}
				(L_{(slm)t}g_{uv}+L_{(slm)u}g_{tv}+L_{(slm)v}g_{tu})\\
				&\qquad\qquad+\tilde{g}^{mt}\tilde{g}^{iu}\tilde{g}^{lv}\tilde{g}^{sw}\tilde{g}^{kj}L_{(lmk)}(L_{(ts)u}g_{vw}
				+L_{(ts)v}g_{uw}+L_{(ts)w}g_{uv})\\
				&\qquad\qquad+\tilde{g}^{mt}\tilde{g}^{ju}\tilde{g}^{lv}\tilde{g}^{sw}\tilde{g}^{ki}L_{(lmk)}(L_{(ts)u}g_{vw}
				+L_{(ts)v}g_{uw}+L_{(ts)w}g_{uv})\\
				&\qquad\qquad+2^{-1}\tilde{g}^{ik}\tilde{g}^{lu}\tilde{g}^{sv}\tilde{g}^{tw}\tilde{g}^{oj}\tilde{g}^{hm}L_{(lmo)}L_{(sth)}
				(g_{ku}g_{vw}+g_{kv}g_{uw}+g_{kw}g_{uv})\\
				&\qquad\qquad+2^{-1}\tilde{g}^{jk}\tilde{g}^{lu}\tilde{g}^{sv}\tilde{g}^{tw}\tilde{g}^{oi}\tilde{g}^{hm}L_{(lmo)}L_{(sth)}
				(g_{ku}g_{vw}+g_{kv}g_{uw}+g_{kw}g_{uv})\\
				&\qquad\qquad+6^{-1}\tilde{g}^{ik}\tilde{g}^{ls}\tilde{g}^{mu}\tilde{g}^{tv}\tilde{g}^{wj}L_{(lmtw)}
				(g_{ks}g_{uv}+g_{ku}g_{sv}+g_{kv}g_{su})\\
				&\qquad\qquad+6^{-1}\tilde{g}^{jk}\tilde{g}^{ls}\tilde{g}^{mu}\tilde{g}^{tv}\tilde{g}^{wi}L_{(lmtw)}
				(g_{ks}g_{uv}+g_{ku}g_{sv}+g_{kv}g_{su})\\
				&\qquad\qquad+\tilde{g}^{ik}\tilde{g}^{js}\tilde{g}^{lt}\tilde{g}^{mu}
				(L_{(kl)(sm)}g_{tu}-\tilde{g}_{kl}\tilde{g}_{sm}g_{tu}+L_{(kl)t}L_{(sm)u}+L_{(kl)u}L_{(sm)t})\\
				&\qquad\qquad+2^{-1}\tilde{g}^{ik}\tilde{g}^{lu}\tilde{g}^{sv}\tilde{g}^{tw}\tilde{g}^{jm}L_{(stm)}
				(L_{(kl)u}g_{vw}+L_{(kl)v}g_{uw}+L_{(kl)w}g_{uv})\\
				&\qquad\qquad+2^{-1}\tilde{g}^{jk}\tilde{g}^{lu}\tilde{g}^{sv}\tilde{g}^{tw}\tilde{g}^{im}L_{(stm)}
				(L_{(kl)u}g_{vw}+L_{(kl)v}g_{uw}+L_{(kl)w}g_{uv})\\
				&\qquad\qquad+4^{-1}\tilde{g}^{lk}\tilde{g}^{mu}\tilde{g}^{sv}\tilde{g}^{tw}\tilde{g}^{io}\tilde{g}^{jh}L_{(lmo)}L_{(sth)}
				(g_{ku}g_{vw}+g_{kv}g_{uw}+g_{kw}g_{uv})
				\Bigr)\\
				&\qquad\qquad-6^{-1}\tau_{ijk}\Bigl(\tilde{g}^{is}\tilde{g}^{jt}\tilde{g}^{ku}L_{stu} \\
				&\qquad\qquad+\tilde{g}^{it}\tilde{g}^{su}\tilde{g}^{jv}\tilde{g}^{kw}(L_{(st)u}g_{vw}+L_{(st)v}g_{uw}+L_{(st)w}g_{uv})\\
				&\qquad\qquad+\tilde{g}^{jt}\tilde{g}^{su}\tilde{g}^{iv}\tilde{g}^{kw}(L_{(st)u}g_{vw}+L_{(st)v}g_{uw}+L_{(st)w}g_{uv})\\
				&\qquad\qquad+\tilde{g}^{kt}\tilde{g}^{su}\tilde{g}^{iv}\tilde{g}^{jw}(L_{(st)u}g_{vw}+L_{(st)v}g_{uw}+L_{(st)w}g_{uv})\\
				&\qquad\qquad+2^{-1} \tilde{g}^{su}\tilde{g}^{tv}\tilde{g}^{jw}\tilde{g}^{km}\tilde{g}^{il}L_{(stl)}(g_{uv}g_{wm}+g_{uw}g_{vm}+g_{um}g_{vw})\\
				&\qquad\qquad+2^{-1} \tilde{g}^{su}\tilde{g}^{tv}\tilde{g}^{iw}\tilde{g}^{km}\tilde{g}^{jl}L_{(stl)}(g_{uv}g_{wm}+g_{uw}g_{vm}+g_{um}g_{vw})\\
				&\qquad\qquad+2^{-1}\tilde{g}^{su}\tilde{g}^{tv}\tilde{g}^{iw}\tilde{g}^{jm}\tilde{g}^{kl}L_{(stl)}(g_{uv}g_{wm}+g_{uw}g_{vm}+g_{um}g_{vw})\Bigr)\\
				&\qquad\qquad-24^{-1}\tau_{ijkl}\tilde{g}^{is}\tilde{g}^{jt}\tilde{g}^{ku}\tilde{g}^{lv}(g_{st}g_{uv}+g_{su}g_{tv}+g_{sv}g_{tu})
				\Bigr]\\
				&\quad + O(n^{-3}).
			\end{split}
		\end{equation}
	\end{theo}
	\begin{proof}
		This proof is omitted as it is lengthy; however, it is given in Section \ref{Proof_Theo} of the Appendix.
	\end{proof}
	As shown by Efron \cite{Efron} and Amari \cite{Amari0}, the first- and second-order terms ($n^{-1}$- and $n^{-2}$-order terms, respectively) are respectively related to the metrics and connections in the extended manifold that includes both $g(x)$ and $\mathcal{M}$. This issue is not investigated here. Note only that, if $g(x)$ exists within the model, then $G=\tilde{G}=G^*$; hence, the first-order term equals $p/(2n)$ (see also \cite{Sheena2}). Thus, the first-order term is mainly determined by $p$ if $g(x;\theta_*)$ is close to $g(x)$.
	
	The second-order term is very complex, and the calculation of a specific distribution model requires extensive computational resources. Because it is difficult to use the term for practical purposes, the first-order term is the focus of the remainder of this section; it is studied using some examples.
	%
	%
	%
	% Example 1
	%
	%
	%
	
	First, a normal regression model is taken as an example and the first-order term of Theorem \ref{theo_general_risk} is applied. \\
	\\
	-- \textit{Example 1: Normal regression model } -- 
	
	Let $h(x)$ be the true p.d.f. of the explanatory variables $X \triangleq(X_1,\ldots,X_p) $ with respect to the Lebesgue measure. Consider the following normal regression model:
	\[
	Y = \sum_{i=1}^p \theta^i X_i + \epsilon, \qquad \epsilon \sim N(0, (\theta^0)^{-1}),
	\]
	where $X$ and $\epsilon$ are independently distributed. 
	The parametric distribution model of  $(Y,X)$ is given by
	\[
	\mathcal{M} =   \Bigl \{ g(y,x; \theta) \Big| \, \theta^0 > 0, \ -\infty < \theta^i  < \infty,\ i=1,\ldots,p \Bigr\},
	\]
	where
	\[
	g(y,x ;\theta) = \exp\Bigl(-\frac{\theta^0}{2}\bigl(y- \sum_{i=1}^p \theta^i x_i \bigr)^2 -\frac{1}{2}\log(2\pi)+\frac{1}{2}\log \theta^0 \Bigr),
	\]
	and $d\mu = h(x) dx.$
	
	Under the true distribution for $(Y,X)$, consider the distribution of the random variable
	\[
	\epsilon(Y,X; \theta_*) \triangleq Y - \sum_{i=1}^p \theta_*^i X_i.
	\]
	As the true distribution of $(Y,X)$ is outside $\mathcal{M}$, $\epsilon(Y,X; \theta_*)$ is differently distributed from the normal regression model assumption. That is, the assumption that $\epsilon(Y,X:\theta_*)$ is independent of $X$ and is normally distributed is invalid. Let us consider the following three non-normal cases for the distribution of $\epsilon(Y, X;\theta_*)$. The higher the case, the greater the discrepancy between the true distribution and the regression model assumption. \\
	Case 1: $\epsilon(Y,X;\theta_*)$ is correlated with $X$.\\
	Case 2: $\epsilon(Y,X;\theta_*)$ is independently distributed from $X$ but has different moments from those of $N(0,(\theta^0_*)^{-1})$. \\
	Case 3: $\epsilon(Y,X;\theta_*)$ is independently distributed from $X$ and has the same kurtosis as $N(0,(\theta^0_*)^{-1})$; i.e.,
	\[
	\frac{E[\epsilon^4(Y,X;\theta_*)]}{E^2[\epsilon^2(Y,X;\theta_*)]}=3.
	\]
	\\
	For each case, the asymptotic risk is determined as follows (for the derivation, see \ref{nomral_risk_deriv} of the Appendix).
	For Case 1,
	\begin{equation}
		\label{Risk_norm_case1}
		\begin{split}
			&R[g(x;\theta_*)\,|\,g(x;\hat{\theta}) ] \\
			&=\frac{1}{2n} \Bigl(  \mathrm{tr} \Bigl(S^{-1} T\Bigr) /E[\epsilon^2(Y,X;\theta_*)]
			+\frac{1}{2}\Bigl(E[\epsilon^4(Y,X;\theta_*)]/E^2[\epsilon^2(Y,X;\theta_*)]-1\Bigr)\Bigr)+ o(n^{-1}),
		\end{split}
	\end{equation}
	where $(S)_{ij} \triangleq E[X_i X_j],\ (T)_{ij} \triangleq E[X_i X_j \epsilon^2(Y,X;\theta_*)],\quad i,j=1,\ldots,p$.  \\
	For Case 2,
	\begin{equation}
		\label{Risk_norm_case2}
		R[g(x;\theta_*)\,|\,g(x;\hat{\theta}) ]= \frac{1}{2n} \Bigl(p +\frac{1}{2}\Bigl(E[\epsilon^4(Y,X;\theta_*)]/E^2[\epsilon^2(Y,X;\theta_*)]-1\Bigr)\Bigr)+ o(n^{-1}).
	\end{equation}
	For Case 3, 
	\begin{equation}
		\label{Risk_norm_case3}
		R[g(x;\theta_*)\,|\,g(x;\hat{\theta}) ]= \frac{p+1}{2n} + o(n^{-1}).
	\end{equation}
	Hence,
	\begin{itemize}
		%\item Even if the true distribution exists outside the model, the first-order term could attain the same value $(p+1)/(2n)$ as the case that true distribution exists within the model (see Case 3).  
		\item If the kurtosis of $\epsilon(Y,X;\theta_*)$ is less than 3 (``platykurtic'') in \eqref{Risk_norm_case2}, the first-order term is clearly less than $(p+1)/(2n)$. This means that the risk converges faster than the case in which the model contains the true distribution.
	\end{itemize}
Note that the small estimation risk does not guarantee the small total risk. In this example, if the kurtosis of $\epsilon$ under $g(x)$ is smaller than that under $g(x;\theta_*)$, it is an indication of the discrepancy between the two distributions and may produce the large approximation risk (see Section \cite{aic_tic}).
	%
	%
	%
	% Example 2
	%
	%
	%
	
	Next, the Poisson regression model is considered as another example.
	\\
	\\
	-- \textit{Example 2: Poisson Regression Model } --
	
	Let $h(x)$ be the true p.d.f. of the explanatory variables $X \triangleq(X_1,\ldots,X_p) $ with respect to the Lebesgue measure. Suppose that, when $X$ is given as $x=(x_1,\ldots,x_p)$, $Y$ is distributed as the Poisson distribution with mean 
	\[
	\lambda(x;\theta) = \exp\Bigl(\sum_{i=1}^p \theta^i x_i\Bigr).
	\]
	Then, the distribution model of $(Y,X)$ is given as the p.d.f. of the form
	\[
	g(y,x | \theta) =\exp\Bigl( \sum_{i=1}^p \theta^i x_i y - \lambda\Bigr) = \lambda^y \exp(-\lambda)
	\]
	where the reference measure $d\mu$ is the product measure between the discrete measure $1/y!$ on $\{ y \,| \, y=0,1,2,\ldots \}$ and  the continuous measure $h(x) dx$ on $\Re^p$. 
	
	As for the true distribution, we postulate that the conditional distribution of $Y$ under $X=x$ is the Poisson distribution with mean $\lambda_0(x) = \exp(\xi(x))$. This is different from the model in that the conditional log mean of $Y$ is nonlinear. 
	
	As
	\[
	\frac{\partial \hfill}{\partial \theta^i} \log g(y,x|\theta) =x_i y - \lambda(x;\theta) x_i,
	\quad \frac{\partial^2 \hfill}{\partial \theta^i \partial \theta^j} \log g(y,x|\theta) = -\lambda(x;\theta) x_i x_j,
	\]
	for $i, j=1\ldots,p$,
	\begin{align*}
		\tilde{g}_{ij}(\theta_*) &= E[X_i X_j \lambda(X;\theta_*)] = g^*_{ij}(\theta_*),\\
		g_{ij}(\theta_*) &= E[X_i X_j (Y - \lambda(X;\theta_*))^2]\\
		&= E[X_i X_j Y^2] - 2E[X_i X_j Y \lambda(X;\theta_*)] + E[X_i X_j \lambda^2(X;\theta_*)] \\
		&= E[X_i X_j (\lambda_0(X) + \lambda_0^2(X))] - 2E[X_i X_j \lambda_0(X) \lambda(X;\theta_*)] +E[X_i X_j \lambda^2(X;\theta_*)]\\
		&=  E[X_i X_j (\lambda_0(X) + (\lambda(X;\theta_*)-\lambda_0(X))^2)].
	\end{align*}
	Consequently, $\tilde{G}(\theta_*)=G^*(\theta_*)$ and 
	\begin{align*}
		\mathrm{tr} \Bigl(\tilde{G}(\theta_*)^{-1}G(\theta_*)\tilde{G}(\theta_*)^{-1}G^*(\theta_*) \Bigr)
		&=\mathrm{tr} \Bigl(\tilde{G}(\theta_*)^{-1}G(\theta_*) \Bigr)\\
		&=p + \mathrm{tr} \Bigl(\tilde{G}(\theta_*)^{-1}\Bigl(G(\theta_*)- \tilde{G}(\theta_*)\Bigr)\Bigr)
	\end{align*}
	and 
	\[
	\Bigl(G(\theta_*)- \tilde{G}(\theta_*)\Bigr)_{ij} = E\Bigl[X_i X_j \Bigl(\bigl(\lambda(X;\theta_*)-\lambda_0(X)\bigr)\bigl(\lambda(X;\theta_*)-\lambda_0(X)-1\bigr)\Bigr)\Bigr].
	\]
	Hence,
	\begin{itemize}
		\item If $0 < \lambda(X;\theta_*)-\lambda_0(X) < 1$ almost everywhere, the estimation risk converges faster than the case when the model includes the true distribution.
	\end{itemize}
	\subsection{Estimation Risk for Exponential Family}
	\label{Risk_e}
	\ \ This subsection investigates the estimation risk when the parametric model is an exponential family (for general references on exponential families, see \cite{Brown} and \cite{Barndorff-Nielsen}).
	% As we will explain later, an exponential family is naturally associated with K-L divergence. 
	
	Let the model $\mathcal{M}$ be given by
	\begin{equation}
		\label{manifold_exp}
		\mathcal{M} =  \Bigl\{g(x ;\theta) = \exp\Bigl(\sum_{i=1}^p \theta^i \xi_i(x) -\Psi(\theta)\Bigr) \ \Big | \theta\in \Theta\Bigr\}.
	\end{equation}
	where $\Psi(\theta)$ is the cumulant-generating function of the $\xi$ terms, such that, 
	\[
	\Psi(\theta) = \log \int \exp\Bigl(\sum_{i=1}^p \theta^i \xi_i(x)\Bigr) d\mu.
	\]
		
	The ``dual coordinate'' $\eta$ is defined as 
	\begin{equation}
		\eta_i(\theta) \triangleq \frac{\partial \Psi(\theta)}{\partial \theta^i} = E_{\theta}[\xi_i],
		\quad i=1,\ldots,p.
	\end{equation}
	In particular, from the definition of $\theta_*$ (see \eqref{solution_center_3}),
	\begin{equation}
		\eta_i^* \triangleq \eta_i(\theta_*) = E_{\theta_*}[\xi_i] =E[\xi_i], \quad i=1,\ldots,p.
	\end{equation}
	The last equation requires the means of $\xi_i$ to coincide under $g(x)$ and $g(x;\theta_*)$. It is known that $g(x;\theta_*)$ maximizes the Shannon entropy among all probability distributions for a given $E[\xi_i],\ i=1,\ldots p$ (the ``entropy maximization property'' of an exponential family; see, e.g., \cite{Wainwright&Jordan}). The K-L divergence is the difference between the cross-entropy and Shannon entropy. Another association of K-L divergence with the exponential family derives from a geometrical perspective. That is, the $\alpha$-divergence induces a corresponding ``flat'' manifold of the parametric distribution model ($\alpha$-family).  When $\alpha=1$, the divergence is the conjugate divergence of the K-L divergence, and the corresponding manifold is the exponential family (see \cite{Amari&Nagaoka}).
	
	The $\eta$ coordinate is easily estimated. In fact, $\hat{\eta}$, the MLE for $\eta$, is the sample mean of $\xi$. Hence,
	\begin{equation}
		\label{hat_eta}
		\hat{\eta}_i= \frac{\partial \Psi}{\partial \theta_i}(\hat{\theta}) = \bar{\xi}_i \bigl(\triangleq n^{-1}\sum_{t=1}^n \xi_i(X_t)\bigr).
	\end{equation}
	In contrast, $\hat\theta$ is difficult to obtain explicitly because $\Psi$ or its derivative cannot be theoretically obtained for a complex model. This could pose a serious obstacle to application of an exponential family model to a practical problem, and is discussed in Section \ref{est_pro}.
	
	Let the matrix $\Ddot{\Psi}(\theta)$ be defined by
	\begin{equation}
		\label{ddotPsi}
		(\Ddot{\Psi}(\theta))_{ij}\triangleq
		\frac{\partial^2 \Psi(\theta) }{\partial \theta^i\partial \theta^j} 
		= E_{\theta}[(\xi_i-\eta_i)(\xi_j-\eta_j)], \quad 1 \leq i,j \leq p.
	\end{equation}
	Thus, $\Ddot{\Psi}$ is a covariance matrix of the $\xi_i$ terms under $g(x;\theta)$; hence, it is positive definite. Therefore, $\Psi(\theta)$ is a convex function. The notable property 
	\[
	g^*_{ij}(\theta) =\tilde{g}_{ij}(\theta), \quad 1 \leq i,j \leq p, \quad \forall \theta 
	\]
	is proven by the fact that both sides are equal to $(\Ddot{\Psi}(\theta))_{ij}$.
	
	The following notation is used for the third- or fourth-order cumulant:
	\begin{equation}
		\label{cumu_exp}
		\begin{split}
			&\kappa_{ijk} \triangleq E[(\xi_i-\eta^*_i)(\xi_j-\eta^*_j)(\xi_i-\eta^*_k)] = L_{ijk}\\
			&\kappa^*_{ijk} \triangleq E_{\theta_*}[(\xi_i-\eta^*_i)(\xi_j-\eta^*_j)(\xi_i-\eta^*_k)]= \frac{\partial^3\Psi(\theta_*)}{\partial \theta^i \partial \theta^j \partial \theta^k}
			= -L_{(ijk)}\\
			&\kappa^*_{ijkl} \triangleq E_{\theta_*}[(\xi_i-\eta^*_i)(\xi_j-\eta^*_j)(\xi_i-\eta^*_k)(\xi_i-\eta^*_l)]\\
			&\qquad\quad-E_{\theta_*}[(\xi_i-\eta^*_i)(\xi_j-\eta^*_j)]E_{\theta_*}[(\xi_k-\eta^*_k)(\xi_l-\eta^*_l)]\\
			&\qquad\quad-E_{\theta_*}[(\xi_i-\eta^*_i)(\xi_k-\eta^*_k)]E_{\theta_*}[(\xi_j-\eta^*_j)(\xi_l-\eta^*_l)]\\
			&\qquad\quad-E_{\theta_*}[(\xi_i-\eta^*_i)(\xi_l-\eta^*_l)]E_{\theta_*}[(\xi_j-\eta^*_j)(\xi_k-\eta^*_k)]\\
			&\qquad=\frac{\partial^4\Psi(\theta_*)}{\partial \theta^i \partial \theta^j \partial \theta^k \partial \theta^l}
			= -L_{(ijkl)}
		\end{split}
	\end{equation}
	for $1\leq i, j, k,l \leq p$.
	%
	%
	% Corollary
	%
	%
	As the corollary of Theorem \ref{theo_general_risk}, the following result holds.
	\begin{coro}
		Additionally, if the parametric model is an exponential family, the estimation risk is given by
		\begin{equation}
			\label{expan_est_risk_exp}
			\begin{split}
				&R[g(x;\theta_*)\,|\,g(x;\hat{\theta}) ] \\
				&= \frac{1}{2n} \mathrm{tr} \Bigl(\tilde{G}^{-1}G \Bigr) \\
				&\quad+\frac{1}{24n^2}\Bigl[-8\tilde{g}^{uk}\tilde{g}^{ls}\tilde{g}^{mt}\kappa_{kst}\kappa^*_{lmu}\\
				&\qquad\qquad+9\tilde{g}^{ko}\tilde{g}^{lu}\tilde{g}^{sv}\tilde{g}^{tw}\tilde{g}^{hm}\kappa^*_{lmo}\kappa^*_{sth}
				(g_{ku}g_{vw}+g_{kv}g_{uw}+g_{kw}g_{uv})\\
				&\qquad\qquad-3\tilde{g}^{kw}\tilde{g}^{ls}\tilde{g}^{mu}\tilde{g}^{tv}\kappa^*_{lmtw}
				(g_{ks}g_{uv}+g_{ku}g_{sv}+g_{kv}g_{su})\Bigr]\\
				&\quad + O(n^{-3}).
			\end{split}
		\end{equation}
	\end{coro}
	\begin{proof}
		See Section \ref{Proof_Coro} of the Appendix.
	\end{proof}
	The estimation risk up to the second-order term is determined by the moments of the $\xi_i$ terms, $g_{ij}$, and $\kappa_{ijk}$ under $g(x)$, as well as their moments under $g(x;\theta_*)$, $\tilde{g}^{ij}$, $\kappa^*_{ijk}$, and $\kappa^*_{ijkl}$.
	
	Note that the first-order term can be rewritten in different ways:
	\begin{align}
		\label{expan_est_risk_2}
		R[g(x;\theta_*)\,|\,g(x;\hat{\theta}) ] &= \frac{1}{2n} \mathrm{tr} \Bigl(\tilde{G}(\theta_*)^{-1}G(\theta_*) \Bigr)  + o(n^{-1})\\
		\label{expan_est_risk_4}
		&=\frac{p}{2n}+ \frac{1}{2n}\mathrm{tr} \Bigl(\tilde{G}(\theta_*)^{-1}\Bigl(G(\theta_*)- \tilde{G}(\theta_*)\Bigr)\Bigr)  + o(n^{-1})\\
		\label{expan_est_risk_3}
		&=\frac{p}{2n}+\frac{1}{2n} \mathrm{tr} \Bigl(\tilde{G}(\theta_*)^{-1}(S - S^*) \Bigr)
		+ o(n^{-1}),
	\end{align}
	where 
	\[
	(S)_{ij} \triangleq E[\xi_i \xi_j],\qquad (S^*)_{ij} \triangleq E_{\theta_*}[\xi_i \xi_j],\quad i, j=1,\ldots,p.
	\]
	Further, \eqref{expan_est_risk_3} can be proven as follows.
	As $E[\xi_i] = E_{\theta_*}[\xi_i]=\eta_i^*$, 
	\begin{align*}
		(G(\theta_*))_{ij}& = E[(\xi_i - \eta_i^*)(\xi_j- \eta_j^*)]\\
		&=E[\xi_i \xi_j]- E[\xi_i]E[\xi_j]\\
		&=E[\xi_i \xi_j]- E_{\theta_*}[\xi_i \xi_j]+E_{\theta_*}[\xi_i \xi_j]-E_{\theta_*}[\xi_i]E_{\theta_*}[\xi_j]\\
		&=(S)_{ij} - (S^*)_{ij} + (G^*(\theta_*))_{ij}\\
		&=(S)_{ij} - (S^*)_{ij} + (\tilde{G}(\theta_*))_{ij},\\
	\end{align*}
	which means
	\begin{equation}
		G(\theta_*) - \tilde{G}(\theta_*) = S - S^*.
	\end{equation}
	
	Because $G(\theta_*)$ and $\tilde{G}(\theta_*)$ are the variance--covariance matrices of the $\xi$ terms, respectively, under $g(x)$ and $g(x;\theta_*)$, the first-order term in \eqref{expan_est_risk_2} is interpreted as the distance between the variance--covariance matrices under $g(x)$ and $g(x;\theta_*)$.
	%We observe that
	%\begin{align*}
	%(G(\theta_*))_{ij}&= E\Bigl[\Bigl(
	%\xi_i(X) - \frac{\partial \Psi}{\partial \theta^i}\Big|_{\theta=\theta_*}\Bigr)\Bigl(
	%\xi_j(X) - \frac{\partial \Psi}{\partial \theta^j}\Big|_{\theta=\theta_*}\Bigr)\Bigl]\\
	%&=E\Bigl[\Bigl(\xi_i(X) - E[\xi_i(X)] \Bigr)\Bigl(\xi_j(X) - E[\xi_j(X)] \Bigr)\Bigr],
	%\end{align*}
	%which means $G(\theta_*)$ is the variance-covariance matrix of $\xi$'s under the true distribution $g(x)$. 
	%On the contrary, $\tilde{G}$ is the variance-covariance matrix of $\xi$'s under the distribution $g(x;\theta_*)$ as in \eqref{ddotPsi}.
	%\[
	%(\tilde{G}(\theta_*))_{ij}=\frac{\partial^2 \Psi(\theta) }{\partial \theta^i\partial \theta^j} \Big|_{\theta=\theta_*} .
	%\]
	
	From equation \eqref{expan_est_risk_4} or \eqref{expan_est_risk_3}, 
	\[
	G(\theta_*)- \tilde{G}(\theta_*)=S- S^* < \bigl(>\bigr) 0 \Longrightarrow \text{First-order terms of $R[g(x;\theta_*)\,|\,g(x;\hat{\theta}) ]$} < \bigl(>\bigr) \frac{p}{2n}.
	\]
	The first-order risk convergence speed is higher than the case where the model includes $g(x)$, if the second moment matrix of the $\xi_i$ terms under $g(x;\theta_*)$  is larger than that of $g(x)$.
	%
	%
	% Example 3
	%
	%
	%
	
	As a first example, let us consider the special case in which $\Psi(\theta)$ is a quadratic function.\\
	\\
	-- \textit{Example 3: Quadratic Exponential Model} --\\
	Let $\Psi$ be defined by
	\begin{equation}
		\label{quad_Psi}
		\Psi(\theta) = \sum_i m_i \theta^i + \frac{1}{2} \sum_{i,j} \theta^i \theta^j q_{ij}
		=m \theta^t+ \frac{1}{2}\theta Q \theta^t,
	\end{equation}
	where 
	\[
	\theta=(\theta^1,\ldots,\theta^p),\quad m=(m_1,\ldots,m_p),\quad (Q)_{ij} = q_{ij},
	\]
	and $Q$ be positive-definite. Because the higher-order cumulant vanishes for the model distribution, the estimation risk is given by
	\begin{equation}
		\label{risk_normal}
		\begin{split}
			&\frac{1}{2n} \mathrm{tr} \Bigl(\tilde{G}^{-1}(\theta_*)G(\theta_*) \Bigr)=\frac{1}{2n} \mathrm{tr} \Bigl(Q^{-1}G(\theta_*) \Bigr)\\
			&=\frac{p}{2n}+ \frac{1}{2n} \mathrm{tr} \Bigl(Q^{-1}(G(\theta_*)-Q) \Bigr).
		\end{split}
	\end{equation}
	For this type of exponential family, the $\xi_i$ terms follow the normal distribution
	\[
	\xi(X) = (\xi_1(X), \ldots, \xi_p(X)) \sim N_p( m+\theta Q, Q).
	\]
	From \eqref{risk_normal}, $Q>G(\theta_*)$ indicates faster convergence than the well-specified model (i.e., the model that contains the true distribution) case.
	
	The next example is the multinomial distribution, where the explicit form of the second-order term is given.\\
	%
	%
	% Example 4
	%
	%
	%
	\\
	-- \textit{Example 4: Multinomial Distribution Model } --\\
	Consider a multinomial distribution with $p+1$ possible values $x_i, i=0,\ldots,p$, with the corresponding probabilities $m=(m_0,\ldots,m_p)$. This is an exponential family \eqref{manifold_exp}, where
	\begin{align*}
		&\theta^i \triangleq \log (m_i/m_0), \quad i=0,\ldots, p,\\
		&\xi_i(x) \triangleq 
		\begin{cases}
			1,&\text{ if $x = x_i$,}\\
			0,&\text{ otherwise,}
		\end{cases}
		\quad i=1,\ldots,p
	\end{align*}
	and $d\mu$ is the counting measure on $\{x_1,\ldots,x_p\}$. Here,
	\[
	\Psi(\theta) = \log\bigl(\sum_{i=0}^p \exp(\theta_i)\bigr)=-\log m_0=-\log\bigl(1-\sum_{i=1}^p m_i\bigr).
	\]
	
	Suppose that $g(x)$ is continuous, and the parametric model $g(x;m)$ is an approximation of $g(x)$ with the step function 
	\[
	g(x;m) = \sum_{i=0}^p I(x \in S_i) \frac{m_i}{Vol(S_i)},
	\]
	where $S_i, i=0,1,\ldots,p$ is the partition of the range of $x$ with volume 
	\[
	Vol(S_i) \triangleq  \int_{S_i} 1 d\mu(x),
	\]
	and  $I(x \in S_i)$ is an indicator function of $S_i$. In this case, from \eqref{solution_center_3}, the information projection $g(x; m^*)$  is given by $m^*_i = P(X \in S_i | g(x))$. The step-function model is not an exponential family. However, because $\alpha$-divergence is invariant with respect to contraction by a sufficient statistic, the divergence between the two multinomial distributions (where $d\mu$ is the counting measure) equals the divergence between the corresponding step functions (where $d\mu$ is the continuous measure). Hence,  the argument of the estimation risk can be deduced from that of the multinomial distribution model. It is notable that, if $X$ is originally a discrete random variable, the model always contains $g(x)$. 
	
	The asymptotic expansion of the estimation risk up to second order can be derived as follows (this corresponds to equation (41) of \cite{Sheena2} with $\alpha=-1$, which investigates the asymptotic estimation risk for a well-specified  model).
	\begin{equation}
		\label{est_risk_discrete}
		R[g(x;\theta)\, | \, g(x;\hat\theta)]=\frac{p}{2n} + \frac{1}{12n^2} (M-1) +O(n^{-3}), \qquad M = \sum_{i=0}^p {m_i}^{-1},
	\end{equation}
	where $\theta=(m_1,\ldots,m_p)$ is the true-distribution free parameter. Because $M \geq (p+1)^2$, the second-order term is always positive. If $m_i$ is close to zero,  the convergence speed slows considerably. A numerical example of this model is given in the next section.
	\section{Criterion for Model Complexity and \& Sample Size }
	\label{Criteria}
	\subsection{$p-n$ criterion}
	\label{p-n_criteria}
	\ \ In this section, the aim is to derive a simple criterion to indicate whether the MLE is sufficiently close to the best distribution in the model (the information projection). 
	To use \eqref{expan_est_risk} or \eqref{expan_est_risk_exp}, the distributional properties in \eqref{expan_est_risk} or \eqref{expan_est_risk_exp}, which depend on unknown $\theta_*$ and/or $g(x)$, must first be estimated. Next, a certain threshold $C$ 
	with which we compare the estimated risk must be set. 
	
	If the estimated risk is not sufficiently small compared with the $C$, there are two possible remedies: increasing $n$ or reducing $p$. As the risk convergence speed depends on the geometrical properties of $g(x)$ and $\mathcal{M}$, the author conjectures that $p$ reduction does not always reduce the risk. However, as the first-order term of the risk expansion is almost $p/(2n)$ when $g(x)$ is close to $g(x;\theta_*)$, $p$ reduction is likely to reduce the risk in many cases. Based on this observation, the criterion developed in this study is named the ``$p-n$ criterion.'' 
	
	First, the risk estimation is considered. To use  \eqref{expan_est_risk}, the following properties must be estimated:
	\[
	G^*=(g^*_{ij}),\quad \tilde{G}=(\tilde{g}_{ij}), \quad G=(g_{ij}), \quad L_{ijk}, \quad L_{(ij)k}, \quad \cdots,\qquad  1\leq i,j,k,l \leq p,
	\]
	Naive estimators of these properties (denoted by the ``hat'' mark: $\hat{G}$, $\hat{L}_{ijk}$, etc.) are gained by replacing $\theta_*$ with MLE $\hat{\theta}$, and $g(x)$ with the empirical distribution. For example, 
	\begin{align}
		(\hat{G})_{ij} &\triangleq n^{-1}\sum_{t=1}^n 
		\frac{\partial \hfill}{\partial \theta^i}\log g(X_t;\theta)\Big|_{\theta=\hat{\theta}}
		\frac{\partial \hfill}{\partial \theta^j}\log g(X_t;\theta)\Big|_{\theta=\hat{\theta}}\\
		(\hat{\tilde{G}})_{ij} &\triangleq -n^{-1}\sum_{t=1}^n 
		\frac{\partial^2 \hfill}{\partial \theta^i\partial \theta^j}\log g(X_t;\theta)\Big|_{\theta=\hat{\theta}}\\
		(\hat{G}^*)_{ij} & \triangleq \int g(x;\hat\theta)\Bigl( \frac{\partial \ }{\partial \theta^i}\log g(x;\theta) \Big|_{\theta=\hat\theta}\Bigr) 
		\Bigl( \frac{\partial \ }{\partial \theta^j} \log g(x;\theta)\Big|_{\theta=\hat\theta} \Bigr) d\mu. \\
		L_{ijk} & \triangleq n^{-1}\sum_{t=1}^n 
		\frac{\partial \hfill}{\partial \theta^i}\log g(X_t;\theta)\Big|_{\theta=\hat{\theta}}
		\frac{\partial \hfill}{\partial \theta^j}\log g(X_t;\theta)\Big|_{\theta=\hat{\theta}}
		\frac{\partial \hfill}{\partial \theta^k}\log g(X_t;\theta)\Big|_{\theta=\hat{\theta}}\\
		L_{(ij)k} & \triangleq n^{-1}\sum_{t=1}^n 
		\frac{\partial^2 \hfill}{\partial \theta^i \partial \theta^j}\log g(X_t;\theta)\Big|_{\theta=\hat{\theta}}
		\frac{\partial \hfill}{\partial \theta^k}\log g(X_t;\theta)\Big|_{\theta=\hat{\theta}}
	\end{align}
	The estimated risk can be obtained using these estimators; however, because of the complication of the asymptotic risk \eqref{expan_est_risk}, the given $p-n$ criterion is difficult to handle. Here, only the criterion gained from the first-order asymptotic risk in \eqref{expan_est_risk} is stated. For $C$, the $p-n$ criterion is given as follows.
	\begin{description}
		%
		% Criteria for a general model
		%
		\item[Criterion for a general model]
		\begin{equation}
			\label{criteria_general}
			C \geq \frac{1}{2n}\mathrm{tr} \Bigl(\hat{\tilde{G}}^{-1}\hat{G}\hat{\tilde{G}}^{-1}\hat{G}^* \Bigr)
		\end{equation}
		
		For the exponential family, a simpler criterion can be derived. To use  \eqref{expan_est_risk_exp}, the following properties must be estimated:
		\[
		\tilde{G}=(\tilde{g}_{ij}), \quad G=(g_{ij}), \quad \kappa_{ijk},\quad \kappa^*_{ijk}, \quad \kappa^*_{ijkl},\qquad  1\leq i,j,k,l \leq p.
		\]
		
		$\hat{G}$ is the sample covariance matrix of the $\xi_i$ terms, $\hat{\Sigma}$:
		\begin{equation}
			\label{est_g}
			\hat{G} = \hat{\Sigma},\quad {\hat{g}}_{ij} = \bigl(\hat{\Sigma}\bigr)_{ij},\quad \bigl(\hat{\Sigma}\bigr)_{ij} \triangleq n^{-1}\sum_{t=1}^n
			(\xi_i(X_t)- \bar{\xi}_i)(\xi_j(X_t)- \bar{\xi}_j).
		\end{equation}
		Similarly, the estimator of the true third-order cumulant is given by the sample third-order cumulant: 
		\begin{equation}
			\label{est_true_kappa3}
			\hat{\kappa}_{ijk}  = n^{-1}\sum_{t=1}^n
			(\xi_i(X_t)- \bar{\xi}_i)(\xi_j(X_t)- \bar{\xi}_j)(\xi_k(X_t)- \bar{\xi}_k).
		\end{equation}
		Further, 
		\begin{align}
			\hat{\tilde{G}}& = \Ddot{\Psi}(\hat{\theta} ), \quad {\hat{\tilde{g}}}_{ij}= \bigl(\Ddot{\Psi}(\hat{\theta})\bigr)_{ij} \label{est_g_tilde}\\
			\hat{\kappa}^*_{ijk} & = \frac{\partial^3 \ }{\partial \theta^i \partial \theta^j \partial \theta^k}\Psi(\theta)\Big|_{\theta=\hat\theta} \label{est_kappa3}\\
			\hat{\kappa}^*_{ijkl} &=\frac{\partial^4 \ }{\partial \theta^i \partial \theta^j \partial \theta^k \partial \theta_l}\Psi(\theta)\Big|_{\theta=\hat\theta}. \label{est_kappa4}
		\end{align}
		%
		% Criteria for an exponential family
		%
		Consequently, for an exponential family, the $p-n$ criterion is given as follows.
		\item[Criterion for an exponential family]
		\begin{equation}
			\begin{split}
				\label{criteria_expo}
				C &\geq \frac{1}{2n} \mathrm{tr} \Bigl(\hat{\Sigma}(\Ddot{\Psi}(\hat{\theta}) )^{-1}\Bigr) +\\
				&\quad+\frac{1}{24n^2}\Bigl[-8\hat{\tilde{g}}^{uk}\hat{\tilde{g}}^{ls}\hat{\tilde{g}}^{mt}\hat{\kappa}_{kst}\hat{\kappa}^*_{lmu}\\
				&\qquad\qquad+9\hat{\tilde{g}}^{ko}\hat{\tilde{g}}^{lu}\hat{\tilde{g}}^{sv}\hat{\tilde{g}}^{tw}\hat{\tilde{g}}^{hm}\hat{\kappa}^*_{lmo}\hat{\kappa}^*_{sth}
				(\hat{g}_{ku}\hat{g}_{vw}+\hat{g}_{kv}\hat{g}_{uw}+\hat{g}_{kw}\hat{g}_{uv})\\
				&\qquad\qquad-3\hat{\tilde{g}}^{kw}\hat{\tilde{g}}^{ls}\hat{\tilde{g}}^{mu}\hat{\tilde{g}}^{tv}\hat{\kappa}^*_{lmtw}
				(\hat{g}_{ks}\hat{g}_{uv}+\hat{g}_{ku}\hat{g}_{sv}+\hat{g}_{kv}\hat{g}_{su})\Bigr].
			\end{split}
		\end{equation}
	\end{description}
	Only a naive estimator, such as the MLE or empirical moments, is used. More sophisticated estimators, such as unbiased estimators, shrinkage estimators, and bootstrap estimators, can be used, but they are outside the scope of this paper.
	%Unless the above inequality does not hold, we need to increase the sample size or decrease the model complexity, which is usually done by eliminating some $\xi_i$'s. Since the model complexity is deeply related to the model dimension $p$, we will call these criteria ``$p-n$ criteria''.
	%
	%
	% Error Rate Criteria
	%
	%
	
	Now, let us move to the second concern: selection of $C$. Another often used measure of the closeness between two distributions is the error rate, which is more intuitive than the divergence and suitable for setting a threshold. Let $g_i(x), i=1,2$ be the p.d.f. If both $g_i(x),\ i=1,2$, are known, the discrimination rule is as follows. 
	
	For the sample $X$ from either $g_1(x)$ or $g_2(x)$, 
	\[
	\frac{g_{i_1}(X)}{g_{i_2}(X)} > 1 \Longleftrightarrow \text{ Judge that $X$ is generated from $g(x;\theta_{i_1})$}
	\]
	The Bayes error rate, $Er$, i.e., the probability that this rule gives an error, is formally defined by
	\[
	Er[g_1(x)\, | \, g_2(x)] \triangleq \frac{1}{2}\int \min\Bigl(g_1(x), g_2(x)\Bigr) d\mu.
	\]
	%
	%
	% Theorem 2
	%
	%
	%
	The next theorem states the relation between $Er$ and the K-L divergence. 
	\begin{theo}
		If $D[g_1(X) \,| \, g_2(x)] \leq \delta$, then
		\begin{equation}
			\label{rel_error_D}
			Er[g(x;\theta_1)\, | \, g(x;\theta_2)] \geq \min\{ t \,|\, (x,t) \in A(\delta) \}, 
		\end{equation}
		where 
		\[
		A(\delta)\triangleq \Bigl\{ (x,t) \,\Big|\, x\log{\Bigl(\frac{1-2t}{x} + 1\Bigr)}+(1-x) \log{\Bigl(\frac{2t-1}{1-x}+1\Bigr)} = -\delta,\quad 0< x < 2t < 1\Bigr\}.
		\]
	\end{theo}
	\begin{proof}
		See \ref{dive_error} in the Appendix.
	\end{proof}
	Suppose that the standard of closeness between two distributions in view of the Bayes error rate is set to
	\begin{equation}
		\label{closeness_error}
		Er \geq 1/2- \alpha,
	\end{equation}
	where $\alpha$ is a certain number, such as  $\alpha=0.05, 0.01$. From \eqref{rel_error_D}, if 
	\begin{equation}
		\label{dev_stand_from_error}
		\min\{ t \,|\, (x,t) \in A(\delta) \} \geq 1/2- \alpha,
	\end{equation}
	standard \eqref{closeness_error} is satisfied. 
	
	Analytical calculation of $\min\{ t \,|\, (x,t) \in A(\delta) \}$ is difficult. The approximation when $t$ is close to $1/2$ is given here. As $\log{(1+x)} \doteqdot x - x^2/2$ around $x=0$,
	\begin{align*}
		& x\log{\Bigl(\frac{1-2t}{x} + 1\Bigr)}+(1-x) \log{\Bigl(\frac{2t-1}{1-x}+1\Bigr)} \\
		&= x\Bigl(\frac{1-2t}{x}\Bigr) -\frac{x}{2}\Bigl(\frac{1-2t}{x}\Bigr)^2
		+ (1-x)\frac{2t-1}{1-x}-\frac{(1-x)}{2}\Bigl(\frac{2t-1}{1-x}\Bigr)^2= -\frac{1}{2}\frac{(1-2t)^2}{x(1-x)}.
	\end{align*}
	Therefore, $A(\delta)$ is approximated by 
	\[
	A^*(\delta) \triangleq \Bigl\{ (x,t) \,\Big|\, t= \frac{1}{2}\Bigl(1-\sqrt{2\delta x(1-x)}\Bigr),\quad 0< x < 2t < 1\Bigr\}.
	\]
	Note that
	\[
	\min\{ t \,|\, (x,t) \in A^*(\delta) \} \geq \min_{0<x<1}{\frac{1}{2}\Bigl(1-\sqrt{2\delta x(1-x)}\Bigr)
	}= \frac{1}{2}-\sqrt{\delta/8},
	\]
	Hence, the condition $\sqrt{\delta/8} \leq \alpha$ or, equivalently, $\delta \leq  8\alpha^2$ is approximately sufficient for \eqref{dev_stand_from_error}. This result is stated as a corollary. 
	%
	%
	% Corollary 2
	%
	%
	\begin{coro}
		Let $\delta = D[g(x;\theta_1)\, | \, g(x;\theta_2)]$. If 
		\begin{equation}
			\label{dev_stand_from_error2}
			\min\{ t \,|\, (x,t) \in A(\delta) \} \geq 1/2- \alpha,
		\end{equation}
		then 
		\begin{equation}
			\label{closeness_error2}
			Er[g(x;\theta_1)\, | \, g(x;\theta_2)] \geq 1/2- \alpha.
		\end{equation}
		Condition \eqref{dev_stand_from_error2} is approximately equivalent to 
		\begin{equation}
		\label{closeness_error3}
			\delta  \leq 8\alpha^2.
		\end{equation}
	\end{coro}
	%We have to notice that the divergence is a geometrical measure between the two points on a statistical manifold, hence the value of the divergence between two given distributions varies according to the statistical manifold on which the distributions lie. Especially it depends on the dimension of the manifold. For example, consider the statistical manifold 
	%\[
	%\mathcal{M} = \{N(\mu(t),\sigma^2(t)) \,| \, t \in \Re \}
	%\]
	%and the two points on it, $N(\mu(t_1), \sigma^2(t_1))$ and $N(\mu(t_2), \sigma^2(t_2))$. These two distributions are also be considered as two points on the manifold 
	%\[
	%\tilde{\mathcal{M}} = \{N(\mu,\sigma^2)\, | \,\mu \in \Re, 0< \sigma^2\}
	%\]
	%The divergence between these two normal distributions differs in choosing the statistical manifold $\mathcal{M}$ or $\tilde{\mathcal{M}}$. In general, roughly speaking, the divergence increases in the proportional order to the dimension $p$. The term inside the square root in  \eqref{rel_error_D} is approximately proportional to $\delta/p$. $1/p$ can be seen as a kind of normalization of the divergence.
	
	Consequently, the $C$ in \eqref{criteria_general} or \eqref{criteria_expo} corresponding to the error rate $1/2-\alpha$ is given by the solution of $\delta$ (say, $C_\alpha$) for the equation 
	\begin{equation}
		\label{critical_eq}
		\min\{ t \,|\, (x,t) \in A(\delta) \} = 1/2- \alpha.
	\end{equation}
	More simply, $C_\alpha$ is given by 
	\begin{equation}
		\label{C_alpha}
		C_\alpha = 8 \alpha^2. 
	\end{equation}
	Thus, if $\alpha=0.01(0.05)$, then $C_\alpha=1/1250(1/50)$.
	\\
	\\
	-- \textit{Example 3 (continued)} --\\
	For the quadratic exponential model of \eqref{quad_Psi}, the right-hand side of \eqref{criteria_expo} 
	is given by 
	\[
	\frac{1}{2n}\mathrm{tr} \Bigl(Q^{-1}\hat{\Sigma} \Bigr).
	\]
	Taking an empirical approach and letting $Q$ be the sample covariance matrix $\Sigma$, the r.h.s. of  \eqref{criteria_expo} equals $p/2n$. With $C_\alpha$ in \eqref{C_alpha}, the $p-n$ criterion is approximately equivalent to
	\begin{equation}
		\label{p-n criteria quad}
		\frac{p}{n} \leq 16\alpha^2.
	\end{equation}
	Note that this criterion does not guarantee that $n$ is sufficiently large to allow $O(n^{-3})$ in \eqref{expan_est_risk_exp} to be neglected. This is a different concern, but we do not address it here.
	\\
	\\
	-- \textit{Example 4 (continued)} --
	
	For a multinomial distribution and the first-order approximation in \eqref{est_risk_discrete}, the $p-n$ criterion equals \eqref{p-n criteria quad}. The second-order approximation gives the following $p-n$ criterion:
	\begin{equation}
		\label{p-n criteria multi}
		96 n^2 \alpha^2 - 6 n p - (\hat{M}-1) > 0,
	\end{equation}
	where 
	\[
	\hat{M}=\sum_{i=0}^{p} {\hat{m}_i}^{-1}
	\]
	and $\hat{m}_i$ is the MLE, the sample relative frequency, for each $i$. Applying the criterion for $n$ determination gives the formula
	\begin{equation}
		\label{sample_size_formula_discrete}
		n \geq \frac{3p+\sqrt{9p^2+96\alpha^2(\hat{M}-1)}}{96\alpha^2}.
	\end{equation}
	In contrast, if the criterion is used for category determination, i.e., the ``bin number'' or ``bin width'' problem with regard to the histogram, the formula is given by
	\begin{equation}
		6np + \hat{M}  < 96n^2\alpha^2+1.
	\end{equation}
	Use of these criteria for practical examples is discussed in Section \ref{example_sec}.
	%As we saw before, $5M-6p-5 > p(5p+4)$. From this and the approximation $2p^2/(p-1)
	%\doteqdot 2p$, the following inequality holds;
	%\[
	%24\alpha^2 n^2 - 6n -5p - 4 > 0, 
	%\]
	%equivalently 
	%\[
	%n > \frac{1}{24\alpha^2} (3+\sqrt{9+24\alpha^2(5p+4)})
	%\]
	%Though AIC is derived under the strong assumption that the model includes the true distribution, its practicalness  is in the simplicity that the penalty term $p$ is a constant (i.e. independent of $\theta_*$) and does not need further estimation. Here we propose a model construction that results in a constant term in the first-order expansion of the risk. 
	%
	%
	%
	%
	%
	\subsection{Algorithm for $p-n$ Criterion of Exponential Family}
	\label{est_pro}
	This section describes calculation of the right-hand side of \eqref{criteria_expo}. 
	%\subsection{Two-stage Estimation Procedure}
	%\label{Two-stage}
	% The distributions of the model  \eqref{manifold_exp} are the variations of the base distribution $d\mu$. The base distribution is varied by $\exp(\sum_i \xi_i \theta^i)$, where $\theta=0$ corresponds to the base distribution. 
	If we can calculate the function $\Psi(\theta)$ analytically, the algorithm is simply the following.
	%If we explicitly give either one of the convex function $\Psi(\theta)$ or the base measure $d\mu$ , the other is simultaneously determined. We will describe the algorithm based on a given $\Psi(\theta)$.
	%(in most cases, implicitly). The algorithm for estimation differs according to which one is explicitly given.
	\begin{description}
		\item[Step 1] Calculate $\hat{\eta}_i=\bar{\xi}_i,\ i=1,\ldots,p$ from the sample.
		\item[Step 2] Solve the simultaneous equations w.r.t. $\theta$ in \eqref{hat_eta} to give $\hat{\theta}=(\hat{\theta}_1,\ldots,\hat{\theta}_p)$:
		\begin{equation}
			\label{eq_for_theta}
			\hat{\eta}_i = \eta_i(\hat\theta) = \frac{\partial \Psi}{\partial \theta_i}(\hat\theta),\qquad i=1,\ldots,p.
		\end{equation}
		\item[Step 3] Calculate \eqref{est_g_tilde}, \eqref{est_kappa3}, and \eqref{est_kappa4} from $\Psi(\hat\theta)$.
		\item[Step 4] Calculate \eqref{est_g} and \eqref{est_true_kappa3} from the sample.
		\item[Step 5] Calculate the right-hand side of \eqref{criteria_expo} and compare it with $C_\alpha$.
	\end{description}
	Often, $\Psi(\theta)$ is not explicitly given, especially for a complex model.  Then, $\hat\theta$ can be iteratively calculated using the Newton--Raphson method with the Jacobian matrix \eqref{est_g_tilde}. Because $\Ddot{\Psi}(\theta)$ is the variance-covariance matrix of the $\xi_i$ terms under the $g(x;\theta)$ distribution, its value can be approximated from the generated sample. The alternative methods are as follows.
	\begin{description}
		\item[Step 2']
		Iteratively search for $\hat\theta$ with
		\[
		\theta^{(n+1)}  = \theta^{(n)}  - \bigl(\eta(\theta^{(n)})-\hat{\eta}\bigr)\bigl(\ddot{\Psi}(\theta^{(n)})\bigr)^{-1},
		\]
		where $\eta(\theta^{(n)})$ and $\ddot{\Psi}(\theta^{(n)})$ are approximated by the sample mean and the sample covariance matrix of the $\xi_i$ terms from the  $g(x;\theta^{(n)})$ distribution. 
	\end{description}
	Further, \eqref{est_g_tilde}, \eqref{est_kappa3}, and \eqref{est_kappa4} can also be approximated using the generated sample.
	\begin{description}
		\item[Step 3']
		Approximate \eqref{est_g_tilde}, \eqref{est_kappa3}, and \eqref{est_kappa4} using the sample moments and cumulants, where the sample is  generated from $g(x;\hat{\theta})$.
	\end{description}
	
	The point here is that $\Psi(\theta)$ is not required for sample generation in Steps 2' and 3' if methods such as MCMC (requiring no normalizing constant) are used.  Although Steps 2' and 3' are computationally heavy tasks, they enable construction of a complex model without calculation of $\Psi$.

	\subsection{Example -- $p-n$ criterion application --}
	\label{example_sec}
	\ \ This section demonstrates use of the $p-n$ criterion for a particular problem through two practical examples under the exponential family model. Systematic selection of the $\xi_i$ terms is important for model construction, but it is not studied here. For further discussion of this issue, see the reference given in the Introduction.
	%
	%
	% Example 5
	%
	%
	%
	\\
	\\
	-- \textit{Example 5: Red Wine } --\\
		The first example is a well-known dataset on wine quality, taken from the U.C.I. Machine Learning Repository (https://archive.ics.uci.edu/ml/datasets/wine+quality).
	
	Only red wine data are used. The sample size is 1599,  and the variables consist of 11 chemical substances (continuous variables) and ``quality'' indexes (integers from 3 to 8). The vector of the chemical substances and the ``quality'' variable are denoted by $x^{(1)} =(x^{(1)}_{1},\ldots, x^{(1)}_{11})$ and $x^{(2)}$, respectively.  
	We divided the sample into two halves randomly, one of which (``data\_base'') was used for the model formulation and the other (``data\_est'') was used for the estimation of the parameter. 
	
	As the model formulation, we determined the following: normalization method of the original data, the reference (probability) measure $d\mu(x)$ and $\xi$ elements. Using ``data\_base'', we proceed as;
		\\
	1. Each variable $x^{(1)}_i(i=1,\ldots,11)$ is divided by twice of its maximum such that its range is $[0,\  1)$. Further, 2 is subtracted from each ``quality'' index to give a range of $\{1,2,\ldots,6\}$.\\
	2. As $d\mu(x)$, 11 independent Beta distributions are applied to $x^{(1)}$ so that their means and variances are equal to those of the ``data\_base''. The multinomial distribution of $x^{(2)}$ is adopted, using each category's sample relative frequency as the category probability parameter (say, $m_i, \ i=1,\ldots,6$). In addition, $x^{(1)}$ and $x^{(2)}$ are taken to be independent. 

Consequently, $d\mu$ is selected as
	\[
	x=(x^{(1)},x^{(2)}),\quad d\mu(x) = \prod_{i=1}^{11} {x^{(1)}_i}^{(\beta_{1i}-1)} {(1-x^{(1)}_i)}^{(\beta_{2i}-1)} d(x^{(1)}) \times \prod_{i=1}^6
	m_i^{I(x^{(2)}=i)} d^*(x^{(2)}), 
	\]
	where $d(x^{(1)})$ is the Lebesgue measure on $[0,\ 1]^{11}$, $d^*(x^{(2)})$ is the counting measure on $\{1,2,\ldots,6\}$, and $I(\cdot)$ is the indicator function. Further, $\beta_{1i}$, $\beta_{2i}$, and $m_i$ satisfy the relations
	\begin{align*}
		&\frac{\beta_{1i}}{\beta_{1i}+\beta_{2i}} = \text{ Sample mean of $x^{(1)}_i$},\quad i=1,\ldots,11\\
		&\frac{\beta_{1i}\beta_{2i}}{(\beta_{1i}+\beta_{2i})^2(\beta_{1i}+\beta_{2i}+1)} = \text{Sample variance of $x^{(1)}_i$},\quad i=1,\ldots,11\\
		&m_i = \text{Relative frequency of $i$ in $x^{(2)}$}
	\end{align*}
	3. The candidate for the $\xi_i$ terms  are  as follows: 
	\begin{align*}
		&\xi_1(x) = x^{(1)}_1 x^{(1)}_2, \quad \xi_2(x)=x^{(1)}_1 x^{(1)}_3, \quad \ldots \quad \xi_{10}(x) = x^{(1)}_1x^{(1)}_{11} \\
		&\xi_{11}(x) = x^{(1)}_2 x^{(1)}_3,\quad \ldots \quad \xi_{19}(x) = x^{(1)}_2 x^{(1)}_{11}\\
		&\hspace{20mm} \cdots \\
		&\xi_{55}(x) = x^{(1)}_{10} x^{(1)}_{11}
	\end{align*}
	and 
	\[
	\xi_{56}(x) = x^{(1)}_1x^{(2)}, \quad \ldots \quad \xi_{66}(x)= x^{(1)}_{11}x^{(2)}.
	\]
 Since some of these terms are highly correlated, we eliminate one of the pair with the correlation higher than 0.95. Actually the following 20 $\xi_i$ terms were removed from the full model: 
	\[
	\xi_i,\ i=8, 17, 19, 24, 25, 27, 32, 34, 38, 40, 43, 45, 46, 47, 49, 53, 58, 62, 64.
	\]  
	
	Consequently, an exponential family model with $p=47$ is formulated.  As the probability distribution $g(x;\theta)d\mu$ equals $d\mu$ when the $\theta$ terms all equal zero, it is denoted by $g(x; 0)$. Note that the  $g(x ;\theta_*)$ of this model is the closest to $g(x; 0)$ in the sense that
	\[
	D[g(x;\theta_*) | g(x;0)] = \min_{h \in \mathcal{H}} D[h(x) | g(x;0)],
	\]
	where $\mathcal{H}$ is the p.d.f. set of $h(x)$ (w.r.t. $d\mu$) that satisfies
	\[
	E_h[\xi_i(X)] \triangleq \int h(x) \xi_i(x) d\mu(x) = E[ \xi_i(X) ],
	\]
	for each $\xi_i$ in the model. This is the consequence of so-called ``minimum relative entropy characterization'' of an exponential family'' (see \cite{Csiszar1}). 
	
	Under the formulated exponential family model, the algorithm in the previous section was implemented and the right-hand side of \eqref{criteria_expo} was calculated using the ``data\_est'', the size of which ($n$) equals 799.  Because of the model complexity, the explicit form of $\Psi(\theta)$ could not be obtained; hence, Alternative Steps 2' and 3' were used. The R and RStan program codes are presented in GitHub (\verb|https://github.com/YSheena/P-N_Criteria_Program.git)|. The first-and second-order terms and the estimation risk in the total of \eqref{criteria_expo}  were as follows:
	\\
	First-order term: 2.95e-02, Second-order term: -1.30e-04, Estimation Risk: 2.93e-02
	
	Note that the second-order term contributes little to the estimation risk; thus, the first-order approximation seems sufficient for this model and data. Using \eqref{closeness_error3} as the equation with $\delta =$2.93e-02, we have $\alpha\doteqdot 0.06$. Hence the Bayes error rate between $g(x;\hat\theta)$ and $g(x;\theta_*)$ is higher than 0.44. If we set the threshold as $\alpha=0.05$, then we must trim the model further. For example, if we eliminate one of the $\xi$ elements from the pair with correlation higher than 0.9, then $p$ becomes as small as 37. For this model, the estimation risk is lower than the target value 0.02 as follows:
	\\
	First-order term: 1.60e-02, Second-order term: 2.04e-04, Estimation Risk: 1.62e-02

	As mentioned in the Introduction, the distribution $g(x;\theta_*)$ as the best approximation of $g(x)$ can be used for many purposes. As an example, classification of each wine into a ``quality'' class ($x^{(2)}$) based on its ``chemical substances'' ($x^{(1)}$) is briefly shown. The algorithm is quite simple: wine with $x^{(1)}$ is classified into the class where the function of $x^{(2)}$, $g((x^{(1)},x^{(2)});\hat\theta)$, attains the maximum. This approach is called a ``generative'' classifier in the field of machine learning, in contrast with a ``discriminative'' classifier like a decision tree. (Note that, in the ordinary ``generative'' model approach, after the distributions of $x^{(1)}$ are learned for each $x^{(2)}$ class, the conditional distribution of $x^{(2)}$ given $x^{(1)}$  is calculated using Bayes' theorem. In the above example, the simultaneous distribution of $(x^{(1)}, x^{(2)})$ is directly obtained. This approach is useful when the sample size in some classes is so small that it is difficult to estimate the distribution of the explanatory variables within the class.) 
	
	Cross-validation was performed ten times, with 10\% of the sample ($n=160$) being randomly chosen for the test. For each training set, the model formulation and the estimation are made as above using the whole training set in common.  The accuracy of the overall test data ($n=1600$) was 58\%. For comparison,  similar cross-validation for the naive decision tree (no bagging, no boosting) classifier was also performed, having 63\% accuracy. (The ``C50'' R package was used with the default setting.)  Although the generative model had inferior accuracy (i.e., higher uncertainty in the ``uncertain knowledge'' of Rao's equation), it could provide more reliable ``knowledge of the amount of uncertainty'' in the prediction. The accuracy difference between the training  and test data  illustrates this point.  The model accuracy for the overall training data ($n=14390$) was 56\%; hence, the accuracy difference between the training and test data was 2\% on average, whereas the accuracy of the decision tree model for the overall training data was 91\%, such that the average difference was 28\%. The generative model obtained here had no  ``overfitting'' (rather ``underfitting''); hence, we could present the accuracy safely  before applying it to new data, or even before cross-validation.\\
	%
	%
	% Example 6
	%
	%
	%
	-- \textit{Example 6: Abalone Data } --
	
	\sloppy The next example also features a well-known dataset, in this case, for the physical measurement of abalones (U.C.I. Machine Learning Repository, https://archive.ics.uci.edu/ml/datasets/Abalone). This data comprise eight properties (sex, length, diameter, etc.) of 4177 abalones. Here, only two discrete variables were considered: ``sex'' and ``ring,'' where ``sex'' had three values ``Female,'' ``Infant,'' and ``Male''; and ``rings'' had integer values from 1 to 29. The frequency of each classified group by ``sex'' and ``rings'' is  given in Table \ref{table_sex_rings}.
The original frequencies were aggregated at both ends. In the table, if a cell with a star mark is located to the immediate left or right, the number in the cell is aggregated.  For example, of the  female abalones, cells with 24 or more rings were aggregated to frequency 4.  The total number of cells was 63.
	%\begin{table}
	%\caption{Abalones by sex \& rings}
	%\label{table_sex_rings}
	%\begin{center}
	%\begin{tabular}{ |c| c|c |c |c |c | c|c |c | c|c |c | c|c |c |c|}
	%\hline
	%	   & 1 & 2 & 3 & 4 & 5 & 6 & 7 & 8 & 9 & 10 & 11 & 12 & 13 & 14 & 15  \\ \hline
	%	F & 0 & 0 & 0 & 0 & 4 & 16 & 44 & 122 & 238 & 248 & 200 & 128 & 88 & 56 & 41 \\ \hline
	%	I & 1 & 1 & 12 & 51 & 100 & 216 & 267 & 274 & 173 & 92 & 62 & 21 & 24 & 14 & 10 \\ \hline
	%	M & 0 & 0 & 3 & 6 & 11 & 27 & 80 & 172 & 278 & 294 & 225 & 118 & 91 & 56 & 52 \\ \hline
	%	   &16 & 17 & 18 & 19 & 20 & 21 & 22 & 23 & 24 & 25 & 26 & 27 & 28 & 29 & \   \\ \hline
	%	F & 30 & 26 & 19 & 15 & 12 & 7 & 3 & 6 & 1 & 1 & 0 & 1 & 0 & 1 & \  \\ \hline
	%	I & 7 & 7 & 5 & 2 & 2 & 1 & 0 & 0 & 0 & 0 & 0 & 0 & 0 & 0 & \  \\ \hline
	%	M & 30 & 25 & 18 & 15 & 12 & 6 & 3 & 3 & 1 & 0 & 1 & 1 & 0 & 0 & \  \\ \hline
	%\end{tabular}
	%\end{center}
	%\end{table}
	
	\begin{table}
		\caption{Abalones by sex \& rings}
		\label{table_sex_rings}
		\centering
		\begin{tabular}{|l|l|l|l|l|l|l|l|l|l|l|l|l|l|}
			\hline
			& 1 & 2 & 3 & 4 & 5 & 6 & 7 & 8 & 9 & 10 & 11 & 12 & 13 \\ \hline
			F & * & * & * & * & 4 & 16 & 44 & 122 & 238 & 248 & 200 & 128 & 88 \\ \hline
			I & 1 & 1 & 12 & 51 & 100 & 216 & 267 & 274 & 173 & 92 & 62 & 21 & 24 \\ \hline
			M & * & * & 3 & 6 & 11 & 27 & 80 & 172 & 278 & 294 & 225 & 118 & 91 \\ \hline
			& 14 & 15 & 16 & 17 & 18 & 19 & 20 & 21 & 22 & 23 & 24 & $25\leq$ &  \\ \hline
			F & 56 & 41 & 30 & 26 & 19 & 15 & 12 & 7 & 3 & 6 & 4 & * &  \\ \hline
			I & 14 & 10 & 7 & 7 & 5 & 2 & 2 & 1 & * & * & * & * &  \\ \hline
			M & 56 & 52 & 30 & 25 & 18 & 15 & 12 & 6 & 3 & 3 & 3 & * &  \\ \hline
		\end{tabular}
	\end{table}
	A multinomial distribution  over 63 cells was considered; hence, $p=62$. From the sample relative frequency of each cell $\hat{m}_i,$ where $\ i=0,\ldots,62$, 
	\[
	\hat{M} = \sum_{i=0}^{62} {\hat{m}_i}^{-1} = 36128.33,
	\]
	The first-order risk and estimated second-order risk in \eqref{est_risk_discrete} were respectively 
	0.0074 and 1.73e-04. Consequently, for $\alpha=0.05$, the estimation risk in total satisfied
	\[
	C_\alpha = 0.02 >  0.0074  + 1.73/10^4 \doteqdot 0.0076;
	\]
	hence, the model satisfied the $p-n$ criterion. Use of the $n$ formula \eqref{sample_size_formula_discrete} yielded
	\[
	n \geq 1642.
	\]
	However, setting $\alpha=0.01$ increased the required $n$ to 38847, with the actual $n$ being far below this value.
	\section{Total Risk Decomposition and Model Comparison}
	\label{aic_tic}
	This section treats the third problem mentioned in the Introduction, i.e., the distance between $g(x)$ and $g(x;\theta_*)$: $D[g(x) \,|\, g(x;\theta_*)]$. 
	
	For an exponential family model, the following ``generalized Pythagorean theorem'' holds:
	\begin{equation}
		\label{decomp_dive}
		D[g(x) \,|\, g(x;\hat\theta)] =D[g(x) \,|\,  g(x;\theta_*)] + D[g(x;\theta_*) \,|\, g(x;\hat\theta)]
	\end{equation}
	(see, e.g., Lemma 3 of \cite{Barron&Sheu} and Theorem 3.8 of \cite{Amari&Nagaoka}). 
	The convergence of the total distance $D[g(x) \,|\, g(x;\hat\theta)]$ as both $n$ and $p$ goes to infinity has been studied in \cite{Portnoy}, \cite{Barron&Sheu}, and \cite{Stone1}, for the exponential family model.
	
	In general, estimation of $D[g(x) \,|\, g(x;\theta_*)]$ is difficult compared to estimation of $D[g(x;\theta_*) \,|\, g(x;\hat\theta)]$ (or its expectation as treated in the previous sections), because it very subtly depends on $g(x)$. For example,  Theorem 1 of \cite{Barron&Sheu} indicates that the convergence rate in probability as $p$ goes to infinity depends on the square integrability of the higher-order derivative of $\log{g(x)}$ for an exponential family on $[0,1]$ with basis functions such as polynomial, spline, and trigonometric functions. 
	
	Taking the expectation of both sides of \eqref{decomp_dive} with respect to $g(x)$, 
	\begin{equation}
		\label{decomp_total_risk}
		R[g(x) \,|\, g(x;\hat{\theta})] = D [g(x) \,|\, g(x;\theta_*)] + R[g(x;\theta_*) \,|\, g(x;\hat\theta)],
	\end{equation}
	where the $R[ g(x)\, |\, g(x;\hat{\theta})]$ (say ``total risk'') is  defined by 
	\begin{equation}
		R[ g(x)\, |\, g(x;\hat{\theta})] \triangleq E\Bigl[D[g(x) \,|\, g(x;\hat\theta)]\Bigr]. \\
	\end{equation}
	Taking $D[g(x) \,|\, g(x;\theta_*)]$ as the ``approximation risk'' (although it is a constant and need not to be averaged  over $g(x)$), \eqref{decomp_total_risk} states that the ``total risk'' is the sum of the ``approximation  risk'' and the ``estimation risk.''
	If a model satisfies the $p-n$ criterion, the estimation risk is relatively small; hence, the total risk is mostly determined by the approximation risk.
	
	The approximation risk is decomposed as
	\begin{align*}
		D[g(x)\,|\,g(x;\theta_*) ] &= \int g(x) \log (g(x)/g(x;\theta_*) )d\mu, \\
		& = \int g(x) \log{g(x)} d\mu - \int g(x) \log{g(x;\theta_*)} d\mu.
	\end{align*}
	Because the latter part (including minus) is the cross entropy between $g(x)$ and $g(x;\theta)$, and it is determined by the model $\mathcal{M}$, let it be denoted by $Ce(M)$. Naturally, the following estimation is performed:
	\[
	\widehat{Ce(M)} \triangleq -\frac{1}{n} \sum_{t=1}^n \log{g(X_t; \hat{\theta}(X))},
	\]
	based on the sample $X=(X_1,\ldots,X_n)$ from $g(x)$. The bias of this estimator up to $n^{-1}$ order is evaluated as
	\begin{equation}
		\label{bias_Ce}
		E[\widehat{Ce(M)}] - Ce(M) = - \frac{1}{2n} \mathrm{tr} \Bigl(\tilde{G}^{-1}G \Bigr) + o(n^{-1}). 
	\end{equation} 
	(For the proof, see Section \ref{proof_bias_Ce} of the Appendix). 
	
	Using the bias-corrected estimator, the approximation risk is evaluated as 
	\[
	D[g(x)\,|\,g(x;\theta_*) ] \doteqdot \int g(x) \log{g(x)} d\mu -  \frac{1}{n} \sum_{t=1}^n \log{g(X_t; \hat{\theta}(X))} + \frac{1}{2n} \mathrm{tr} \Bigl(\tilde{G}^{-1}G \Bigr).
	\]
	Combining this with the estimation risk \eqref{expan_est_risk_exp}, the estimated total risk to $n^{-1}$-order is 
	\[
	\int g(x) \log{g(x)} d\mu -  \frac{1}{n} \sum_{t=1}^n \log{g(X_t; \hat{\theta}(X))} + \frac{1}{n} \mathrm{tr} \Bigl(\tilde{G}^{-1}G \Bigr).
	\]
	As the first term on the right-hand side is common between the models, the second and third terms, or equivalently those times $2n$, 
	\[
	-2 \sum_{t=1}^n \log{g(X_t; \hat{\theta}(X))} + 2 \mathrm{tr} \Bigl(\tilde{G}^{-1}G \Bigr),
	\]
	can be used as the criteria for the total risk comparison between the two exponential family models.  If $G$ is exchanged with its consistent estimator (e.g., $\hat{\Sigma}$, as in \eqref{est_g}), the TIC is obtained (see \cite{Takeuchi} and \cite{Konishi&Kitagawa}). Needless to say, this criterion is equivalent to the AIC for the case when the model includes $g(x)$; hence, $\tilde{G} = G$. Recall that, to calculate the first term in the TIC, $\Psi(\hat{\theta})$ must be calculated. Because it is difficult to obtain $\Psi(\theta)$ analytically in a complicated model, numerical alternatives are required.
	
	Finally, consider a comparison between the two models. Suppose that, between the two models, 
	\begin{align*}
		&\mathcal{M}_1 \triangleq \{ g_1(x;\theta) | \theta \in \Theta\}, \\
		&\mathcal{M}_2 \triangleq \{ g_2(x;\tau) | \tau \in T \}, 
	\end{align*}
	$\mathcal{M}_1$ is preferable in terms of the information criteria.  This indicates only that $g_1(x;\hat{\theta})$ is likely to be closer to $g(x)$ than $g_2(x;\hat{\tau})$. However, $g_1(x;\theta_*)$ is not guaranteed preferable to $g_2(x;\tau_*)$; hence, $\mathcal{M}_1$ is not confirmed to be a better model.  It is possible that  $g_2(x;\hat{\tau})$ is far from $g_2(x;\tau_*)$, but the approximation risk of $\mathcal{M}_2$ is smaller than that of $\mathcal{M}_1$.
	
	To compare the two models, $\mathcal{M}_1$ and $\mathcal{M}_2$, it is better to first determine whether the present $n$ is sufficiently large to satisfy the $p-n$ criterion for each model. If both models satisfy this criterion, their approximation risks can be compared based on $\widehat{Ce(M)}$ or the bias-corrected term.  (As observed above, the bias equates to the first-order term of the estimation risk; hence, $p-n$ criterion satisfaction indicates that the bias is somewhat negligible.) When $\mathcal{M}_1$ includes $\mathcal{M}_2$, the approximation risk of $\mathcal{M}_1$ is obviously smaller than that of $\mathcal{M}_2$.

\section{Appendix}
\subsection{The proof of \eqref{expan_est_risk}}
\label{Proof_Theo}
We denote an i.i.d. sample from $g(x)$ by $\xx=(X_1,\ldots, X_n)$. 

For $1\leq i \leq p$, let 
$$
\bareu{i}(\xx;\theta)\triangleq \frac{1}{n}\sum_{a=1}^n \frac{\partial }{\partial\theta^i}\log g(X_a;\theta),\qquad \bareo{i}(\xx;\theta)\triangleq \sum_{j=1}^p \tilde{g}^{ij}(\theta) \bareu{j}(\xx;\theta), 
$$
where
\[
\tilde{g}(\theta)^{ij} \triangleq (\tilde{G}(\theta)^{-1})_{ij}.
\]
Since MLE $\hat\theta$ maximizes $\log$-likelihood $\sum_{a=1}^n \log g(x_a;\theta)$
\begin{equation}
\label{bareu=0}
\bareu{i}(\xx; \hat{\theta})=0,\qquad i=1,\ldots,p
\end{equation}
Taylor exqansion of $\bareu{i}(\xx; \hat\theta)$ around $\theta_*$ is given by
\begin{align*}
\bareu{i}(\xx;\hat\theta)&=\bareu{i}(\xx;\theta_*)+\sum_{j}\frac{\partial\bareu{i}(\xx;\theta)}{\partial\theta^j}\bigg|_{\theta=\theta_{*}}  (\hat{\theta}^j-\theta_*^{j})\\
&\quad +\frac{1}{2}\sum_{j,k}\frac{\partial^2 \bareu{i}(\xx;\theta)}{\partial \theta^{j}\partial\theta^{k}}\bigg|_{\theta=\theta_*}(\hat{\theta}^{j}-\theta_*^{j})\hat{\theta}^{j}-\theta_*^{k}) \\
&\quad+ \frac{1}{3!}\sum_{j,k,l}\frac{\partial^3 \bareu{i}(\xx;\theta)}{\partial \theta^{j}\partial\theta^{k}\partial\theta^{l}}\bigg|_{\theta=\theta_{*}}(\hat{\theta}^{j}-\theta_*^{j}) (\hat{\theta}^{k}-\theta_*^{k})(\hat{\theta}^{l}-\theta_*^{l})\\
&\quad + \frac{1}{4!}\sum_{j,k,l,m}\frac{\partial^4 \bareu{i}(\xx;\theta)}{\partial \theta^{j}\partial\theta^{k}\partial\theta^{l}\partial\theta^{m}}\bigg|_{\theta=\tilde{\theta}_{ijklm}}(\hat{\theta}^{j}-\theta_*^{j}) (\hat{\theta}^{k}-\theta_*^{k})(\hat{\theta}^{l}-\theta_*^{l})(\hat{\theta}^{m}-\theta_*^{m}),
\end{align*}
where $\tilde{\theta}_{ijklm}$ is on the segment from $\theta_*$ to $\hat{\theta}$. If we add $\sum_j \tilde{g}_{ij}(\theta_*)(\hat{\theta}^j-\theta_*^{j})$ on the both sides of the above expansion and use \eqref{bareu=0}, then we have
\begin{align*}
\sum_j \tilde{g}_{ij}(\theta_*)(\hat{\theta}^j-\theta_*^{j})&=\bareu{i}(\xx;\theta_*)+\sum_{j}\biggl(\frac{\partial\bareu{i}(\xx;\theta)}{\partial\theta^j}\bigg|_{\theta=\theta_{*}}+\tilde{g}_{ij}(\theta_*)\biggr) (\hat{\theta}^j-\theta_*^j)\\
&\quad +\frac{1}{2}\sum_{j,k}\frac{\partial^2 \bareu{i}(\xx;\theta)}{\partial \theta^{j}\partial\theta^{k}}\bigg|_{\theta=\theta_*}(\hat{\theta}^{j}-\theta_*^{j})\hat{\theta}^{j}-\theta_*^{k})\\
&\quad+ \frac{1}{3!}\sum_{j,k,l}\frac{\partial^3 \bareu{i}(\xx;\theta)}{\partial \theta^{j}\partial\theta^{k}\partial\theta^{l}}\bigg|_{\theta=\theta_{*}}(\hat{\theta}^{j}-\theta_*^{j}) (\hat{\theta}^{k}-\theta_*^{k})(\hat{\theta}^{l}-\theta_*^{l})\\
&\quad + \frac{1}{4!}\sum_{j,k,l,m}\frac{\partial^4 \bareu{i}(\xx;\theta)}{\partial \theta^{j}\partial\theta^{k}\partial\theta^{l}\partial\theta^{m}}\bigg|_{\theta=\tilde{\theta}_{ijklm}}(\hat{\theta}^{j}-\theta_*^{j}) (\hat{\theta}^{k}-\theta_*^{k})(\hat{\theta}^{l}-\theta_*^{l})(\hat{\theta}^{m}-\theta_*^{m}).
\end{align*}
Furthermore if we multiply the both sides with $\tilde{g}^{is}(\theta_*)$ and sum them up over $i$ from 1 to $p$, then  we have
\begin{align}
\barteta{s}&=\bareo{s}+\sum_j \aaa{j}{s}\barteta{j}+\sum_{j,k}\bee{jk}{s}\barteta{j}\barteta{k}+\sum_{j,k}\bbar{jk}{s}\barteta{j}\barteta{k}+
\sum_{j,k,l}\cee{jkl}{s}\barteta{j}\barteta{k}\barteta{l}+\sum_{jkl}\cbar{jkl}{s}\barteta{j}\barteta{k}\barteta{l}\nonumber\\
&\qquad+\sum_{j,k,l,m}\dee{jklm}{s}\barteta{j}\barteta{k}\barteta{l}\barteta{m},
\label{expand_barteta}
\end{align}
where we used the following notations: For $1\leq j, k, l, m, s \leq p,$
\begin{align*}
\bar{\theta}^s &\triangleq \hat{\theta}^s - \theta^s_* \\
\aaa{j}{s}&\triangleq \wai \tilde{g}^{is}(\theta_*) \biggl(\frac{\partial \bareu{i}(\xx;\theta_*)}{\partial \theta^j}+\tilde{g}_{ij}(\theta_*)\biggr)\\
\bee{jk}{s}&\triangleq \frac{1}{2}\wai \tilde{g}^{is}(\theta_*)\biggl(\frac{\partial^2 \bareu{i}(\xx;\theta_*)}{\partial \theta^j \partial\theta^k}-E \biggl[\frac{\partial^2 \bareu{i}(\xx;\theta_*)}{\partial \theta^j \partial\theta^k}\biggr]\biggr)\\
\bbar{jk}{s}&\triangleq \frac{1}{2}\wai  \tilde{g}^{is}(\theta_*) E \biggl[\frac{\partial^2 \bareu{i}(\xx;\theta_*)}{\partial \theta^j \partial\theta^k}\biggr]
=\frac{1}{2}\wai \tilde{g}^{is}(\theta_*)L_{(ijk)},\\
\cee{jkl}{s}&\triangleq \frac{1}{3!}\sum_i \tilde{g}^{is}(\theta_*)\biggl(\frac{\partial^3 \bareu{i}(\xx;\theta_*)}{\partial \theta^j \partial\theta^k \partial\theta^l}-E \biggl[\frac{\partial^3 \bareu{i}(\xx;\theta_*)}{\partial \theta^j \partial\theta^k \partial\theta^l}\biggr]\biggr),\\
\cbar{jkl}{s}&\triangleq \frac{1}{3!}\wai  \tilde{g}^{is}(\theta_*)E \biggl[\frac{\partial^3 \bareu{i}(\xx;\theta_*)}{\partial \theta^j \partial\theta^k \partial\theta^l}\biggr]
=\frac{1}{3!}\wai \tilde{g}^{is}(\theta_*) L_{(ijkl)},\\
\dee{jklm}{s}&\triangleq \frac{1}{4!}\sum_i  \tilde{g}^{is}(\theta_*) \frac{\partial^4 \bareu{i}(\xx;\tilde{\theta}_{ijklm})}{\partial \theta^j \partial\theta^k \partial\theta^l \partial\theta^m}.
\end{align*}
If we insert the right-hand side of \eqref{expand_barteta} into $\barteta{j},\barteta{k},\barteta{l},\barteta{m}$ of itself, we have the following equation. 
\begin{align}
&\barteta{s}\nonumber\\
&=\bareo{s}+\sum_j \aaa{j}{s}
\biggl(
\bareo{j}+\sum_i \aaa{i}{j}\barteta{i}+\sum_{i,k}\bee{ik}{j}\barteta{i}\barteta{k}+\sum_{i,k}\bbar{ik}{j}\barteta{i}\barteta{k}+
\sum_{i,k,l}\cee{ikl}{j}\barteta{i}\barteta{k}\barteta{l}\nonumber\\
&\hspace{30mm}+\sum_{ikl}\cbar{ikl}{j}\barteta{i}\barteta{k}\barteta{l}
+\sum_{i,k,l,m}\dee{iklm}{j}\barteta{i}\barteta{k}\barteta{l}\barteta{m}
\biggr)\nonumber\\
&+\sum_{j,k}\biggl(\bee{jk}{s}+\bbar{jk}{s}\biggr)\nonumber\\
&\hspace{20mm} \times\biggl(
\bareo{j}+\sum_i \aaa{i}{j}\barteta{i}+\sum_{i,k}\bee{ik}{j}\barteta{i}\barteta{k}+\sum_{i,k}\bbar{ik}{j}\barteta{i}\barteta{k}+
\sum_{i,k,l}\cee{ikl}{j}\barteta{i}\barteta{k}\barteta{l}\nonumber\\
&\hspace{30mm}+\sum_{ikl}\cbar{ikl}{j}\barteta{i}\barteta{k}\barteta{l}
+\sum_{i,k,l,m}\dee{iklm}{j}\barteta{i}\barteta{k}\barteta{l}\barteta{m}
\biggr)\nonumber\\
&\hspace{20mm} \times\biggl(
\bareo{k}+\sum_i \aaa{i}{k}\barteta{i}+\sum_{i,j}\bee{ij}{k}\barteta{i}\barteta{j}+\sum_{i,j}\bbar{ij}{k}\barteta{i}\barteta{j}+
\sum_{i,j,l}\cee{ijl}{k}\barteta{i}\barteta{j}\barteta{l}\nonumber\\
&\hspace{30mm}+\sum_{ijl}\cbar{ijl}{k}\barteta{i}\barteta{j}\barteta{l}
+\sum_{i,j,l,m}\dee{ijlm}{k}\barteta{i}\barteta{j}\barteta{l}\barteta{m}
\biggr)\nonumber\\
&+\sum_{j,k,l}\biggl(\cee{jkl}{s}+\cbar{jkl}{s}\biggr)\nonumber\\
&\hspace{20mm} \times\biggl(
\bareo{j}+\sum_i \aaa{i}{j}\barteta{i}+\sum_{i,k}\bee{ik}{j}\barteta{i}\barteta{k}+\sum_{i,k}\bbar{ik}{j}\barteta{i}\barteta{k}+
\sum_{i,k,l}\cee{ikl}{j}\barteta{i}\barteta{k}\barteta{l}\nonumber\\
&\hspace{30mm}+\sum_{ikl}\cbar{ikl}{j}\barteta{i}\barteta{k}\barteta{l}
+\sum_{i,k,l,m}\dee{iklm}{j}\barteta{i}\barteta{k}\barteta{l}\barteta{m}
\biggr)\nonumber\\
&\hspace{20mm} \times\biggl(
\bareo{k}+\sum_i \aaa{i}{k}\barteta{i}+\sum_{i,j}\bee{ij}{k}\barteta{i}\barteta{j}+\sum_{i,j}\bbar{ij}{k}\barteta{i}\barteta{j}+
\sum_{i,j,l}\cee{ijl}{k}\barteta{i}\barteta{j}\barteta{l}\nonumber\\
&\hspace{30mm}+\sum_{i,j,l}\cbar{ijl}{k}\barteta{i}\barteta{j}\barteta{l}+\sum_{i,j,l,m}\dee{ijlm}{k}\barteta{i}\barteta{j}\barteta{l}\barteta{m}
\biggr)\nonumber\\
&\hspace{20mm} \times\biggl(
\bareo{l}+\sum_i \aaa{i}{l}\barteta{i}+\sum_{i,j}\bee{ij}{l}\barteta{i}\barteta{j}+\sum_{i,j}\bbar{ij}{l}\barteta{i}\barteta{j}+
\sum_{i,j,k}\cee{ijk}{l}\barteta{i}\barteta{j}\barteta{k}\nonumber\\
&\hspace{30mm}+\sum_{i,j,k}\cbar{ijk}{l}\barteta{i}\barteta{j}\barteta{k}+\sum_{i,j,k,m}\dee{ijkm}{l}\barteta{i}\barteta{j}\barteta{k}\barteta{m}
\biggr)\nonumber\\
&+\sum_{j,k,l,m}\dee{jklm}{s}\nonumber\\
&\hspace{20mm} \times\biggl(
\bareo{j}+\sum_i \aaa{i}{j}\barteta{i}+\sum_{i,k}\bee{ik}{j}\barteta{i}\barteta{k}+\sum_{i,k}\bbar{ik}{j}\barteta{i}\barteta{k}+
\sum_{i,k,l}\cee{ikl}{j}\barteta{i}\barteta{k}\barteta{l}\nonumber\\
&\hspace{30mm}+\sum_{ikl}\cbar{ikl}{j}\barteta{i}\barteta{k}\barteta{l}
+\sum_{i,k,l,m}\dee{iklm}{j}\barteta{i}\barteta{k}\barteta{l}\barteta{m}
\biggr)\nonumber\\
&\hspace{20mm} \times\biggl(
\bareo{k}+\sum_i \aaa{i}{k}\barteta{i}+\sum_{i,j}\bee{ij}{k}\barteta{i}\barteta{j}+\sum_{i,j}\bbar{ij}{k}\barteta{i}\barteta{j}+
\sum_{i,j,l}\cee{ijl}{k}\barteta{i}\barteta{j}\barteta{l}\nonumber\\
&\hspace{30mm}+\sum_{i,j,l}\cbar{ijl}{k}\barteta{i}\barteta{j}\barteta{l}+\sum_{i,j,l,m}\dee{ijlm}{k}\barteta{i}\barteta{j}\barteta{l}\barteta{m}
\biggr)\nonumber\\
&\hspace{20mm} \times\biggl(
\bareo{l}+\sum_i \aaa{i}{l}\barteta{i}+\sum_{i,j}\bee{ij}{l}\barteta{i}\barteta{j}+\sum_{i,j}\bbar{ij}{l}\barteta{i}\barteta{j}+
\sum_{i,j,k}\cee{ijk}{l}\barteta{i}\barteta{j}\barteta{k}\nonumber\\
&\hspace{30mm}+\sum_{i,j,k}\cbar{ijk}{l}\barteta{i}\barteta{j}\barteta{k}+\sum_{i,j,k,m}\dee{ijkm}{l}\barteta{i}\barteta{j}\barteta{k}\barteta{m}
\biggr)\nonumber\\
&\hspace{20mm} \times\biggl(
\bareo{m}+\sum_i \aaa{i}{m}\barteta{i}+\sum_{i,j}\bee{ij}{m}\barteta{i}\barteta{j}+\sum_{i,j}\bbar{ij}{m}\barteta{i}\barteta{j}+
\sum_{i,j,k}\cee{ijk}{m}\barteta{i}\barteta{j}\barteta{k}\nonumber\\
&\hspace{30mm}+\sum_{i,j,k}\cbar{ijk}{m}\barteta{i}\barteta{j}\barteta{k}+\sum_{i,j,k,l}\dee{ijkl}{m}\barteta{i}\barteta{j}\barteta{k}\barteta{l}
\biggr) \label{expand_barteta_2}
\end{align}
Expanding the equation, counting the order of each term, we can rewrite \eqref{expand_barteta_2} as
\begin{align*}
\barteta{s}&=\bareo{s}+\sum_j \aaa{j}{s}
\biggl(
\bareo{j}+\sum_i \aaa{i}{j}\barteta{i}+\sum_{i,k}\bbar{ik}{j}\barteta{i}\barteta{k}
\biggr)\\
&\quad+\sum_{j,k}\biggl(\bee{jk}{s}+\bbar{jk}{s}\biggr)\\
&\hspace{20mm} \times\biggl(
\bareo{j}+\sum_i \aaa{i}{j}\barteta{i}+\sum_{i,k}\bbar{ik}{j}\barteta{i}\barteta{k}
\biggr)\\
&\hspace{20mm} \times\biggl(
\bareo{k}+\sum_i \aaa{i}{k}\barteta{i}+\sum_{i,j}\bbar{ij}{k}\barteta{i}\barteta{j}
\biggr)\\
&\quad+\sum_{j,k,l}\biggl(\cee{jkl}{s}+\cbar{jkl}{s}\biggr)\bareo{j}\bareo{k}\bareo{l}+Re1,
\end{align*}
where $Re1$  is the polynomial  with respect to the variables $\barteta{s}$, $\bareo{s}$, $\aaa{j}{s}$, $\bee{jk}{s}$, $\cee{jkl}{s}$, $\dee{jklm}{s}$ $(1\leq j, k, l, m, s \leq p)$, and each term is of at least fourth order with respect to $\barteta{s}$, $\bareo{s}$, $\aaa{j}{s}$, $\bee{jk}{s}$, $\cee{jkl}{s}$ $(1\leq j, k, l, s \leq p)$.
If we insert this result into the right-hand side of itself, then we  yield the result.
\begin{align}
\barteta{s}&=\bareo{s}+\sum_j \aaa{j}{s}
\bareo{j}+\sum_{i,j}\aaa{j}{s}\aaa{i}{j}\bareo{i}+\sum_{i,j,k}\aaa{j}{s}\bbar{ik}{j}\bareo{i}\bareo{k}\nonumber\\
&+\sum_{j,k}\biggl(\bee{jk}{s}+\bbar{jk}{s}\biggr)\nonumber\\
&\hspace{20mm} \times\biggl(
\bareo{j}+\sum_i \aaa{i}{j}\bareo{i}+\sum_{i,l}\bbar{il}{j}\bareo{i}\bareo{l}
\biggr)\nonumber\\
&\hspace{20mm} \times\biggl(
\bareo{k}+\sum_i \aaa{i}{k}\bareo{i}+\sum_{i,l}\bbar{il}{k}\bareo{i}\bareo{l}
\biggr)\nonumber\\
&+\sum_{j,k,l}\biggl(\cee{jkl}{s}+\cbar{jkl}{s}\biggr)\bareo{j}\bareo{k}\bareo{l}+Re2\nonumber\\
&=\bareo{s}+\sum_{j}\aaa{j}{s}\bareo{j}+\sum_{j,k}\bbar{jk}{s}\bareo{j}\bareo{k}+\sum_{i,j}\aaa{j}{s}\aaa{i}{j}\bareo{i}+\sum_{i,j,k}\aaa{j}{s}\bbar{ik}{j}\bareo{i}\bareo{k}\nonumber\\
&\ +\sum_{j,k}\bee{jk}{s}\bareo{j}\bareo{k}+2\sum_{i,j,k}\bbar{jk}{s}\aaa{i}{k}\bareo{i}\bareo{j}+2\sum_{i,j,k,l}\bbar{jk}{s}\bbar{il}{k}\bareo{i}\bareo{j}\bareo{l}\nonumber\\
&\ +\sum_{j,k,l}\cbar{jkl}{s}\bareo{j}\bareo{k}\bareo{l}+Re3, \label{expan_barteta}
\end{align}
where $Re2$ and $Re3$ have the same property as $Re1$.

Here we impose the moment conditions as follows. The suitably higher-order joint moments composed of the following variables are bounded with respect to $n$;
\begin{equation}
\label{comp_of_residuals}
\sqrt{n}\,\barteta{s},\quad\sqrt{n}\,\bareo{s},\quad\sqrt{n}\,\aaa{j}{s},\quad \sqrt{n}\,\bee{jk}{s},\quad \sqrt{n}\,\cee{jkl}{s},\quad \dee{jklm}{s},
\end{equation}
where $1\leq j,  k,  l,  m,  s \leq p.$ Then the following result on the expectations
hold. In the process of the calculation, we use Einstein summation notation for brevity. \\
%
%  $E[\barteta{i}\barteta{j}]$.
%
First we consider $E[\barteta{i}\barteta{j}]$. From \eqref{expan_barteta}, we have
\begin{equation}
\label{barteta^2}
\begin{split}
&\barteta{i}\barteta{j}\\
&=\Bigl(\bareo{i}+\aaa{l}{i}\bareo{l}+\bbar{lm}{i}\bareo{l}\bareo{m}+\aaa{l}{i}\aaa{m}{l}\bareo{m}+\aaa{l}{i}\bbar{ms}{l}\bareo{m}\bareo{s}+\bee{lm}{i}\bareo{l}\bareo{m}+2\bbar{lm}{i}\aaa{s}{m}\bareo{l}\bareo{s}\\
&\qquad+2\bbar{lm}{i}\bbar{st}{m}\bareo{l}\bareo{s}\bareo{t}+\cbar{lmt}{i}\bareo{l}\bareo{m}\bareo{t}+Re3\Bigr)\\
&\times\Bigl(\bareo{j}+\aaa{l}{j}\bareo{l}+\bbar{lm}{j}\bareo{l}\bareo{m}+\aaa{l}{j}\aaa{m}{l}\bareo{m}+\aaa{l}{j}\bbar{ms}{l}\bareo{m}\bareo{s}+\bee{lm}{j}\bareo{l}\bareo{m}+2\bbar{lm}{j}\aaa{s}{m}\bareo{l}\bareo{s}\\
&\qquad+2\bbar{lm}{j}\bbar{st}{m}\bareo{l}\bareo{s}\bareo{t}+\cbar{lmt}{j}\bareo{l}\bareo{m}\bareo{t}+Re3\Bigr)\\
&=\bareo{i}\bareo{j}+\aaa{l}{j}\bareo{i}\bareo{l}+\aaa{l}{i}\bareo{j}\bareo{l}+\bbar{lm}{j}\bareo{i}\bareo{l}\bareo{m}+\bbar{lm}{i}\bareo{j}\bareo{l}\bareo{m}\\
&\quad+\aaa{l}{j}\aaa{m}{l}\bareo{i}\bareo{m}+\aaa{l}{j}\bbar{ms}{l}\bareo{i}\bareo{m}\bareo{s}+\bee{lm}{j}\bareo{i}\bareo{l}\bareo{m}\\
&\quad+2\bbar{lm}{j}\aaa{s}{m}\bareo{i}\bareo{l}\bareo{s}+2\bbar{lm}{j}\bbar{st}{m}\bareo{i}\bareo{l}\bareo{s}\bareo{t}+\cbar{lmt}{j}\bareo{i}\bareo{l}\bareo{m}\bareo{t}\\
&\quad+\aaa{l}{i}\aaa{m}{l}\bareo{j}\bareo{m}+\aaa{l}{i}\bbar{ms}{l}\bareo{j}\bareo{m}\bareo{s}+\bee{lm}{i}\bareo{j}\bareo{l}\bareo{m}\\
&\quad+2\bbar{lm}{i}\aaa{s}{m}\bareo{j}\bareo{l}\bareo{s}+2\bbar{lm}{i}\bbar{st}{m}\bareo{j}\bareo{l}\bareo{s}\bareo{t}+\cbar{lmt}{i}\bareo{j}\bareo{l}\bareo{m}\bareo{t}\\
&\quad+\aaa{l}{i}\aaa{m}{j}\bareo{l}\bareo{m}+\aaa{l}{i}\bbar{st}{j}\bareo{l}\bareo{s}\bareo{t}+\aaa{l}{j}\bbar{st}{i}\bareo{l}\bareo{s}\bareo{t}+\bbar{lm}{i}\bbar{st}{j}\bareo{l}\bareo{m}\bareo{s}\bareo{t}+Re4,
\end{split}
\end{equation}
where $Re4$ is  a polynomial  with respect to the variables $\barteta{s}$, $\bareo{s}$, $\aaa{j}{s}$, $\bee{jk}{s}$, $\cee{jkl}{s}$, $\dee{jklm}{s}$ $(1\leq j, k, l, m, s \leq p)$, and each term is of at least fifth order with respect to $\barteta{s}$, $\bareo{s}$, $\aaa{j}{s}$, $\bee{jk}{s}$, $\cee{jkl}{s}$ $(1\leq j, k, l, s \leq p)$.
We calculate the expectation of each term on the right-hand side of this equation. 
$\tilde{g}^{ij}(\theta_*),g_{ij}(\theta_*) (1\leq i,j \leq p)$ are abbreviated as $\tilde{g}^{ij}, g_{ij}$. Note that
\begin{equation}
\label{E_e_i_zero}
E\biggl[\frac{\partial\hfill }{\partial \theta^i}\log f(X_t;\theta_*)\biggr]=0,\qquad i=1,\ldots,p, \quad t=1,\ldots,n.
\end{equation}
%
% The calculation of each term in E[\barteta{i}\barteta{j}]
%
\begin{align}
&E[\bareo{i}\bareo{j}]\nonumber\\
&=n^{-2}E\biggl[\biggl(\sum_{a=1}^n \tilde{g}^{il}\frac{\partial\hfill}{\partial \theta^l}\log f(X_a;\theta_*)\biggr)\biggl(\sum_{b=1}^n \tilde{g}^{jm}\frac{\partial\hfill}{\partial \theta^m}\log f(X_b;\theta_*)\biggr)
\biggr]\nonumber\\
&=n^{-1}E\biggl[\tilde{g}^{il}\tilde{g}^{jm}\frac{\partial\hfill }{\partial \theta^l}\log f(X;\theta_*)\frac{\partial\hfill }{\partial \theta^m}\log f(X;\theta_*)\biggr]\nonumber\\
&\quad+n^{-2}\sum_{a\ne b}\tilde{g}^{il}E\biggl[\frac{\partial\hfill }{\partial \theta^l}\log f(X_a;\theta_*)\biggr]\tilde{g}^{jm}E\biggl[ \frac{\partial\hfill}{\partial \theta^m}\log f(X_b;\theta_*)\biggr]\nonumber\\
&=n^{-1}\tilde{g}^{il}\tilde{g}^{jm}g_{lm}.\label{bareo_bareo}
\end{align}
\begin{equation}
\label{aaa_bareo_bareo}
\begin{split}
&E[\aaa{l}{j}\bareo{i}\bareo{l}]\\
&=E\biggl[
n^{-1}\tilde{g}^{sj} \sum_{c=1}^n\biggl(\frac{\partial^2\hfill}{\partial \theta^s\partial \theta^l}\log f(X_c;\theta_*)+\tilde{g}_{ls}\biggr)\\
&\qquad\times n^{-2}\tilde{g}^{it}\biggl(\sum_{a=1}^n \frac{\partial\hfill}{\partial\theta^t}\log f(X_a;\theta_*)\biggr)\tilde{g}^{lm}\biggl(\sum_{b=1}^n \frac{\partial\hfill}{\partial\theta^m}\log f(X_b;\theta_*)\biggr)
\biggr]\\
&=n^{-3}\tilde{g}^{sj}\tilde{g}^{it}\tilde{g}^{lm}\\
&\quad\times\biggl\{ n E\biggl[\biggl(\frac{\partial^2\hfill}{\partial \theta^s \partial \theta^l} \log f(X;\theta_*)\biggr)\biggl(\frac{\partial\hfill}{\partial\theta^t}\log f(X;\theta_*)\biggr)\biggl( \frac{\partial\hfill}{\partial\theta^m}\log f(X;\theta_*) \biggr)\biggr]\\
&\qquad\quad+n\tilde{g}_{ls}E\biggl[\biggl(\frac{\partial\hfill}{\partial\theta^t}\log f(X;\theta_*)\biggr)\biggl( \frac{\partial\hfill}{\partial\theta^m}\log f(X;\theta_*) \biggr)\biggr]\\
&\qquad\quad+ \sum_{a\ne c}E\biggl[
\frac{\partial^2\hfill}{\partial \theta^s\partial \theta^l}\log f(X_c;\theta_*)+\tilde{g}_{ls}
\biggr]\\
&\qquad\qquad\times E\biggl[\biggl(\frac{\partial\hfill}{\partial\theta^t}\log f(X_a;\theta_*)\biggr)\biggl( \frac{\partial\hfill}{\partial\theta^m}\log f(X_a;\theta_*) \biggr)\biggr]\\
&\qquad\quad+\sum_{a\ne b} E\biggl[\frac{\partial\hfill}{\partial\theta^m}\log f(X_b;\theta_*)\biggr]\\
&\qquad\qquad\times E\biggl[ \biggl(\frac{\partial^2\hfill}{\partial \theta^s\partial \theta^l}\log f(X_a;\theta_*)+\tilde{g}_{ls}\biggr)\biggl(\frac{\partial\hfill}{\partial\theta^t}\log f(X_a;\theta_*)\biggr)\biggr]\\
&\qquad\quad+\sum_{b\ne a}E\biggl[\frac{\partial\hfill}{\partial\theta^t}\log f(X_a;\theta_*)\biggr]\\
&\qquad\qquad\times E\biggl[ \biggl(\frac{\partial^2\hfill}{\partial \theta^s\partial \theta^l}\log f(X_b;\theta_*)+\tilde{g}_{ls}\biggr)\biggl(\frac{\partial\hfill}{\partial\theta^m}\log f(X_b;\theta_*)\biggr)\biggr]\\\
&\qquad\quad+\sum_{a\ne b, a\ne c, b\ne c}E\biggl[\frac{\partial^2\hfill}{\partial \theta^s\partial \theta^l}\log f(X_c;\theta_*)+\tilde{g}_{ls}\biggr]\\
&\qquad\qquad\times E\biggl[\frac{\partial\hfill}{\partial\theta^t}\log f(X_a;\theta_*)\biggr]E\biggl[\frac{\partial\hfill}{\partial\theta^m}\log f(X_b;\theta_*)\biggr]
\biggr\}\\
&=n^{-2}\tilde{g}^{sj}\tilde{g}^{it}\tilde{g}^{lm}\\
&\quad\times\biggl\{E\biggl[\biggl(\frac{\partial^2\hfill}{\partial \theta^s \partial \theta^l} \log f(X;\theta_*)\biggr)\biggl(\frac{\partial\hfill}{\partial\theta^t}\log f(X;\theta_*)\biggr)\biggl( \frac{\partial\hfill}{\partial\theta^m}\log f(X;\theta_*) \biggr)\biggr]+\tilde{g}_{ls}g_{tm}\biggr\}\\
&=n^{-2}\tilde{g}^{sj}\tilde{g}^{it}\tilde{g}^{lm}(L_{(sl)tm}+\tilde{g}_{ls}g_{tm}).
\end{split}
\end{equation}
\begin{equation}
\label{bbar_bareo^3}
\begin{split}
&E[\bbar{lm}{j}\bareo{i}\bareo{l}\bareo{m}]\\
&=\bbar{lm}{j}E[\bareo{i}\bareo{l}\bareo{m}]\\
&=n^{-3}\bbar{lm}{j}\tilde{g}^{ik}\tilde{g}^{ls}\tilde{g}^{mt}E\biggl[
\biggl(\sum_{a=1}^n\frac{\partial\hfill}{\partial\theta^k}\log f(X_a;\theta_*)\biggr)
\biggl(\sum_{b=1}^n\frac{\partial\hfill}{\partial\theta^s}\log f(X_b;\theta_*)\biggr)
\biggl(\sum_{c=1}^n\frac{\partial\hfill}{\partial\theta^t}\log f(X_c;\theta_*)\biggr)
\biggr]\\
&=n^{-3}\bbar{lm}{j}\tilde{g}^{ik}\tilde{g}^{ls}\tilde{g}^{mt}\\
&\quad\times\biggl\{nE\biggl[\biggl(\frac{\partial\hfill}{\partial \theta^k}\log f(X;\theta_*)\biggr)\biggl(\frac{\partial\hfill}{\partial \theta^s}\log f(X;\theta_*)\biggr)\biggl(\frac{\partial\hfill}{\partial \theta^t}\log f(X;\theta_*)\biggr)\biggr]\\
&\quad\qquad+\sum_{a\ne c}E\biggl[\biggl( \frac{\partial\hfill}{\partial \theta^k}\log f(X_a;\theta_*)\biggr)\biggl(\frac{\partial\hfill}{\partial \theta^s}\log f(X_a;\theta_*)\biggr)\biggr]
E\biggl[\frac{\partial\hfill}{\partial \theta^t}\log f(X_c;\theta_*)\biggr]\\
&\quad\qquad+\sum_{a\ne c}E\biggl[\biggl( \frac{\partial\hfill}{\partial \theta^k}\log f(X_a;\theta_*)\biggr)\biggl(\frac{\partial\hfill}{\partial \theta^t}\log f(X_a;\theta_*)\biggr)\biggr]
E\biggl[\frac{\partial\hfill}{\partial \theta^s}\log f(X_c;\theta_*)\biggr]\\
&\quad\qquad+\sum_{a\ne c}E\biggl[\biggl( \frac{\partial\hfill}{\partial \theta^s}\log f(X_a;\theta_*)\biggr)\biggl(\frac{\partial\hfill}{\partial \theta^t}\log f(X_a;\theta_*)\biggr)\biggr]
E\biggl[\frac{\partial\hfill}{\partial \theta^k}\log f(X_c;\theta_*)\biggr]\\
&\quad\qquad+\sum_{a\ne b, a\ne c, b\ne c} E\biggl[\frac{\partial\hfill}{\partial \theta^k}\log f(X_a;\theta_*)\biggr]E\biggl[\frac{\partial\hfill}{\partial \theta^s}\log f(X_b;\theta_*)\biggr]E\biggl[\frac{\partial\hfill}{\partial \theta^t}\log f(X_c;\theta_*)\biggr]
\biggr\}\\
&=n^{-2}\bbar{lm}{j}\tilde{g}^{ik}\tilde{g}^{ls}\tilde{g}^{mt}\\
&\quad\times E\biggl[\biggl(\frac{\partial\hfill}{\partial \theta^k}\log f(X;\theta_*)\biggr)\biggl(\frac{\partial\hfill}{\partial \theta^s}\log f(X;\theta_*)\biggr)\biggl(\frac{\partial\hfill}{\partial \theta^t}\log f(X;\theta_*)\biggr)\biggr]\\
&=n^{-2}\bbar{lm}{j}\tilde{g}^{ik}\tilde{g}^{ls}\tilde{g}^{mt}L_{kst}.
\end{split}
\end{equation}
\begin{align}
&E[\aaa{l}{j}\aaa{m}{l}\bareo{i}\bareo{m}]\nonumber\\
&=n^{-4}E\biggl[\biggl\{\tilde{g}^{jk}\sum_{a=1}^n \biggl(\frac{\partial^2\hfill}{\partial \theta^k\partial \theta^l}\log f(X_a;\theta_*)+\tilde{g}_{kl}\biggr)\biggr\}\nonumber\\
&\qquad\qquad\times\biggl\{\tilde{g}^{lu} \sum_{b=1}^n \biggl(\frac{\partial^2\hfill}{\partial \theta^u\partial \theta^m}\log f(X_b;\theta_*)+\tilde{g}_{um}\biggr)\biggr\}\nonumber\\
&\qquad\qquad\times\biggl\{\tilde{g}^{is}\sum_{c=1}^n\frac{\partial\hfill}{\partial \theta^s}\log f(X_c;\theta_*)\biggr\}\biggl\{\tilde{g}^{mt}\sum_{d=1}^n\frac{\partial\hfill}{\partial \theta^t}\log f(X_d;\theta_*)\biggr\}\biggr]\nonumber\\
&=n^{-4}\tilde{g}^{jk}\tilde{g}^{lu}\tilde{g}^{is}\tilde{g}^{mt}\nonumber\\
&\quad\times\biggl\{\sum_{a=1}^nE\biggl[ \biggl(\frac{\partial^2\hfill}{\partial \theta^k\partial \theta^l}\log f(X_a;\theta_*)+\tilde{g}_{kl}\biggr)\biggl(\frac{\partial^2\hfill}{\partial \theta^u\partial \theta^m}\log f(X_a;\theta_*)+\tilde{g}_{um}\biggr)\nonumber\\
&\qquad\qquad\qquad\times\biggl(\frac{\partial\hfill}{\partial \theta^s}\log f(X_a;\theta_*)\biggr)\biggl(\frac{\partial\hfill}{\partial \theta^t}\log f(X_a;\theta_*)\biggr)\biggr]\nonumber\\
&\qquad\quad+\sum_{a\ne b} E\biggl[\frac{\partial^2\hfill}{\partial \theta^k\partial \theta^l}\log f(X_a;\theta_*)+\tilde{g}_{kl}\biggr]\nonumber\\
&\qquad\qquad\times
E\biggl[\biggl(\frac{\partial^2\hfill}{\partial \theta^u\partial \theta^m}\log f(X_b;\theta_*)+\tilde{g}_{um}\biggr)\biggl(\frac{\partial\hfill}{\partial \theta^s}\log f(X_b;\theta_*)\biggr)\biggl(\frac{\partial\hfill}{\partial \theta^t}\log f(X_b;\theta_*)\biggr)
\biggr]\nonumber\\
&\qquad\quad+\sum_{a\ne b} E\biggl[\frac{\partial^2\hfill}{\partial \theta^u\partial \theta^m}\log f(X_a;\theta_*)+\tilde{g}_{um}\biggr]\nonumber\\
&\qquad\qquad\times
E\biggl[\biggl(\frac{\partial^2\hfill}{\partial \theta^k\partial \theta^l}\log f(X_b;\theta_*)+\tilde{g}_{kl}\biggr)\biggl(\frac{\partial\hfill}{\partial \theta^s}\log f(X_b;\theta_*)\biggr)\biggl(\frac{\partial\hfill}{\partial \theta^t}\log f(X_b;\theta_*)\biggr)
\biggr]\nonumber\\
&\qquad\quad+ \sum_{a\ne c}E\biggl[\biggl(\frac{\partial^2\hfill}{\partial \theta^k\partial \theta^l}\log f(X_a;\theta_*)+\tilde{g}_{kl}\biggr)\biggl(\frac{\partial^2\hfill}{\partial \theta^u\partial \theta^m}\log f(X_a;\theta_*)+\tilde{g}_{um}\biggr)\nonumber\\
&\qquad\qquad\qquad\times\biggl(\frac{\partial\hfill}{\partial \theta^t}\log f(X_a;\theta_*)\biggr)\biggr] E\biggl[\frac{\partial\hfill}{\partial \theta^s}\log f(X_c;\theta_*)\biggr]\nonumber\\
&\qquad\quad+ \sum_{a\ne c}E\biggl[\biggl(\frac{\partial^2\hfill}{\partial \theta^k\partial \theta^l}\log f(X_a;\theta_*)+\tilde{g}_{kl}\biggr)\biggl(\frac{\partial^2\hfill}{\partial \theta^u\partial \theta^m}\log f(X_a;\theta_*)+\tilde{g}_{um}\biggr)\nonumber\\
&\qquad\qquad\qquad\times\biggl(\frac{\partial\hfill}{\partial \theta^s}\log f(X_a;\theta_*)\biggr)\biggr] E\biggl[\frac{\partial\hfill}{\partial \theta^t}\log f(X_c;\theta_*)\biggr]\nonumber\\
&\qquad\quad+\sum_{a\ne c} E\biggl[\biggl(\frac{\partial^2\hfill}{\partial \theta^k\partial \theta^l}\log f(X_a;\theta_*)+\tilde{g}_{kl}\biggr)\biggl(\frac{\partial^2\hfill}{\partial \theta^u\partial \theta^m}\log f(X_a;\theta_*)+\tilde{g}_{um}\biggr)\biggr]\nonumber\\
&\qquad\qquad\times E\biggl[\biggl(\frac{\partial\hfill}{\partial \theta^s}\log f(X_c;\theta_*)\biggr)\biggl(\frac{\partial\hfill}{\partial \theta^t}\log f(X_c;\theta_*)\biggr)\biggr]\nonumber\\
&\qquad\quad+\sum_{a\ne c} E\biggl[\biggl(\frac{\partial^2\hfill}{\partial \theta^k\partial \theta^l}\log f(X_a;\theta_*)+\tilde{g}_{kl}\biggr)\biggl(\frac{\partial\hfill}{\partial \theta^s}\log f(X_a;\theta_*)\biggr)\biggr]\nonumber\\
&\qquad\qquad\times E\biggl[\biggl(\frac{\partial^2\hfill}{\partial \theta^u\partial \theta^m}\log f(X_c;\theta_*)+\tilde{g}_{um}\biggr)\biggl(\frac{\partial\hfill}{\partial \theta^t}\log f(X_c;\theta_*)\biggr)\biggr]\nonumber\\
&\qquad\quad+\sum_{a\ne c} E\biggl[\biggl(\frac{\partial^2\hfill}{\partial \theta^k\partial \theta^l}\log f(X_a;\theta_*)+\tilde{g}_{kl}\biggr)\biggl(\frac{\partial\hfill}{\partial \theta^t}\log f(X_a;\theta_*)\biggr)\biggr]\nonumber\\
&\qquad\qquad\times E\biggl[\biggl(\frac{\partial^2\hfill}{\partial \theta^u\partial \theta^m}\log f(X_c;\theta_*)+\tilde{g}_{um}\biggr)\biggl(\frac{\partial\hfill}{\partial \theta^s}\log f(X_c;\theta_*)\biggr)\biggr]\nonumber\\
&\qquad\quad+\sum_{a\ne b, a\ne c, b\ne c}E\biggl[\biggl(\frac{\partial^2\hfill}{\partial \theta^k\partial \theta^l}\log f(X_a;\theta_*)+\tilde{g}_{kl}\biggr)\biggl(\frac{\partial^2\hfill}{\partial \theta^u\partial \theta^m}\log f(X_a;\theta_*)+\tilde{g}_{um}\biggr)
\biggr]\nonumber\\
&\qquad\qquad\times E\biggl[\biggl(\frac{\partial\hfill}{\partial \theta^s}\log f(X_b;\theta_*)\biggr)\biggr]E\biggl[\biggl(\frac{\partial\hfill}{\partial \theta^t}\log f(X_c;\theta_*)\biggr)\biggr]\nonumber\\
&\qquad\quad+\sum_{a\ne b, a\ne c, b\ne c}E\biggl[\biggl(\frac{\partial^2\hfill}{\partial \theta^k\partial \theta^l}\log f(X_a;\theta_*)+\tilde{g}_{kl}\biggr)
\biggl(\frac{\partial\hfill}{\partial \theta^s}\log f(X_a;\theta_*)\biggr)
\biggr]\nonumber\\
&\qquad\qquad\times E\biggl[\biggl(\frac{\partial^2\hfill}{\partial \theta^u\partial \theta^m}\log f(X_b\theta)+\tilde{g}_{um}\biggr)\biggr]E\biggl[\biggl(\frac{\partial\hfill}{\partial \theta^t}\log f(X_c;\theta_*)\biggr)\biggr]\nonumber\\
&\qquad\quad+\sum_{a\ne b, a\ne c, b\ne c}E\biggl[\biggl(\frac{\partial^2\hfill}{\partial \theta^k\partial \theta^l}\log f(X_a;\theta_*)+\tilde{g}_{kl}\biggr)
\biggl(\frac{\partial\hfill}{\partial \theta^t}\log f(X_a;\theta_*)\biggr)
\biggr]\nonumber\\
&\qquad\qquad\times E\biggl[\biggl(\frac{\partial^2\hfill}{\partial \theta^u\partial \theta^m}\log f(X_b\theta)+\tilde{g}_{um}\biggr)\biggr]E\biggl[\biggl(\frac{\partial\hfill}{\partial \theta^s}\log f(X_c;\theta_*)\biggr)\biggr]\nonumber\\
&\qquad\quad+\sum_{a\ne b, a\ne c, b\ne c}E\biggl[\biggl(\frac{\partial^2\hfill}{\partial \theta^u\partial \theta^m}\log f(X_a;\theta_*)+\tilde{g}_{um}\biggr)
\biggl(\frac{\partial\hfill}{\partial \theta^s}\log f(X_a;\theta_*)\biggr)
\biggr]\nonumber\\
&\qquad\qquad\times E\biggl[\biggl(\frac{\partial^2\hfill}{\partial \theta^k\partial \theta^l}\log f(X_b\theta)+\tilde{g}_{kl}\biggr)\biggr]E\biggl[\biggl(\frac{\partial\hfill}{\partial \theta^t}\log f(X_c;\theta_*)\biggr)\biggr]\nonumber\\
&\qquad\quad+\sum_{a\ne b, a\ne c, b\ne c}E\biggl[\biggl(\frac{\partial^2\hfill}{\partial \theta^u\partial \theta^m}\log f(X_a;\theta_*)+\tilde{g}_{um}\biggr)
\biggl(\frac{\partial\hfill}{\partial \theta^t}\log f(X_a;\theta_*)\biggr)
\biggr]\nonumber\\
&\qquad\qquad\times E\biggl[\biggl(\frac{\partial^2\hfill}{\partial \theta^k\partial \theta^l}\log f(X_b\theta)+\tilde{g}_{kl}\biggr)\biggr]E\biggl[\biggl(\frac{\partial\hfill}{\partial \theta^s}\log f(X_c;\theta_*)\biggr)\biggr]\nonumber\\
&\qquad\quad+\sum_{a\ne b, a\ne c, b\ne c}E\biggl[
\biggl(\frac{\partial\hfill}{\partial \theta^s}\log f(X_a;\theta_*)\biggr)
\biggl(\frac{\partial\hfill}{\partial \theta^t}\log f(X_a;\theta_*)\biggr)
\biggr]\nonumber\\
&\qquad\qquad\times E\biggl[\biggl(\frac{\partial^2\hfill}{\partial \theta^k\partial \theta^l}\log f(X_b;\theta_*)+\tilde{g}_{kl}\biggr)\biggr]E\biggl[\biggl(\frac{\partial^2\hfill}{\partial \theta^u\partial \theta^m}\log f(X_c;\theta_*)+\tilde{g}_{um}\biggr)\biggr]\nonumber\\
&\qquad\quad+\sum_{a\ne b, a\ne c, a\ne d, b\ne c, b\ne d, c\ne d}E\biggl[
\biggl(\frac{\partial\hfill}{\partial \theta^s}\log f(X_a;\theta_*)\biggr)\biggr]
E\biggl[\biggl(\frac{\partial\hfill}{\partial \theta^t}\log f(X_b;\theta_*)\biggr)
\biggr]\nonumber\\
&\qquad\qquad\times E\biggl[\biggl(\frac{\partial^2\hfill}{\partial \theta^k\partial \theta^l}\log f(X_c;\theta_*)+\tilde{g}_{kl}\biggr)\biggr]E\biggl[\biggl(\frac{\partial^2\hfill}{\partial \theta^u\partial \theta^m}\log f(X_d;\theta_*)+\tilde{g}_{um}\biggr)\biggr]
\biggr\}\nonumber\\
&=n^{-2}\tilde{g}^{jk}\tilde{g}^{lu}\tilde{g}^{is}\tilde{g}^{mt}\nonumber\\
&\qquad\times \biggl\{ E\biggl[\biggl(\frac{\partial^2\hfill}{\partial \theta^k\partial \theta^l}\log f(X;\theta_*)+\tilde{g}_{kl}\biggr)\biggl(\frac{\partial^2\hfill}{\partial \theta^u\partial \theta^m}\log f(X;\theta_*)+\tilde{g}_{um}\biggr)\biggr]\nonumber\\
&\qquad\qquad\times E\biggl[\biggl(\frac{\partial\hfill}{\partial \theta^s}\log f(X;\theta_*)\biggr)\biggl(\frac{\partial\hfill}{\partial \theta^t}\log f(X;\theta_*)\biggr)\biggr]\nonumber\\
&\qquad\quad+ E\biggl[\biggl(\frac{\partial^2\hfill}{\partial \theta^k\partial \theta^l}\log f(X;\theta_*)+\tilde{g}_{kl}\biggr)\biggl(\frac{\partial\hfill}{\partial \theta^s}\log f(X;\theta_*)\biggr)\biggr]\nonumber\\
&\qquad\qquad\times E\biggl[\biggl(\frac{\partial^2\hfill}{\partial \theta^u\partial \theta^m}\log f(X;\theta_*)+\tilde{g}_{um}\biggr)\biggl(\frac{\partial\hfill}{\partial \theta^t}\log f(X;\theta_*)\biggr)\biggr]\nonumber\\
&\qquad\quad+ E\biggl[\biggl(\frac{\partial^2\hfill}{\partial \theta^k\partial \theta^l}\log f(X;\theta_*)+\tilde{g}_{kl}\biggr)\biggl(\frac{\partial\hfill}{\partial \theta^t}\log f(X;\theta_*)\biggr)\biggr]\nonumber\\
&\qquad\qquad\times E\biggl[\biggl(\frac{\partial^2\hfill}{\partial \theta^u\partial \theta^m}\log f(X;\theta_*)+\tilde{g}_{um}\biggr)\biggl(\frac{\partial\hfill}{\partial \theta^s}\log f(X;\theta_*)\biggr)\biggr]\biggr\}+O(n^{-3}). \label{aaa^2_bareo^2}\nonumber\\
&=n^{-2}\tilde{g}^{jk}\tilde{g}^{lu}\tilde{g}^{is}\tilde{g}^{mt} \nonumber\\
&\qquad\times(L_{(kl)(um)}g_{st}-\tilde{g}_{kl}\tilde{g}_{um}g_{st}+L_{(kl)s}L_{(um)t}+L_{(kl)t}L_{(um)s})
+O(n^{-3}).
\end{align}
\begin{equation}
\label{aaa_bbar_bareo^3}
\begin{split}
&E[\aaa{l}{j}\bbar{ms}{l}\bareo{i}\bareo{m}\bareo{s}]\\
&=n^{-4}\bbar{ms}{l}E\biggl[\tilde{g}^{jk}\tilde{g}^{it}\tilde{g}^{mu}\tilde{g}^{sv}\biggl(\sum_{a=1}^n \biggl(\frac{\partial^2\hfill}{\partial\theta^l\partial\theta^k}\log f(X_a;\theta_*)+\tilde{g}_{kl}\biggr)\biggr)\biggl(\sum_{b=1}^n\frac{\partial\hfill}{\partial\theta^t}\log f(X_b;\theta_*)\biggr)\\
&\hspace{25mm}\times\biggl(\sum_{c=1}^n\frac{\partial\hfill}{\partial\theta^u}\log f(X_c;\theta_*)\biggr)\biggl(\sum_{d=1}^n\frac{\partial\hfill}{\partial\theta^v}\log f(X_d;\theta_*)\biggr)\biggr]\\
&=n^{-4}\bbar{ms}{l}\tilde{g}^{jk}\tilde{g}^{it}\tilde{g}^{mu}\tilde{g}^{sv}\\
&\qquad\times\biggl\{\sum_{a\ne b} E\biggl[\biggl(\frac{\partial^2\hfill}{\partial\theta^l\partial\theta^k}\log f(X_a;\theta_*)+\tilde{g}_{kl}\biggr)\biggl(\frac{\partial\hfill}{\partial\theta^t}\log f(X_a;\theta_*)\biggr)\biggr]\\
&\qquad\qquad\qquad\times E\biggl[\biggl(\frac{\partial\hfill}{\partial\theta^u}\log f(X_b;\theta_*)\biggr)\biggl(\frac{\partial\hfill}{\partial\theta^v}\log f(X_b;\theta_*)\biggr)
\biggr]\\
&\qquad\qquad+\sum_{a\ne b} E\biggl[\biggl(\frac{\partial^2\hfill}{\partial\theta^l\partial\theta^k}\log f(X_a;\theta_*)+\tilde{g}_{kl}\biggr)\biggl(\frac{\partial\hfill}{\partial\theta^u}\log f(X_a;\theta_*)\biggr)\biggr]\\
&\qquad\qquad\qquad\times E\biggl[\biggl(\frac{\partial\hfill}{\partial\theta^t}\log f(X_b;\theta_*)\biggr)\biggl(\frac{\partial\hfill}{\partial\theta^v}\log f(X_b;\theta_*)\biggr)
\biggr]\\
&\qquad\qquad+\sum_{a\ne b} E\biggl[\biggl(\frac{\partial^2\hfill}{\partial\theta^l\partial\theta^k}\log f(X_a;\theta_*)+\tilde{g}_{kl}\biggr)\biggl(\frac{\partial\hfill}{\partial\theta^v}\log f(X_a;\theta_*)\biggr)\biggr]\\
&\qquad\qquad\qquad\times E\biggl[\biggl(\frac{\partial\hfill}{\partial\theta^u}\log f(X_b;\theta_*)\biggr)\biggl(\frac{\partial\hfill}{\partial\theta^t}\log f(X_b;\theta_*)\biggr)
\biggr]
\biggr\}+O(n^{-3})\\
&=n^{-2}\bbar{ms}{l}\tilde{g}^{jk}\tilde{g}^{it}\tilde{g}^{mu}\tilde{g}^{sv}\\
&\qquad\times\biggl\{
E\biggl[\biggl(\frac{\partial^2\hfill}{\partial\theta^l\partial\theta^k}\log f(X;\theta_*)+\tilde{g}_{kl}\biggr)\biggl(\frac{\partial\hfill}{\partial\theta^t}\log f(X;\theta_*)\biggr)\biggr]\\
&\qquad\qquad\qquad\times E\biggl[\biggl(\frac{\partial\hfill}{\partial\theta^u}\log f(X;\theta_*)\biggr)\biggl(\frac{\partial\hfill}{\partial\theta^v}\log f(X;\theta_*)\biggr)\biggr]
\biggr\}\\
&\qquad\qquad+E\biggl[\biggl(\frac{\partial^2\hfill}{\partial\theta^l\partial\theta^k}\log f(X;\theta_*)+\tilde{g}_{kl}\biggr)\biggl(\frac{\partial\hfill}{\partial\theta^u}\log f(X;\theta_*)\biggr)\biggr]\\
&\qquad\qquad\qquad\times E\biggl[\biggl(\frac{\partial\hfill}{\partial\theta^t}\log f(X;\theta_*)\biggr)\biggl(\frac{\partial\hfill}{\partial\theta^v}\log f(X;\theta_*)\biggr)
\biggr]\\
&\qquad\qquad+E\biggl[\biggl(\frac{\partial^2\hfill}{\partial\theta^l\partial\theta^k}\log f(X;\theta_*)+\tilde{g}_{kl}\biggr)\biggl(\frac{\partial\hfill}{\partial\theta^v}\log f(X;\theta_*)\biggr)\biggr]\\
&\qquad\qquad\qquad\times E\biggl[\biggl(\frac{\partial\hfill}{\partial\theta^t}\log f(X;\theta_*)\biggr)\biggl(\frac{\partial\hfill}{\partial\theta^u}\log f(X;\theta_*)\biggr)
\biggr]\biggr\}+O(n^{-3})\\
&=n^{-2}\bbar{ms}{l}\tilde{g}^{jk}\tilde{g}^{it}\tilde{g}^{mu}\tilde{g}^{sv}(L_{(lk)t}g_{uv}+L_{(lk)u}g_{tv}+L_{(lk)v}g_{tu}+O(n^{-3}).
\end{split}
\end{equation}
\begin{align}
&E[\bee{lm}{j}\bareo{i}\bareo{l}\bareo{m}]\nonumber\\
&=n^{-4}2^{-1}\tilde{g}^{js}\tilde{g}^{it}\tilde{g}^{lu}\tilde{g}^{mv}\nonumber\\
&\quad\times E\biggl[\sum_{a=1}^n\biggl(\frac{\partial^3\hfill}{\partial\theta^s\partial\theta^l\partial\theta^m} \log f(X_a;\theta_*)-E\biggl[\frac{\partial^3\hfill}{\partial\theta^s\partial\theta^l\partial\theta^m} \log f(X_a;\theta_*)\biggr]\biggr)\nonumber\\
&\qquad\qquad\times \biggl(\sum_{b=1}^n\frac{\partial\hfill}{\partial\theta^t}\log f(X_b;\theta_*)\biggr)\biggl(\sum_{c=1}^n\frac{\partial\hfill}{\partial\theta^u}\log f(X_c;\theta_*)\biggr)\biggl(\sum_{d=1}^n\frac{\partial\hfill}{\partial\theta^v}\log f(X_d;\theta_*)\biggr)
\biggr]\nonumber\\
&=n^{-2}2^{-1}\tilde{g}^{js}\tilde{g}^{it}\tilde{g}^{lu}\tilde{g}^{mv}\nonumber\\
&\quad\times \biggl\{E\biggl[\biggl(\frac{\partial^3\hfill}{\partial\theta^s\partial\theta^l\partial\theta^m} \log f(X;\theta_*)-E\biggl[\frac{\partial^3\hfill}{\partial\theta^s\partial\theta^l\partial\theta^m} \log f(X;\theta_*)\biggr]\biggr) \biggl(\frac{\partial\hfill}{\partial\theta^t}\log f(X;\theta_*)\biggr)\biggr]\nonumber\\
&\qquad\qquad\times E\biggl[\biggl(\frac{\partial\hfill}{\partial\theta^u}\log f(X;\theta_*)\biggr)\biggl(\frac{\partial\hfill}{\partial\theta^v}\log f(X;\theta_*)\biggr)
\biggr]\nonumber\\
&\qquad+E\biggl[\biggl(\frac{\partial^3\hfill}{\partial\theta^s\partial\theta^l\partial\theta^m} \log f(X;\theta_*)-E\biggl[\frac{\partial^3\hfill}{\partial\theta^s\partial\theta^l\partial\theta^m} \log f(X;\theta_*)\biggr]\biggr) \biggl(\frac{\partial\hfill}{\partial\theta^u}\log f(X;\theta_*)\biggr)\biggr]\nonumber\\
&\qquad\qquad\times E\biggl[\biggl(\frac{\partial\hfill}{\partial\theta^t}\log f(X;\theta_*)\biggr)\biggl(\frac{\partial\hfill}{\partial\theta^v}\log f(X;\theta_*)\biggr)
\biggr]\nonumber\\
&\qquad+E\biggl[\biggl(\frac{\partial^3\hfill}{\partial\theta^s\partial\theta^l\partial\theta^m} \log f(X;\theta_*)-E\biggl[\frac{\partial^3\hfill}{\partial\theta^s\partial\theta^l\partial\theta^m} \log f(X;\theta_*)\biggr]\biggr) \biggl(\frac{\partial\hfill}{\partial\theta^v}\log f(X;\theta_*)\biggr)\biggr]\nonumber\\
&\qquad\qquad\times E\biggl[\biggl(\frac{\partial\hfill}{\partial\theta^t}\log f(X;\theta_*)\biggr)\biggl(\frac{\partial\hfill}{\partial\theta^u}\log f(X;\theta_*)\biggr)
\biggr]\biggr\}+O(n^{-3})\nonumber\\
&=n^{-2}2^{-1}\tilde{g}^{js}\tilde{g}^{it}\tilde{g}^{lu}\tilde{g}^{mv}
(L_{(slm)t}g_{uv}+L_{(slm)u}g_{tv}+L_{(slm)v}g_{tu}+O(n^{-3}).\label{bee_bareo^3}
\end{align}
\begin{align}
&E[\bbar{lm}{j}\aaa{s}{m}\bareo{i}\bareo{l}\bareo{s}]\nonumber\\
&=n^{-4}\bbar{lm}{j}\tilde{g}^{mt}\tilde{g}^{iu}\tilde{g}^{lv}\tilde{g}^{sw}\nonumber\\
&\quad\times E\biggl[\biggl(\sum_{a=1}^p\biggl(\frac{\partial^2\hfill}{\partial\theta^t\partial\theta^s}\log f(X_a;\theta_*)+\tilde{g}_{ts}\biggr)\biggr)\biggl(\sum_{b=1}^n\frac{\partial\hfill}{\partial\theta^u}\log f(X_b;\theta_*)\biggr)\nonumber\\
&\qquad\qquad\times\biggl(\sum_{c=1}^n\frac{\partial\hfill}{\partial\theta^v}\log f(X_c;\theta_*)\biggr)\biggl(\sum_{d=1}^n\frac{\partial\hfill}{\partial\theta^w}\log f(X_d;\theta_*)\biggr)
\biggr]\nonumber\\
&=n^{-4}\bbar{lm}{j}\tilde{g}^{mt}\tilde{g}^{iu}\tilde{g}^{lv}\tilde{g}^{sw}\nonumber\\
&\quad\times \biggl\{
\sum_{a\ne b}E\biggl[ \biggl(\frac{\partial^2\hfill}{\partial\theta^t\partial\theta^s}\log f(X_a;\theta_*)+\tilde{g}_{ts}\biggr)\biggl(\frac{\partial\hfill}{\partial\theta^u}\log f(X_a;\theta_*)\biggr)\biggr]\nonumber\\
&\qquad\qquad\times E\biggl[\biggl(\frac{\partial\hfill}{\partial\theta^v}\log f(X_b;\theta_*)\biggr)\biggl(\frac{\partial\hfill}{\partial\theta^w}\log f(X_b;\theta_*)\biggr)\biggr]\nonumber\\
&\qquad\quad+\sum_{a\ne b}E\biggl[ \biggl(\frac{\partial^2\hfill}{\partial\theta^t\partial\theta^s}\log f(X_a;\theta_*)+\tilde{g}_{ts}\biggr)\biggl(\frac{\partial\hfill}{\partial\theta^v}\log f(X_a;\theta_*)\biggr)\biggr]\nonumber\\
&\qquad\qquad\times E\biggl[\biggl(\frac{\partial\hfill}{\partial\theta^u}\log f(X_b;\theta_*)\biggr)\biggl(\frac{\partial\hfill}{\partial\theta^w}\log f(X_b;\theta_*)\biggr)\biggr]\nonumber\\
&\qquad\quad+\sum_{a\ne b}E\biggl[ \biggl(\frac{\partial^2\hfill}{\partial\theta^t\partial\theta^s}\log f(X_a;\theta_*)+\tilde{g}_{ts}\biggr)\biggl(\frac{\partial\hfill}{\partial\theta^w}\log f(X_a;\theta_*)\biggr)\biggr]\nonumber\\
&\qquad\qquad\times E\biggl[\biggl(\frac{\partial\hfill}{\partial\theta^u}\log f(X_b;\theta_*)\biggr)\biggl(\frac{\partial\hfill}{\partial\theta^v}\log f(X_b;\theta_*)\biggr)\biggr]
\biggr\}+O(n^{-3})\nonumber\\
&=n^{-2}\bbar{lm}{j}\tilde{g}^{mt}\tilde{g}^{iu}\tilde{g}^{lv}\tilde{g}^{sw}\nonumber\\
&\quad\times\biggl\{
E\biggl[ \biggl(\frac{\partial^2\hfill}{\partial\theta^t\partial\theta^s}\log f(X;\theta_*)+\tilde{g}_{ts}\biggr)\biggl(\frac{\partial\hfill}{\partial\theta^u}\log f(X;\theta_*)\biggr)\biggr]\nonumber\\
&\qquad\qquad\times E\biggl[\biggl(\frac{\partial\hfill}{\partial\theta^v}\log f(X;\theta_*)\biggr)\biggl(\frac{\partial\hfill}{\partial\theta^w}\log f(X;\theta_*)\biggr)\biggr]\nonumber\\
&\qquad\quad+E\biggl[ \biggl(\frac{\partial^2\hfill}{\partial\theta^t\partial\theta^s}\log f(X;\theta_*)+\tilde{g}_{ts}\biggr)\biggl(\frac{\partial\hfill}{\partial\theta^v}\log f(X;\theta_*)\biggr)\biggr]\nonumber\\
&\qquad\qquad\times E\biggl[\biggl(\frac{\partial\hfill}{\partial\theta^u}\log f(X;\theta_*)\biggr)\biggl(\frac{\partial\hfill}{\partial\theta^w}\log f(X;\theta_*)\biggr)\biggr]\nonumber\\
&\qquad\quad+E\biggl[ \biggl(\frac{\partial^2\hfill}{\partial\theta^t\partial\theta^s}\log f(X;\theta_*)+\tilde{g}_{ts}\biggr)\biggl(\frac{\partial\hfill}{\partial\theta^w}\log f(X;\theta_*)\biggr)\biggr]\nonumber\\
&\qquad\qquad\times E\biggl[\biggl(\frac{\partial\hfill}{\partial\theta^u}\log f(X;\theta_*)\biggr)\biggl(\frac{\partial\hfill}{\partial\theta^v}\log f(X;\theta_*)\biggr)\biggr]
\biggr\}+O(n^{-3})\nonumber\\
&=n^{-2}\bbar{lm}{j}\tilde{g}^{mt}\tilde{g}^{iu}\tilde{g}^{lv}\tilde{g}^{sw}(L_{(ts)u}g_{vw}
+L_{(ts)v}g_{uw}+L_{(ts)w}g_{uv})+O(n^{-3}). \label{bbar_aaa_bareo^3}
\end{align}
\begin{align}
&E[\bbar{lm}{j}\bbar{st}{m}\bareo{i}\bareo{l}\bareo{s}\bareo{t}]\nonumber\\
&=\bbar{lm}{j}\bbar{st}{m}E[\bareo{i}\bareo{l}\bareo{s}\bareo{t}]\nonumber\\
&=n^{-4}\bbar{lm}{j}\bbar{st}{m}\tilde{g}^{ik}\tilde{g}^{lu}\tilde{g}^{sv}\tilde{g}^{tw}\nonumber\\
&\quad\times E\biggl[\biggl(\sum_{a=1}^n\frac{\partial\hfill}{\partial\theta^k}\log f(X_a;\theta_*)\biggr)\biggl(\sum_{b=1}^n\frac{\partial\hfill}{\partial\theta^u}\log f(X_b;\theta_*)\biggr)\nonumber\\
&\qquad\qquad\times\biggl(\sum_{c=1}^n\frac{\partial\hfill}{\partial\theta^v}\log f(X_c;\theta_*)\biggr)\biggl(\sum_{d=1}^n\frac{\partial\hfill}{\partial\theta^w}\log f(X_d;\theta_*)\biggr)\biggr]\nonumber\\
&=n^{-2}\bbar{lm}{j}\bbar{st}{m}\tilde{g}^{ik}\tilde{g}^{lu}\tilde{g}^{sv}\tilde{g}^{tw}\nonumber\\
&\quad\times E\biggl[\biggl(\frac{\partial\hfill}{\partial\theta^k}\log f(X;\theta_*)\biggr)\biggl(\frac{\partial\hfill}{\partial\theta^u}\log f(X;\theta_*)\biggr)\biggr]\nonumber\\
&\qquad\quad\times E\biggl[\biggl(\frac{\partial\hfill}{\partial\theta^v}\log f(X;\theta_*)\biggr)\biggl(\frac{\partial\hfill}{\partial\theta^w}\log f(X;\theta_*)\biggr)\biggr]\nonumber\\
&\quad\quad+ E\biggl[\biggl(\frac{\partial\hfill}{\partial\theta^k}\log f(X;\theta_*)\biggr)\biggl(\frac{\partial\hfill}{\partial\theta^v}\log f(X;\theta_*)\biggr)\biggr]\nonumber\\
&\qquad\quad\times E\biggl[\biggl(\frac{\partial\hfill}{\partial\theta^u}\log f(X;\theta_*)\biggr)\biggl(\frac{\partial\hfill}{\partial\theta^w}\log f(X;\theta_*)\biggr)\biggr]\nonumber\\
&\quad\quad+ E\biggl[\biggl(\frac{\partial\hfill}{\partial\theta^k}\log f(X;\theta_*)\biggr)\biggl(\frac{\partial\hfill}{\partial\theta^w}\log f(X;\theta_*)\biggr)\biggr]\nonumber\\
&\qquad\quad\times E\biggl[\biggl(\frac{\partial\hfill}{\partial\theta^u}\log f(X;\theta_*)\biggr)\biggl(\frac{\partial\hfill}{\partial\theta^v}\log f(X;\theta_*)\biggr)\biggr]\biggr\}+O(n^{-3})\nonumber\\
&=n^{-2}\bbar{lm}{j}\bbar{st}{m}\tilde{g}^{ik}\tilde{g}^{lu}\tilde{g}^{sv}\tilde{g}^{tw}
(g_{ku}g_{vw}+g_{kv}g_{uw}+g_{kw}g_{uv})+O(n^{-3}).\label{bbar^2_bareo^4}
\end{align}
\begin{align}
&E[\cbar{lmt}{j}\bareo{i}\bareo{l}\bareo{m}\bareo{t}]\nonumber\\
&=n^{-4}\cbar{lmt}{j}\tilde{g}^{ik}\tilde{g}^{ls}\tilde{g}^{mu}\tilde{g}^{tv}\nonumber\\
&\quad\times E\biggl[\biggl(\sum_{a=1}^n\frac{\partial\hfill}{\partial\theta^k}\log f(X_a;\theta_*)\biggr)\biggl(\sum_{b=1}^n\frac{\partial\hfill}{\partial\theta^s}\log f(X_b;\theta_*)\biggr)\nonumber\\
&\qquad\qquad\times\biggl(\sum_{c=1}^n\frac{\partial\hfill}{\partial\theta^u}\log f(X_c;\theta_*)\biggr)\biggl(\sum_{d=1}^n\frac{\partial\hfill}{\partial\theta^v}\log f(X_d;\theta_*)\biggr)\biggr]\nonumber\\
&=n^{-2}\cbar{lmt}{j}\tilde{g}^{ik}\tilde{g}^{ls}\tilde{g}^{mu}\tilde{g}^{tv}\nonumber\\
&\quad\times E\biggl[\biggl(\frac{\partial\hfill}{\partial\theta^k}\log f(X;\theta_*)\biggr)\biggl(\frac{\partial\hfill}{\partial\theta^s}\log f(X;\theta_*)\biggr)\biggr]\nonumber\\
&\qquad\quad\times E\biggl[\biggl(\frac{\partial\hfill}{\partial\theta^u}\log f(X;\theta_*)\biggr)\biggl(\frac{\partial\hfill}{\partial\theta^v}\log f(X;\theta_*)\biggr)\biggr]\nonumber\\
&\quad\quad+ E\biggl[\biggl(\frac{\partial\hfill}{\partial\theta^k}\log f(X;\theta_*)\biggr)\biggl(\frac{\partial\hfill}{\partial\theta^u}\log f(X;\theta_*)\biggr)\biggr]\nonumber\\
&\qquad\quad\times E\biggl[\biggl(\frac{\partial\hfill}{\partial\theta^s}\log f(X;\theta_*)\biggr)\biggl(\frac{\partial\hfill}{\partial\theta^v}\log f(X;\theta_*)\biggr)\biggr]\nonumber\\
&\quad\quad+ E\biggl[\biggl(\frac{\partial\hfill}{\partial\theta^k}\log f(X;\theta_*)\biggr)\biggl(\frac{\partial\hfill}{\partial\theta^v}\log f(X;\theta_*)\biggr)\biggr]\nonumber\\
&\qquad\quad\times E\biggl[\biggl(\frac{\partial\hfill}{\partial\theta^s}\log f(X;\theta_*)\biggr)\biggl(\frac{\partial\hfill}{\partial\theta^u}\log f(X;\theta_*)\biggr)\biggr]\biggr\}+O(n^{-3})\nonumber\\
&=n^{-2}\cbar{lmt}{j}\tilde{g}^{ik}\tilde{g}^{ls}\tilde{g}^{mu}\tilde{g}^{tv}
(g_{ks}g_{uv}+g_{ku}g_{sv}+g_{kv}g_{su})+O(n^{-3}).\label{cbar_bareo^4}
\end{align}
\begin{align}
&E[\aaa{l}{i}\aaa{m}{j}\bareo{l}\bareo{m}]\nonumber\\
&=n^{-4}\tilde{g}^{ik}\tilde{g}^{js}\tilde{g}^{lt}\tilde{g}^{mu}\nonumber\\
&\times E\biggl[\biggl(\sum_{a=1}^n\biggl(\frac{\partial^2\hfill}{\partial\theta^k\partial\theta^l}\log f(X_a;\theta_*)+\tilde{g}_{kl}\biggr)\biggr)\biggl(\sum_{b=1}^n\biggl(\frac{\partial^2\hfill}{\partial\theta^s\partial\theta^m}\log f(X_b;\theta_*)+\tilde{g}_{sm}\biggr)\biggr)\nonumber\\
&\qquad\times \biggl(\sum_{c=1}^n\frac{\partial\hfill}{\partial\theta^t}\log f(X_c;\theta_*)\biggr)\biggl(\sum_{d=1}^n\frac{\partial\hfill}{\partial\theta^u}\log f(X_d;\theta_*)\biggr)\biggr]\nonumber\\
&=n^{-2}\tilde{g}^{ik}\tilde{g}^{js}\tilde{g}^{lt}\tilde{g}^{mu}\nonumber\\
&\times \biggl\{
E\biggl[\biggl(\frac{\partial^2\hfill}{\partial\theta^k\partial\theta^l}\log f(X;\theta_*)+\tilde{g}_{kl}\biggr)\biggl(\frac{\partial^2\hfill}{\partial\theta^s\partial\theta^m}\log f(X;\theta_*)+\tilde{g}_{sm}\biggr)\biggr]\nonumber\\
&\qquad\times E\biggl[\biggl(\frac{\partial\hfill}{\partial\theta^t}\log f(X;\theta_*)\biggr)\biggl(\frac{\partial\hfill}{\partial\theta^u}\log f(X;\theta_*)\biggr)\biggr]\nonumber\\
&\quad+E\biggl[ \biggl(\frac{\partial^2\hfill}{\partial\theta^k\partial\theta^l}\log f(X;\theta_*)+\tilde{g}_{kl}\biggr)\biggl(\frac{\partial\hfill}{\partial\theta^t}\log f(X;\theta_*)\biggr)\biggr]\nonumber\\
&\qquad\times E\biggl[\biggl(\frac{\partial^2\hfill}{\partial\theta^s\partial\theta^m}\log f(X;\theta_*)+\tilde{g}_{sm}\biggr)\biggl(\frac{\partial\hfill}{\partial\theta^u}\log f(X;\theta_*)\biggr)\biggr]\nonumber\\
&\quad+E\biggl[ \biggl(\frac{\partial^2\hfill}{\partial\theta^k\partial\theta^l}\log f(X;\theta_*)+\tilde{g}_{kl}\biggr)\biggl(\frac{\partial\hfill}{\partial\theta^u}\log f(X;\theta_*)\biggr)\biggr]\nonumber\\
&\qquad\times E\biggl[\biggl(\frac{\partial^2\hfill}{\partial\theta^s\partial\theta^m}\log f(X;\theta_*)+\tilde{g}_{sm}\biggr)\biggl(\frac{\partial\hfill}{\partial\theta^t}\log f(X;\theta_*)\biggr)\biggr]
\bigg\}+O(n^{-3})\nonumber\\
&=n^{-2}\tilde{g}^{ik}\tilde{g}^{js}\tilde{g}^{lt}\tilde{g}^{mu}
(L_{(kl)(sm)}g_{tu}-\tilde{g}_{kl}\tilde{g}_{sm}g_{tu}+L_{(kl)t}L_{(sm)u}+L_{(kl)u}L_{(sm)t})
+O(n^{-3}).\label{aaa^2_bareo^2_2}
\end{align}
\begin{align}
&E[\aaa{l}{i}\bbar{st}{j}\bareo{l}\bareo{s}\bareo{t}]\nonumber\\
&=n^{-4}\bbar{st}{j}\tilde{g}^{ik}\tilde{g}^{lu}\tilde{g}^{sv}\tilde{g}^{tw}\nonumber\\
&\quad\times E\biggl[ \biggl(\sum_{a=1}^n\biggl(\frac{\partial^2\hfill}{\partial\theta^k\partial\theta^l}\log f(X_a;\theta_*)+\tilde{g}_{kl}\biggr)\biggr)
 \biggl(\sum_{b=1}^n\frac{\partial\hfill}{\partial\theta^u}\log f(X_b;\theta_*)\biggr)\nonumber\\
&\qquad \qquad\times\biggl(\sum_{c=1}^n\frac{\partial\hfill}{\partial\theta^v}\log f(X_c;\theta_*)\biggr)\biggl(\sum_{d=1}^n\frac{\partial\hfill}{\partial\theta^w}\log f(X_d;\theta_*)\biggr)
\biggr]\nonumber\\
&=n^{-2}\bbar{st}{j}\tilde{g}^{ik}\tilde{g}^{lu}\tilde{g}^{sv}\tilde{g}^{tw}\nonumber\\
&\quad \times\biggl\{
 E\biggl[\biggl(\frac{\partial^2\hfill}{\partial\theta^k\partial\theta^l}\log f(X;\theta_*)+\tilde{g}_{kl}\biggr)\biggl(\frac{\partial\hfill}{\partial\theta^u}\log f(X;\theta_*)\biggr)\biggr]\nonumber\\
&\qquad\qquad \times E\biggl[\biggl(\frac{\partial\hfill}{\partial\theta^v}\log f(X;\theta_*)\biggr)\biggl(\frac{\partial\hfill}{\partial\theta^w}\log f(X;\theta_*)\biggr)\biggr]\nonumber\\
&\qquad\quad+ E\biggl[\biggl(\frac{\partial^2\hfill}{\partial\theta^k\partial\theta^l}\log f(X;\theta_*)+\tilde{g}_{kl}\biggr)\biggl(\frac{\partial\hfill}{\partial\theta^v}\log f(X;\theta_*)\biggr)\biggr]\nonumber\\
&\qquad\qquad \times E\biggl[\biggl(\frac{\partial\hfill}{\partial\theta^u}\log f(X;\theta_*)\biggr)\biggl(\frac{\partial\hfill}{\partial\theta^w}\log f(X;\theta_*)\biggr)\biggr]\nonumber\\
&\qquad\quad+ E\biggl[\biggl(\frac{\partial^2\hfill}{\partial\theta^k\partial\theta^l}\log f(X;\theta_*)+\tilde{g}_{kl}\biggr)\biggl(\frac{\partial\hfill}{\partial\theta^w}\log f(X;\theta_*)\biggr)\biggr]\nonumber\\
&\qquad\qquad \times E\biggl[\biggl(\frac{\partial\hfill}{\partial\theta^u}\log f(X;\theta_*)\biggr)\biggl(\frac{\partial\hfill}{\partial\theta^v}\log f(X;\theta_*)\biggr)\biggr]
\biggr\}+O(n^{-3})\nonumber\\
&=n^{-2}\bbar{st}{j}\tilde{g}^{ik}\tilde{g}^{lu}\tilde{g}^{sv}\tilde{g}^{tw}
(L_{(kl)u}g_{vw}+L_{(kl)v}g_{uw}+L_{(kl)w}g_{uv})+O(n^{-3}).\label{aaa_bbar_bareo^3_2}
\end{align}
\begin{align}
&E[\bbar{lm}{i}\bbar{st}{j}\bareo{l}\bareo{m}\bareo{s}\bareo{t}]\nonumber\\
&=n^{-4}\bbar{lm}{i}\bbar{st}{j}\tilde{g}^{lk}\tilde{g}^{mu}\tilde{g}^{sv}\tilde{g}^{tw}\nonumber\\
&\quad\times 
E\biggl[\biggl(\sum_{a=1}^n\frac{\partial\hfill}{\partial\theta^k}\log f(X_a;\theta_*)\biggr)\biggl(\sum_{b=1}^n\frac{\partial\hfill}{\partial\theta^u}\log f(X_b;\theta_*)\biggr)\nonumber\\
&\qquad\qquad\times\biggl(\sum_{c=1}^n\frac{\partial\hfill}{\partial\theta^v}\log f(X_c;\theta_*)\biggr)\biggl(\sum_{d=1}^n\frac{\partial\hfill}{\partial\theta^w}\log f(X_d;\theta_*)\biggr)\biggr]\nonumber\\
&=n^{-2}\bbar{lm}{i}\bbar{st}{j}\tilde{g}^{lk}\tilde{g}^{mu}\tilde{g}^{sv}\tilde{g}^{tw}\nonumber\\
&\quad\times \biggl\{
E\biggl[\biggl(\frac{\partial\hfill}{\partial\theta^k}\log f(X;\theta_*)\biggr)\biggl(\frac{\partial\hfill}{\partial\theta^u}\log f(X;\theta_*)\biggr)\biggr]\nonumber\\
&\qquad\quad\times E\biggl[\biggl(\frac{\partial\hfill}{\partial\theta^v}\log f(X;\theta_*)\biggr)\biggl(\frac{\partial\hfill}{\partial\theta^w}\log f(X;\theta_*)\biggr)\biggr]\nonumber\\
&\quad\quad+ E\biggl[\biggl(\frac{\partial\hfill}{\partial\theta^k}\log f(X;\theta_*)\biggr)\biggl(\frac{\partial\hfill}{\partial\theta^v}\log f(X;\theta_*)\biggr)\biggr]\nonumber\\
&\qquad\quad\times E\biggl[\biggl(\frac{\partial\hfill}{\partial\theta^u}\log f(X;\theta_*)\biggr)\biggl(\frac{\partial\hfill}{\partial\theta^w}\log f(X;\theta_*)\biggr)\biggr]\nonumber\\
&\quad\quad+ E\biggl[\biggl(\frac{\partial\hfill}{\partial\theta^k}\log f(X;\theta_*)\biggr)\biggl(\frac{\partial\hfill}{\partial\theta^w}\log f(X;\theta_*)\biggr)\biggr]\nonumber\\
&\qquad\quad\times E\biggl[\biggl(\frac{\partial\hfill}{\partial\theta^u}\log f(X;\theta_*)\biggr)\biggl(\frac{\partial\hfill}{\partial\theta^v}\log f(X;\theta_*)\biggr)\biggr]\biggr\}+O(n^{-3})\nonumber\\
&=n^{-2}\bbar{lm}{i}\bbar{st}{j}\tilde{g}^{lk}\tilde{g}^{mu}\tilde{g}^{sv}\tilde{g}^{tw}
(g_{ku}g_{vw}+g_{kv}g_{uw}+g_{kw}g_{uv})+O(n^{-3}).\label{bbar^2_bareo^3}
\end{align}
%
% Rsult of $E[\barteta{i}\barteta{j}]$.
%
Consequently the following result holds.
\begin{equation}
\label{e_barteta_i*barteta_j}
\begin{split}
&E[(\hat{\theta}^i-\theta_*^i)(\hat{\theta}^j-\theta_*^j)]\\
&=n^{-1}\tilde{g}^{il}\tilde{g}^{jm}g_{lm}+n^{-2} \\
&\quad\times\Bigl(
\tilde{g}^{sj}\tilde{g}^{it}\tilde{g}^{lm}(L_{(sl)tm}+\tilde{g}_{ls}g_{tm})
+\tilde{g}^{si}\tilde{g}^{jt}\tilde{g}^{lm}(L_{(sl)tm}+\tilde{g}_{ls}g_{tm})\\
&\qquad+\bbar{lm}{j}\tilde{g}^{ik}\tilde{g}^{ls}\tilde{g}^{mt}L_{kst}
+\bbar{lm}{i}\tilde{g}^{jk}\tilde{g}^{ls}\tilde{g}^{mt}L_{kst}\\
&\qquad+\tilde{g}^{jk}\tilde{g}^{lu}\tilde{g}^{is}\tilde{g}^{mt} 
(L_{(kl)(um)}g_{st}-\tilde{g}_{kl}\tilde{g}_{um}g_{st}+L_{(kl)s}L_{(um)t}+L_{(kl)t}L_{(um)s})\\
&\qquad+\tilde{g}^{ik}\tilde{g}^{lu}\tilde{g}^{js}\tilde{g}^{mt} 
(L_{(kl)(um)}g_{st}-\tilde{g}_{kl}\tilde{g}_{um}g_{st}+L_{(kl)s}L_{(um)t}+L_{(kl)t}L_{(um)s})\\
&\qquad+\bbar{ms}{l}\tilde{g}^{jk}\tilde{g}^{it}\tilde{g}^{mu}\tilde{g}^{sv}(L_{(lk)t}g_{uv}+L_{(lk)u}g_{tv}+L_{(lk)v}g_{tu})\\
&\qquad+\bbar{ms}{l}\tilde{g}^{ik}\tilde{g}^{jt}\tilde{g}^{mu}\tilde{g}^{sv}(L_{(lk)t}g_{uv}+L_{(lk)u}g_{tv}+L_{(lk)v}g_{tu})\\
&\qquad+2^{-1}\tilde{g}^{js}\tilde{g}^{it}\tilde{g}^{lu}\tilde{g}^{mv}
(L_{(slm)t}g_{uv}+L_{(slm)u}g_{tv}+L_{(slm)v}g_{tu})\\
&\qquad+2^{-1}\tilde{g}^{is}\tilde{g}^{jt}\tilde{g}^{lu}\tilde{g}^{mv}
(L_{(slm)t}g_{uv}+L_{(slm)u}g_{tv}+L_{(slm)v}g_{tu})\\
&\qquad+2\bbar{lm}{j}\tilde{g}^{mt}\tilde{g}^{iu}\tilde{g}^{lv}\tilde{g}^{sw}(L_{(ts)u}g_{vw}
+L_{(ts)v}g_{uw}+L_{(ts)w}g_{uv})\\
&\qquad+2\bbar{lm}{i}\tilde{g}^{mt}\tilde{g}^{ju}\tilde{g}^{lv}\tilde{g}^{sw}(L_{(ts)u}g_{vw}
+L_{(ts)v}g_{uw}+L_{(ts)w}g_{uv})\\
&\qquad+2\bbar{lm}{j}\bbar{st}{m}\tilde{g}^{ik}\tilde{g}^{lu}\tilde{g}^{sv}\tilde{g}^{tw}
(g_{ku}g_{vw}+g_{kv}g_{uw}+g_{kw}g_{uv})\\
&\qquad+2\bbar{lm}{i}\bbar{st}{m}\tilde{g}^{jk}\tilde{g}^{lu}\tilde{g}^{sv}\tilde{g}^{tw}
(g_{ku}g_{vw}+g_{kv}g_{uw}+g_{kw}g_{uv})\\
&\qquad+\cbar{lmt}{j}\tilde{g}^{ik}\tilde{g}^{ls}\tilde{g}^{mu}\tilde{g}^{tv}
(g_{ks}g_{uv}+g_{ku}g_{sv}+g_{kv}g_{su})\\
&\qquad+\cbar{lmt}{i}\tilde{g}^{jk}\tilde{g}^{ls}\tilde{g}^{mu}\tilde{g}^{tv}
(g_{ks}g_{uv}+g_{ku}g_{sv}+g_{kv}g_{su})\\
&\qquad+\tilde{g}^{ik}\tilde{g}^{js}\tilde{g}^{lt}\tilde{g}^{mu}
(L_{(kl)(sm)}g_{tu}-\tilde{g}_{kl}\tilde{g}_{sm}g_{tu}+L_{(kl)t}L_{(sm)u}+L_{(kl)u}L_{(sm)t})\\
&\qquad+\bbar{st}{j}\tilde{g}^{ik}\tilde{g}^{lu}\tilde{g}^{sv}\tilde{g}^{tw}
(L_{(kl)u}g_{vw}+L_{(kl)v}g_{uw}+L_{(kl)w}g_{uv})\\
&\qquad+\bbar{st}{i}\tilde{g}^{jk}\tilde{g}^{lu}\tilde{g}^{sv}\tilde{g}^{tw}
(L_{(kl)u}g_{vw}+L_{(kl)v}g_{uw}+L_{(kl)w}g_{uv})\\
&\qquad+\bbar{lm}{i}\bbar{st}{j}\tilde{g}^{lk}\tilde{g}^{mu}\tilde{g}^{sv}\tilde{g}^{tw}
(g_{ku}g_{vw}+g_{kv}g_{uw}+g_{kw}g_{uv})
\Bigr)\\
&\quad + O(n^{-3})
\end{split}
\end{equation}
If we substitute $\bbar{ij}{k}$ and $\cbar{ijk}{l}$ for the expression with $\tilde{g}^{ij}$, $L_{(ijk)}$, $L_{(ijkl)}$, then
\begin{equation}
\begin{split}
&E[(\hat{\theta}^i-\theta_*^i)(\hat{\theta}^j-\theta_*^j)]\\
&=n^{-1}\tilde{g}^{il}\tilde{g}^{jm}g_{lm}+n^{-2} \\
&\quad\times\Bigl(
\tilde{g}^{sj}\tilde{g}^{it}\tilde{g}^{lm}(L_{(sl)tm}+\tilde{g}_{ls}g_{tm})
+\tilde{g}^{si}\tilde{g}^{jt}\tilde{g}^{lm}(L_{(sl)tm}+\tilde{g}_{ls}g_{tm})\\
&\qquad+ 2^{-1}\tilde{g}^{uj}\tilde{g}^{ik}\tilde{g}^{ls}\tilde{g}^{mt}L_{kst}L_{(lmu)}
+2^{-1}\tilde{g}^{ui}\tilde{g}^{jk}\tilde{g}^{ls}\tilde{g}^{mt}L_{kst}L_{(lmu)}\\
&\qquad+\tilde{g}^{jk}\tilde{g}^{lu}\tilde{g}^{is}\tilde{g}^{mt} 
(L_{(kl)(um)}g_{st}-\tilde{g}_{kl}\tilde{g}_{um}g_{st}+L_{(kl)s}L_{(um)t}+L_{(kl)t}L_{(um)s})\\
&\qquad+\tilde{g}^{ik}\tilde{g}^{lu}\tilde{g}^{js}\tilde{g}^{mt} 
(L_{(kl)(um)}g_{st}-\tilde{g}_{kl}\tilde{g}_{um}g_{st}+L_{(kl)s}L_{(um)t}+L_{(kl)t}L_{(um)s})\\
&\qquad+2^{-1}\tilde{g}^{jk}\tilde{g}^{it}\tilde{g}^{mu}\tilde{g}^{sv}\tilde{g}^{wl}L_{(msw)}(L_{(lk)t}g_{uv}+L_{(lk)u}g_{tv}+L_{(lk)v}g_{tu})\\
&\qquad+2^{-1}\tilde{g}^{ik}\tilde{g}^{jt}\tilde{g}^{mu}\tilde{g}^{sv}\tilde{g}^{wl}L_{(msw)}(L_{(lk)t}g_{uv}+L_{(lk)u}g_{tv}+L_{(lk)v}g_{tu})\\
&\qquad+2^{-1}\tilde{g}^{js}\tilde{g}^{it}\tilde{g}^{lu}\tilde{g}^{mv}
(L_{(slm)t}g_{uv}+L_{(slm)u}g_{tv}+L_{(slm)v}g_{tu})\\
&\qquad+2^{-1}\tilde{g}^{is}\tilde{g}^{jt}\tilde{g}^{lu}\tilde{g}^{mv}
(L_{(slm)t}g_{uv}+L_{(slm)u}g_{tv}+L_{(slm)v}g_{tu})\\
&\qquad+\tilde{g}^{mt}\tilde{g}^{iu}\tilde{g}^{lv}\tilde{g}^{sw}\tilde{g}^{kj}L_{(lmk)}(L_{(ts)u}g_{vw}
+L_{(ts)v}g_{uw}+L_{(ts)w}g_{uv})\\
&\qquad+\tilde{g}^{mt}\tilde{g}^{ju}\tilde{g}^{lv}\tilde{g}^{sw}\tilde{g}^{ki}L_{(lmk)}(L_{(ts)u}g_{vw}
+L_{(ts)v}g_{uw}+L_{(ts)w}g_{uv})\\
&\qquad+2^{-1}\tilde{g}^{ik}\tilde{g}^{lu}\tilde{g}^{sv}\tilde{g}^{tw}\tilde{g}^{oj}\tilde{g}^{hm}L_{(lmo)}L_{(sth)}
(g_{ku}g_{vw}+g_{kv}g_{uw}+g_{kw}g_{uv})\\
&\qquad+2^{-1}\tilde{g}^{jk}\tilde{g}^{lu}\tilde{g}^{sv}\tilde{g}^{tw}\tilde{g}^{oi}\tilde{g}^{hm}L_{(lmo)}L_{(sth)}
(g_{ku}g_{vw}+g_{kv}g_{uw}+g_{kw}g_{uv})\\
&\qquad+6^{-1}\tilde{g}^{ik}\tilde{g}^{ls}\tilde{g}^{mu}\tilde{g}^{tv}\tilde{g}^{wj}L_{(lmtw)}
(g_{ks}g_{uv}+g_{ku}g_{sv}+g_{kv}g_{su})\\
&\qquad+6^{-1}\tilde{g}^{jk}\tilde{g}^{ls}\tilde{g}^{mu}\tilde{g}^{tv}\tilde{g}^{wi}L_{(lmtw)}
(g_{ks}g_{uv}+g_{ku}g_{sv}+g_{kv}g_{su})\\
&\qquad+\tilde{g}^{ik}\tilde{g}^{js}\tilde{g}^{lt}\tilde{g}^{mu}
(L_{(kl)(sm)}g_{tu}-\tilde{g}_{kl}\tilde{g}_{sm}g_{tu}+L_{(kl)t}L_{(sm)u}+L_{(kl)u}L_{(sm)t})\\
&\qquad+2^{-1}\tilde{g}^{ik}\tilde{g}^{lu}\tilde{g}^{sv}\tilde{g}^{tw}\tilde{g}^{jm}L_{(stm)}
(L_{(kl)u}g_{vw}+L_{(kl)v}g_{uw}+L_{(kl)w}g_{uv})\\
&\qquad+2^{-1}\tilde{g}^{jk}\tilde{g}^{lu}\tilde{g}^{sv}\tilde{g}^{tw}\tilde{g}^{im}L_{(stm)}
(L_{(kl)u}g_{vw}+L_{(kl)v}g_{uw}+L_{(kl)w}g_{uv})\\
&\qquad+4^{-1}\tilde{g}^{lk}\tilde{g}^{mu}\tilde{g}^{sv}\tilde{g}^{tw}\tilde{g}^{io}\tilde{g}^{jh}L_{(lmo)}L_{(sth)}
(g_{ku}g_{vw}+g_{kv}g_{uw}+g_{kw}g_{uv})
\Bigr)\\
&\quad + O(n^{-3})
\end{split}
\end{equation}
%
%
%   $E[\barteta{i}\barteta{j}\barteta{k}]$
%
%
%
Now we consider $E[\barteta{i}\barteta{j}\barteta{k}]$. From \eqref{expan_barteta}, we have
\begin{align}
\barteta{i}\barteta{j}\barteta{k}&=\bareo{i}\bareo{j}\bareo{k}+\aaa{s}{i}\bareo{s}\bareo{j}\bareo{k}+\aaa{s}{j}\bareo{s}\bareo{i}\bareo{k}+\aaa{s}{k}\bareo{s}\bareo{i}\bareo{j}\nonumber\\
&\quad+\bbar{st}{i}\bareo{s}\bareo{t}\bareo{j}\bareo{k}+\bbar{st}{j}\bareo{s}\bareo{t}\bareo{i}\bareo{k}+\bbar{st}{k}\bareo{s}\bareo{t}\bareo{i}\bareo{j}+Re4, \label{barteta^3}
\end{align}
where $Re4$ is similarly defined as before. We evaluate the expectation of each term.
%
%  Each term of E[\barteta{i}\barteta{j}\barteta{k}]
%
\begin{align}
&E[\bareo{i}\bareo{j}\bareo{k}]\nonumber\\
&=n^{-3}E \biggl[\biggl(\tilde{g}^{is}\sum_{a=1}^n\biggl(\frac{\partial \hfill}{\partial \theta^s} \log f(X_a;\theta_*)\biggr)\biggr)\biggl(\tilde{g}^{jt}\sum_{b=1}^n\biggl(\frac{\partial \hfill}{\partial \theta^t} \log f(X_b;\theta_*)\biggr)\biggr)\nonumber\\
&\qquad\times\biggl(\tilde{g}^{ku}\sum_{c=1}^n\biggl(\frac{\partial \hfill}{\partial \theta^u} \log f(X_c;\theta_*)\biggr)\biggr)
\biggr]\nonumber\\
&=n^{-3}\sum_{a=1}^n \tilde{g}^{is}\tilde{g}^{jt}\tilde{g}^{ku} E\biggl[\biggl(\frac{\partial \hfill}{\partial \theta^s} \log f(X_a;\theta_*)\biggr)\biggl(\frac{\partial \hfill}{\partial \theta^t} \log f(X_a;\theta_*)\biggr)\biggl(\frac{\partial \hfill}{\partial \theta^u} \log f(X_a;\theta_*)\biggr)
\biggr]\nonumber\\
&\quad+n^{-3}\sum_{a\ne b} \tilde{g}^{is}\tilde{g}^{jt}\tilde{g}^{ku}\nonumber\\
&\quad\qquad\times \biggl\{ E\biggl[\biggl(\frac{\partial \hfill}{\partial \theta^s} \log f(X_a;\theta_*)\biggr)\biggl(\frac{\partial \hfill}{\partial \theta^t} \log f(X_a;\theta_*)\biggr)\biggr]E\biggl[\frac{\partial \hfill}{\partial \theta^u} \log f(X_b;\theta_*)\biggr]\nonumber\\
&\quad\qquad\qquad+E\biggl[\biggl(\frac{\partial \hfill}{\partial \theta^s} \log f(X_a;\theta_*)\biggr)\biggl(\frac{\partial \hfill}{\partial \theta^u} \log f(X_a;\theta_*)\biggr)\biggr]E\biggl[\frac{\partial \hfill}{\partial \theta^t} \log f(X_b;\theta_*) \biggr]\nonumber\\
&\quad\qquad\qquad+E\biggl[\biggl(\frac{\partial \hfill}{\partial \theta^t} \log f(X_a;\theta_*)\biggr)\biggl(\frac{\partial \hfill}{\partial \theta^u} \log f(X_a;\theta_*)\biggr)\biggr]E\biggl[\frac{\partial \hfill}{\partial \theta^s} \log f(X_b;\theta_*) \biggr]\biggr\}\nonumber\\
&\quad+n^{-3}\sum_{a\ne b, a\ne c, b \ne c} \tilde{g}^{is}\tilde{g}^{jt}\tilde{g}^{ku}\nonumber\\
&\qquad\qquad\times E\biggl[\frac{\partial \hfill}{\partial \theta^s} \log f(X_a;\theta_*)\biggr]E\biggl[\frac{\partial \hfill}{\partial \theta^t} \log f(X_b;\theta_*)\biggr]E\biggl[\frac{\partial \hfill}{\partial \theta^u} \log f(X_c;\theta_*)\biggr]\nonumber\\
&=n^{-2}\tilde{g}^{is}\tilde{g}^{jt}\tilde{g}^{ku}E\biggl[\biggl(\frac{\partial \hfill}{\partial \theta^s} \log f(X;\theta_*)\biggr)\biggl(\frac{\partial \hfill}{\partial \theta^t} \log f(X;\theta_*)\biggr)\biggl(\frac{\partial \hfill}{\partial \theta^u} \log f(X;\theta_*)\biggr)\biggr]+O(n^{-3})\nonumber\\
&=n^{-2}\tilde{g}^{is}\tilde{g}^{jt}\tilde{g}^{ku}L_{stu}+O(n^{-3}).\label{bareo^4}
\end{align}
\begin{align}
&E[\aaa{s}{i}\bareo{s}\bareo{j}\bareo{k}]\nonumber\\
&=n^{-4}\tilde{g}^{it}\tilde{g}^{su}\tilde{g}^{jv}\tilde{g}^{kw}\sum_{1\leq a, b, c, d \leq n} E\biggl[
\biggl(\frac{\partial^2\hfill}{\partial\theta^s\partial\theta^t}\log f(X_a;\theta_*)+\tilde{g}_{st}\biggr)\nonumber\\
&\quad\times\biggl(\frac{\partial \hfill}{\partial \theta^u} \log f(X_b;\theta_*)\biggr)\biggl(\frac{\partial \hfill}{\partial \theta^v} \log f(X_c;\theta_*)\biggr)\biggl(\frac{\partial \hfill}{\partial \theta^w} \log f(X_d;\theta_*)\biggr)
\biggr]\nonumber\\
&=n^{-4}\tilde{g}^{it}\tilde{g}^{su}\tilde{g}^{jv}\tilde{g}^{kw}\nonumber\\
&\quad\times \sum_{a\ne b}\biggl\{E\biggl[\biggl(\frac{\partial^2\hfill}{\partial\theta^s\partial\theta^t}\log f(X_a;\theta_*)+\tilde{g}_{st}\biggr)\biggl(\frac{\partial \hfill}{\partial \theta^u} \log f(X_a;\theta_*)\biggr)\biggr]\nonumber\\
&\qquad\qquad\times E\biggl[\biggl(\frac{\partial \hfill}{\partial \theta^v} \log f(X_b;\theta_*)\biggr)\biggl(\frac{\partial \hfill}{\partial \theta^w} \log f(X_b;\theta_*)\biggr)\biggr]\nonumber\\
&\qquad\qquad+E\biggl[\biggl(\frac{\partial^2\hfill}{\partial\theta^s\partial\theta^t}\log f(X_a;\theta_*)+\tilde{g}_{st}\biggr)\biggl(\frac{\partial \hfill}{\partial \theta^v} \log f(X_a;\theta_*)\biggr)\biggr]\nonumber\\
&\qquad\qquad\times E\biggl[\biggl(\frac{\partial \hfill}{\partial \theta^u} \log f(X_b;\theta_*)\biggr)\biggl(\frac{\partial \hfill}{\partial \theta^w} \log f(X_b;\theta_*)\biggr)\biggr]\nonumber\\
&\qquad\qquad+E\biggl[\biggl(\frac{\partial^2\hfill}{\partial\theta^s\partial\theta^t}\log f(X_a;\theta_*)+\tilde{g}_{st}\biggr)\biggl(\frac{\partial \hfill}{\partial \theta^w} \log f(X_a;\theta_*)\biggr)\biggr]\nonumber\\
&\qquad\qquad\times E\biggl[\biggl(\frac{\partial \hfill}{\partial \theta^u} \log f(X_b;\theta_*)\biggr)\biggl(\frac{\partial \hfill}{\partial \theta^v} \log f(X_b;\theta_*)\biggr)\biggr]\biggr\}+O(n^{-3})\nonumber\\
&=n^{-2}\tilde{g}^{it}\tilde{g}^{su}\tilde{g}^{jv}\tilde{g}^{kw}
(L_{(st)u}g_{vw}+L_{(st)v}g_{uw}+L_{(st)w}g_{uv})+O(n^{-3}). \label{aaa_bareo^3}
\end{align}
\begin{align}
&E[\bbar{st}{i}\bareo{s}\bareo{t}\bareo{j}\bareo{k}]\nonumber\\
&=n^{-4}\bbar{st}{i}\tilde{g}^{su}\tilde{g}^{tv}\tilde{g}^{jw}\tilde{g}^{km}\nonumber\\
&\quad\times\sum_{1 \leq a,b,c,d \leq n}E\biggl[\biggl(\frac{\partial \hfill}{\partial \theta^u} \log f(X_a;\theta_*)\biggr)\biggl(\frac{\partial \hfill}{\partial \theta^v} \log f(X_b;\theta_*)\biggr)\nonumber\\
&\qquad\qquad\times\biggl(\frac{\partial \hfill}{\partial \theta^w} \log f(X_c;\theta_*)\biggr)\biggl(\frac{\partial \hfill}{\partial \theta^m} \log f(X_d;\theta_*)\biggr)
\biggr]\nonumber\\
&=n^{-2}\bbar{st}{i}\tilde{g}^{su}\tilde{g}^{tv}\tilde{g}^{jw}\tilde{g}^{km}\nonumber\\
&\quad\times\biggl\{
E\biggl[ \biggl(\frac{\partial \hfill}{\partial \theta^u} \log f(X;\theta_*)\biggr)\biggl(\frac{\partial \hfill}{\partial \theta^v} \log f(X;\theta_*)\biggr)\biggr]\nonumber\\
&\qquad\qquad\times E\biggl[ \biggl(\frac{\partial \hfill}{\partial \theta^w} \log f(X;\theta_*)\biggr)\biggl(\frac{\partial \hfill}{\partial \theta^m} \log f(X;\theta_*)\biggr)\biggr]\nonumber\\
&\quad\qquad+E\biggl[ \biggl(\frac{\partial \hfill}{\partial \theta^u} \log f(X;\theta_*)\biggr)\biggl(\frac{\partial \hfill}{\partial \theta^w} \log f(X;\theta_*)\biggr)\biggr]\nonumber\\
&\qquad\qquad\times E\biggl[ \biggl(\frac{\partial \hfill}{\partial \theta^v} \log f(X;\theta_*)\biggr)\biggl(\frac{\partial \hfill}{\partial \theta^m} \log f(X;\theta_*)\biggr)\biggr]\nonumber\\
&\quad\qquad+E\biggl[ \biggl(\frac{\partial \hfill}{\partial \theta^u} \log f(X;\theta_*)\biggr)\biggl(\frac{\partial \hfill}{\partial \theta^m} \log f(X;\theta_*)\biggr)\biggr]\nonumber\\
&\qquad\qquad\times E\biggl[ \biggl(\frac{\partial \hfill}{\partial \theta^v} \log f(X;\theta_*)\biggr)\biggl(\frac{\partial \hfill}{\partial \theta^w} \log f(X;\theta_*)\biggr)\biggr]\bigg\}+O(n^{-3})\nonumber\\
&=n^{-2}\bbar{st}{i} \tilde{g}^{su}\tilde{g}^{tv}\tilde{g}^{jw}\tilde{g}^{km}(g_{uv}g_{wm}+g_{uw}g_{vm}+g_{um}g_{vw})+O(n^{-3}).\label{bbar_bareo^4}
\end{align}
%
%  Result of E[\barteta{i}\barteta{j}\barteta{k}]
%
Therefore 
\begin{equation}
\label{e_barteta_i*barteta_j*barteta_k}
\begin{split}
&E[(\hat{\theta}^i-\theta_*^i)(\hat{\theta}^j-\theta_*^j)(\hat{\theta}^k-\theta_*^k)]\\
&=n^{-2}\Bigl(\tilde{g}^{is}\tilde{g}^{jt}\tilde{g}^{ku}L_{stu} \\
&\quad\qquad+\tilde{g}^{it}\tilde{g}^{su}\tilde{g}^{jv}\tilde{g}^{kw}(L_{(st)u}g_{vw}+L_{(st)v}g_{uw}+L_{(st)w}g_{uv})\\
&\quad\qquad+\tilde{g}^{jt}\tilde{g}^{su}\tilde{g}^{iv}\tilde{g}^{kw}(L_{(st)u}g_{vw}+L_{(st)v}g_{uw}+L_{(st)w}g_{uv})\\
&\quad\qquad+\tilde{g}^{kt}\tilde{g}^{su}\tilde{g}^{iv}\tilde{g}^{jw}(L_{(st)u}g_{vw}+L_{(st)v}g_{uw}+L_{(st)w}g_{uv})\\
&\quad\qquad+\bbar{st}{i} \tilde{g}^{su}\tilde{g}^{tv}\tilde{g}^{jw}\tilde{g}^{km}(g_{uv}g_{wm}+g_{uw}g_{vm}+g_{um}g_{vw})\\
&\quad\qquad+\bbar{st}{j} \tilde{g}^{su}\tilde{g}^{tv}\tilde{g}^{iw}\tilde{g}^{km}(g_{uv}g_{wm}+g_{uw}g_{vm}+g_{um}g_{vw})\\
&\quad\qquad+\bbar{st}{k} \tilde{g}^{su}\tilde{g}^{tv}\tilde{g}^{iw}\tilde{g}^{jm}(g_{uv}g_{wm}+g_{uw}g_{vm}+g_{um}g_{vw})\Bigr)\\
&\qquad+O(n^{-3})\\
&=n^{-2}\Bigl(\tilde{g}^{is}\tilde{g}^{jt}\tilde{g}^{ku}L_{stu} \\
&\quad\qquad+\tilde{g}^{it}\tilde{g}^{su}\tilde{g}^{jv}\tilde{g}^{kw}(L_{(st)u}g_{vw}+L_{(st)v}g_{uw}+L_{(st)w}g_{uv})\\
&\quad\qquad+\tilde{g}^{jt}\tilde{g}^{su}\tilde{g}^{iv}\tilde{g}^{kw}(L_{(st)u}g_{vw}+L_{(st)v}g_{uw}+L_{(st)w}g_{uv})\\
&\quad\qquad+\tilde{g}^{kt}\tilde{g}^{su}\tilde{g}^{iv}\tilde{g}^{jw}(L_{(st)u}g_{vw}+L_{(st)v}g_{uw}+L_{(st)w}g_{uv})\\
&\quad\qquad+2^{-1} \tilde{g}^{su}\tilde{g}^{tv}\tilde{g}^{jw}\tilde{g}^{km}\tilde{g}^{il}L_{(stl)}(g_{uv}g_{wm}+g_{uw}g_{vm}+g_{um}g_{vw})\\
&\quad\qquad+2^{-1} \tilde{g}^{su}\tilde{g}^{tv}\tilde{g}^{iw}\tilde{g}^{km}\tilde{g}^{jl}L_{(stl)}(g_{uv}g_{wm}+g_{uw}g_{vm}+g_{um}g_{vw})\\
&\quad\qquad+2^{-1}\tilde{g}^{su}\tilde{g}^{tv}\tilde{g}^{iw}\tilde{g}^{jm}\tilde{g}^{kl}L_{(stl)}(g_{uv}g_{wm}+g_{uw}g_{vm}+g_{um}g_{vw})\Bigr)\\
&\qquad+O(n^{-3}).
\end{split}
\end{equation}
%
%
%   $E[\barteta{i}\barteta{j}\barteta{k}\barteta{l}]$
%
%
%
Finally we calculate $E[\barteta{i}\barteta{j}\barteta{k}\barteta{l}]$. Notice
\begin{equation}
\barteta{i}\barteta{j}\barteta{k}\barteta{l}=\bareo{i}\bareo{j}\bareo{k}\bareo{l}+Re4,
\end{equation}
where $Re4$ is defined as before. Therefore
\begin{align}
&E[\barteta{i}\barteta{j}\barteta{k}\barteta{l}]\\
&=n^{-4}E \biggl[\biggl(\tilde{g}^{is}\sum_{a=1}^n\biggl(\frac{\partial \hfill}{\partial \theta^s} \log f(X_a;\theta_*)\biggr)\biggr)\biggl(\tilde{g}^{jt}\sum_{b=1}^n\biggl(\frac{\partial \hfill}{\partial \theta^t} \log f(X_b;\theta_*)\biggr)\biggr)\nonumber\\
&\quad\times\biggl(\tilde{g}^{ku}\sum_{c=1}^n\biggl(\frac{\partial \hfill}{\partial \theta^u} \log f(X_c;\theta_*)\biggr)\biggr)\biggl(\tilde{g}^{lv}\sum_{d=1}^n\biggl(\frac{\partial \hfill}{\partial \theta^v} \log f(X_d;\theta_*)\biggr)\biggr)
\biggr]\nonumber\\
&=n^{-4}\tilde{g}^{is}\tilde{g}^{jt}\tilde{g}^{ku}\tilde{g}^{lv}\nonumber\\
&\qquad\times \sum_{a\ne b} \biggl\{ E\biggl[\frac{\partial \hfill}{\partial \theta^s} \log f(X_a;\theta_*)\frac{\partial \hfill}{\partial \theta^t} \log f(X_a;\theta_*)\biggr]E\biggl[\frac{\partial \hfill}{\partial \theta^u} \log f(X_b;\theta_*)
\frac{\partial \hfill}{\partial \theta^v} \log f(X_b;\theta_*)\biggr]\nonumber\\
&\quad\qquad\qquad+E\biggl[\frac{\partial \hfill}{\partial \theta^s} \log f(X_a;\theta_*)\frac{\partial \hfill}{\partial \theta^u} \log f(X_a;\theta_*)\biggr]E\biggl[\frac{\partial \hfill}{\partial \theta^t} \log f(X_b;\theta_*)
\frac{\partial \hfill}{\partial \theta^v} \log f(X_b;\theta_*) \biggr]\nonumber\\
&\quad\qquad\qquad+E\biggl[\frac{\partial \hfill}{\partial \theta^s} \log f(X_a;\theta_*)\frac{\partial \hfill}{\partial \theta^v} \log f(X_a;\theta_*)\biggr]E\biggl[\frac{\partial \hfill}{\partial \theta^t} \log f(X_b;\theta_*)
\frac{\partial \hfill}{\partial \theta^u} \log f(X_b;\theta_*) \biggr]\biggr\}\nonumber\\
&\qquad +O(n^{-3})\nonumber\\
&=n^{-2}\tilde{g}^{is}\tilde{g}^{jt}\tilde{g}^{ku}\tilde{g}^{lv}
(g_{st}g_{uv}+g_{su}g_{tv}+g_{sv}g_{tu})+O(n^{-3}).
\end{align}
Therefore we have
\begin{align}
&E[(\hat{\theta}^i-\theta_*^i)(\hat{\theta}^j-\theta_*^j)(\hat{\theta}^k-\theta_*^k)(\hat{\theta}^l-\theta_*^l)]\nonumber\\
&=n^{-2}\tilde{g}^{is}\tilde{g}^{jt}\tilde{g}^{ku}\tilde{g}^{lv}(g_{st}g_{uv}+g_{su}g_{tv}+g_{sv}g_{tu})+O(n^{-3}).
\end{align}
	\subsection{Proof of \eqref{expan_est_risk_exp}}
	\label{Proof_Coro}
	Note that, for $1\leq i, j ,k,l\leq p$,
	\begin{align*}
		g_{ij}^*(\theta_*) &= - \int  g(x;\theta_*)\Bigl( \frac{\partial^2 \hfill }{\partial \theta^i\partial \theta^j}\log g(x;\theta) \Big|_{\theta=\theta_*}\Bigr)d\mu \\
		&= \frac{\partial^2 \Psi(\theta) }{\partial \theta^i\partial \theta^j} \Big|_{\theta=\theta_*}=
		\tilde{g}_{ij}(\theta_*),
	\end{align*}
	and
	\[
	L_{(ij)kl} = -\tilde{g}_{ij}g_{kl},\qquad L_{(ij)(kl)}=\tilde{g}_{ij}\tilde{g}_{kl},\qquad L_{(ij)k}=0,\qquad L_{(ijk)l}=0.
	\]
	Combining these relations with \eqref{expan_est_risk} gives the following:
	\begin{equation}
		\label{asy_risk_exp_pre}
		\begin{split}
			&R[g(x;\theta_*)\,|\,g(x;\hat{\theta}) ] \\
			&= \frac{1}{2n} \mathrm{tr} \Bigl(\tilde{G}^{-1}G \Bigr) \\
			&\quad+\frac{1}{24n^2}\Bigl[\Bigl(12\tilde{g}^{uk}\tilde{g}^{ls}\tilde{g}^{mt}L_{kst}L_{(lmu)}\\
			&\qquad\qquad+12\tilde{g}^{ko}\tilde{g}^{lu}\tilde{g}^{sv}\tilde{g}^{tw}\tilde{g}^{hm}L_{(lmo)}L_{(sth)}
			(g_{ku}g_{vw}+g_{kv}g_{uw}+g_{kw}g_{uv})\\
			&\qquad\qquad+4\tilde{g}^{kw}\tilde{g}^{ls}\tilde{g}^{mu}\tilde{g}^{tv}L_{(lmtw)}
			(g_{ks}g_{uv}+g_{ku}g_{sv}+g_{kv}g_{su})\\
			&\qquad\qquad+3\tilde{g}^{lk}\tilde{g}^{mu}\tilde{g}^{sv}\tilde{g}^{tw}\tilde{g}^{oh}L_{(lmo)}L_{(sth)}
			(g_{ku}g_{vw}+g_{kv}g_{uw}+g_{kw}g_{uv})
			\Bigr)\\
			&\qquad\qquad-\tau_{ijk}\Bigl(4\tilde{g}^{is}\tilde{g}^{jt}\tilde{g}^{ku}L_{stu} \\
			&\qquad\qquad+6\tilde{g}^{su}\tilde{g}^{tv}\tilde{g}^{jw}\tilde{g}^{km}\tilde{g}^{il}L_{(stl)}(g_{uv}g_{wm}+g_{uw}g_{vm}+g_{um}g_{vw})\Bigr)\\
			&\qquad\qquad-\tau_{ijkl}\tilde{g}^{is}\tilde{g}^{jt}\tilde{g}^{ku}\tilde{g}^{lv}(g_{st}g_{uv}+g_{su}g_{tv}+g_{sv}g_{tu})
			\Bigr]\\
			&\quad + O(n^{-3}).\\
		\end{split}
	\end{equation}
	As
	\[
	\tau_{ijk} = -\frac{\partial^3 \Psi}{\partial \theta^i \partial \theta^j \partial \theta^k}
	= L_{(ijk)},\qquad 
	\tau_{ijkl} = -\frac{\partial^4 \Psi}{\partial \theta^i \partial \theta^j \partial \theta^k 
		\partial \theta^l} = L_{(ijkl)},
	\]
	for $1\leq i, j,k,l \leq p$, 
	\begin{equation}
		\begin{split}
			&R[g(x;\theta_*)\,|\,g(x;\hat{\theta}) ] \\
			&= \frac{1}{2n} \mathrm{tr} \Bigl(\tilde{G}^{-1}G \Bigr) \\
			&\quad+\frac{1}{24n^2}\Bigl[8\tilde{g}^{uk}\tilde{g}^{ls}\tilde{g}^{mt}L_{kst}L_{(lmu)}\\
			&\qquad\qquad+9\tilde{g}^{ko}\tilde{g}^{lu}\tilde{g}^{sv}\tilde{g}^{tw}\tilde{g}^{hm}L_{(lmo)}L_{(sth)}
			(g_{ku}g_{vw}+g_{kv}g_{uw}+g_{kw}g_{uv})\\
			&\qquad\qquad+3\tilde{g}^{kw}\tilde{g}^{ls}\tilde{g}^{mu}\tilde{g}^{tv}L_{(lmtw)}
			(g_{ks}g_{uv}+g_{ku}g_{sv}+g_{kv}g_{su})\Bigr]\\
			&\quad + O(n^{-3}).\\
		\end{split}
	\end{equation}
	Substituting the cumulant expression \eqref{cumu_exp} gives
	\begin{equation}
		\begin{split}
			&R[g(x;\theta_*)\,|\,g(x;\hat{\theta}) ] \\
			&= \frac{1}{2n} \mathrm{tr} \Bigl(\tilde{G}^{-1}G \Bigr) \\
			&\quad+\frac{1}{24n^2}\Bigl[-8\tilde{g}^{uk}\tilde{g}^{ls}\tilde{g}^{mt}\kappa_{kst}\kappa^*_{lmu}\\
			&\qquad\qquad+9\tilde{g}^{ko}\tilde{g}^{lu}\tilde{g}^{sv}\tilde{g}^{tw}\tilde{g}^{hm}\kappa^*_{lmo}\kappa^*_{sth}
			(g_{ku}g_{vw}+g_{kv}g_{uw}+g_{kw}g_{uv})\\
			&\qquad\qquad-3\tilde{g}^{kw}\tilde{g}^{ls}\tilde{g}^{mu}\tilde{g}^{tv}\kappa^*_{lmtw}
			(g_{ks}g_{uv}+g_{ku}g_{sv}+g_{kv}g_{su})\Bigr]\\
			&\quad + O(n^{-3}).
		\end{split}
	\end{equation}
	\subsection{Derivation of \eqref{Risk_norm_case1}, \eqref{Risk_norm_case2}, \eqref{Risk_norm_case3}}
	\label{nomral_risk_deriv}
	Information projection is given by the solution 
	\begin{align*}
		E\Bigl[\frac{\partial \hfill}{\partial \theta^i} \log g(Y,X;\theta) \Bigr]
		&= 
		\begin{cases}
			-2^{-1}E[\epsilon^2]+2^{-1}(\theta^0)^{-1}=0 ,\quad  \text{ if $i=0$},\\
			\theta^0E[\epsilon X_i] = 0, \quad \text{ if $i=1,\ldots,p$}.
		\end{cases}
	\end{align*}
	In other words, $g(y,x ; \theta_*)$ is given by $\theta_*=(\theta_*^0,\ldots,\theta_*^p)$, which satisfies
	\[
	(\theta_*^0)^{-1} = E[\epsilon^2(Y,X;\theta_*)],\ \quad E[X_i \epsilon(Y,X;\theta_*)] = 0,\ i=1,\ldots p.
	\]
	Note that
	\begin{align*}
		\tilde{g}_{00}(\theta_*) &= \frac{1}{2(\theta^0_*)^2}, \\
		\tilde{g}_{0i}(\theta_*) &=-E[X_i \epsilon(Y,X;\theta_*)]=0,\quad i=1,\ldots,p \\
		\tilde{g}_{ij}(\theta_*) &= \theta^0_* s_{ij},\quad s_{ij} \triangleq E[X_i X_j],\quad i,j=1,\ldots,p\\
		g^*_{00} (\theta_*)&= \frac{1}{2(\theta^0_*)^2} \\
		g^*_{0i} (\theta_*)&=-E_{\theta_*}[X_i \epsilon(Y,X;\theta_*)]=0,\quad i=1,\ldots,p \\
		g^*_{ij}(\theta_*) &= \theta_*^0  E_{\theta_*}[X_i X_j]=\theta_*^0 E[X_i X_j]=\theta_*^0 s_{ij},\quad i,j=1,\ldots,p.\\
	\end{align*}
	Hence, $\tilde{G}=G^*$ and 
	\[
	\tilde{G}^{-1} = 
	\begin{pmatrix}
		2(\theta^0_*)^2 &  0 \\
		0 & (\theta^0_*)^{-1} S^{-1}
	\end{pmatrix}
	, \qquad S=(s_{ij}).
	\]
	Further,  we have
	\begin{align*}
		g_{00}(\theta_*) &= \frac{1}{4}E\Bigl[\Bigl(\epsilon^2(Y,X;\theta_*)-(\theta^0_*)^{-1}\Bigr)^2\Bigr]\\
		&=\frac{1}{4}E[\epsilon^4(Y,X;\theta_*)]-\frac{1}{2}(\theta^0_*)^{-1}E[\epsilon^2(Y,X;\theta_*)]+\frac{1}{4}(\theta^0_*)^{-2}\\
		&=\frac{1}{4}\Bigl(E[\epsilon^4(Y,X;\theta_*)]-E^2[\epsilon^2(Y,X;\theta_*)]\Bigr)\\
		g_{ij}(\theta_*) & =(\theta^0_*)^2 t_{ij},\qquad t_{ij} \triangleq E[X_i X_j \epsilon^2(Y,X;\theta_*)],\quad i,j=1,\ldots,p.\\
	\end{align*}
	Consequently, for Case1,
	\begin{align*}
		&R[g(x;\theta_*)\,|\,g(x;\hat{\theta}) ] \\
		&= \frac{1}{2n} \mathrm{tr} \Bigl(\tilde{G}^{-1}G\tilde{G}^{-1}G^* \Bigr)  + o(n^{-1})\\
		&= \frac{1}{2n} \mathrm{tr} \Bigl(\tilde{G}^{-1}G \Bigr)  + o(n^{-1})\\
		&= \frac{1}{2n} \Bigl( \theta^0_* \mathrm{tr} \Bigl(S^{-1} T\Bigr) 
		+\frac{(\theta^0_*)^2}{2}\Bigl(E[\epsilon^4(Y,X;\theta_*)]-E^2[\epsilon^2(Y,X;\theta_*)]\Bigr)\Bigr)+ o(n^{-1})\\
		&=\frac{1}{2n} \Bigl(  \mathrm{tr} \Bigl(S^{-1} T\Bigr) /E[\epsilon^2(Y,X;\theta_*)]
		+\frac{1}{2}\Bigl(E[\epsilon^4(Y,X;\theta_*)]/E^2[\epsilon^2(Y,X;\theta_*)]-1\Bigr)\Bigr)+ o(n^{-1}),
	\end{align*}
	where $(T)_{ij} = t_{ij}$.  
	
	For Case 2, we observe that
	\[
	t_{ij} = E[X_i X_j] E[\epsilon^2(Y,X;\theta_*)], \quad i,j =1,\ldots,p;
	\]
	hence,
	\[
	T = E[\epsilon^2(Y,X;\theta_*)] S,
	\]
	and 
	\[
	R[g(x;\theta_*)\,|\,g(x;\hat{\theta}) ]= \frac{1}{2n} \Bigl(p +\frac{1}{2}\Bigl(E[\epsilon^4(Y,X;\theta_*)]/E^2[\epsilon^2(Y,X;\theta_*)]-1\Bigr)\Bigr)+ o(n^{-1}).
	\]
	For Case 3, as 
	\[
	E[\epsilon^4(Y,X;\theta_*)]/E^2[\epsilon^2(Y,X;\theta_*)]=3,
	\]
	we have
	\[
	R[g(x;\theta_*)\,|\,g(x;\hat{\theta}) ]= \frac{p+1}{2n} + o(n^{-1}).
	\]
	\subsection{Proof of \eqref{rel_error_D}}
	\label{dive_error} 
	A suitably fine partition $S_i,\ i=1,\ldots,m$ of the domain of $d\mu$ and the associated step functions $\tilde{g}_j(x) = \sum_{i=1}^m c_{ji} I(x \in S_i),\ j=1,2$ are taken such that  the two integrations 
	\begin{align}
		&Er[g_1(x)\, | \, g_2(x)] = \frac{1}{2}\int \min\Bigl(g_1(x), g_2(x)\Bigr) d\mu = \frac{1}{2}\int g_1(x) \min\Bigl(1, g_2(x)/g_1(x)\Bigr) d\mu, \\
		&D[g_1(x) \,| \, g_2(x)] = \int g_1(x) \log \Bigl( g_1(x)/g_2(x) \Bigr) d\mu,
	\end{align}
	are sufficiently well approximated by
	\begin{align}
		&\frac{1}{2}\int \tilde{g}_1(x) \min\Bigl(1, \tilde{g}_2(x)/\tilde{g}_1(x)\Bigr) d\mu =\frac{1}{2}\sum_{i=1}^m  \min (1, c_{2i}/c_{1i}) \int_{S_i} c_{1i} d\mu, \label{appro_error}\\
		&\int \tilde{g}_1(x) \log \Bigl( \tilde{g}_1(x)/\tilde{g}_2(x) \Bigr) d\mu = \sum_{i=1}^m \log (c_{1i}/c_{2i} ) \int_{S_i} c_{1i} d\mu, \label{appro_dive}
	\end{align}  
	respectively. Furthermore, we can choose the partition such that 
	\[
	\int_{S_i} c_{1i} d\mu = 1/m, \quad i=1,\ldots,m.
	\]
	Then, \eqref{appro_error} and \eqref{appro_dive} equal 
	\begin{align*}
		&\frac{1}{2m}\sum_{i=1}^m \min (1, \Delta_i) \ (\triangleq t(\Delta_1,\ldots,\Delta_m) ) \\
		&\frac{1}{m}\sum_{i=1}^m -\log{\Delta_i},
	\end{align*} 
	where $\Delta_i \triangleq c_{2i}/c_{1i},\ i=1,\ldots,m$. Suppose that $D[g(X;\theta_1) \,| \, g(x;\theta_2)] < \delta$. Then, we can suppose 
\begin{equation}
		\label{diver_cond}
		f(\Delta_1,\ldots, \Delta_m) \triangleq \frac{1}{m} \sum_{i=1}^m \log \Delta_i \geq -\delta.
	\end{equation}
	The lower bound of $t(\Delta)$ is searched for, under the condition of \eqref{diver_cond}. Let 
	\begin{equation}
		\label{delta_sum}
		\tilde{m} \triangleq \sum_{i=1}^m \Delta_i,\qquad \tilde{1} \triangleq  \frac{\tilde{m}}{m} .
	\end{equation}
	Note that, as the partition $S_i,\ i=1,\ldots,m$ becomes finer,
	\[
	\sum_{i=1}^m \int_{S_i}  c_{2i} d\mu =\sum_{i=1}^m \Delta_i/m = \tilde{1} \to \int g_2(x) d\mu  =1.
	\]

	Without loss of generality, the following can be assumed:   
	\[
	\Delta_1 \geq  \cdots \geq \Delta_s > 1 > \Delta_{s+1} \geq \cdots \geq \Delta_m > 0,\quad \exists s (\geq 1). 
	\]
	Let $u=m-s$ and 
	\[
	\Delta^+ \triangleq  \frac{1}{s} \sum_{i=1}^s \Delta_i,\qquad \Delta^- \triangleq \frac{1}{u} \sum_{i=s+1}^m \Delta_{i}.
	\]
	Note that 
	\[
	t(\underbrace{\Delta^+, \cdots, \Delta^+}_s, \underbrace{\Delta^-,\cdots, \Delta^-}_u) = t(\Delta_1,\ldots, \Delta_m) \\
	\]
	and, because of the concavity of $\log(\Delta_i)$, 
	\[
	f(\underbrace{\Delta^+, \cdots, \Delta^+}_s, \underbrace{\Delta^-,\cdots, \Delta^-}_u) \geq f(\Delta_1,\ldots, \Delta_m) \geq -\delta.
	\]
	Therefore, in search of the lower bound of $t(\Delta)$, we must only consider the case where
	\begin{equation}
		\label{delta_uniform}
		\begin{split}
			&\Delta_1 = \Delta_2 = \cdots = \Delta_s  = \Delta^+ > 1, \\
			&0 < \Delta_{s+1} = \Delta_{s+2} = \cdots = \Delta_m =\Delta ^- < 1,
		\end{split}
	\end{equation}

	Under condition \eqref{delta_uniform}, the relations \eqref{diver_cond} and \eqref{delta_sum}  are 
	\begin{align*}
		&\frac{1}{m} (s\log{\Delta^+}+ u\log{\Delta^{-}}) \geq -\delta, \\
		& s\Delta^+ + u \Delta^{-} = \tilde{m},\\
	\end{align*}
	respectively, or equivalently,
	\begin{align}
		&  x \log{\Delta^+} + (1-x) \log{\Delta^-} \geq -\delta \label{diver_cond2},\\
		& x\Delta^+ +(1-x) \Delta^- =  \tilde{1}, \label{delta_sum2}
	\end{align}
	where 
	\begin{equation}
		\label{domain_x}
		0< x=s/m <1.
	\end{equation}
	Substituting the relation from \eqref{delta_sum2}, i.e.,
	\[
	\Delta^- = \frac{\tilde{1}-x\Delta^+}{1-x}
	\]
	into $\Delta^- > 0$ and  \eqref{diver_cond2} gives 
	\begin{align}
		& 1 < \Delta^+ < \frac{\tilde{1}}{x} \label{x_Delta}\\
		&h(x;\Delta^+) \triangleq x \log{\Delta^+} +(1-x) \log{\Bigl(\frac{\tilde{1}-x\Delta^+}{1-x}\Bigr)} \geq -\delta. \label{diver_cond3}
	\end{align}
	Furthermore, under condition \eqref{delta_uniform},
	\begin{align*}
		\frac{1}{2m}\sum_{i=1}^m \min (1, \Delta_i)& = t(\Delta_1,\ldots,\Delta_m) \\
		& = \frac{1}{2m}( s + u\Delta^- ) \\
		& =\frac{1}{2}( x + (1-x) \Delta^- ) \\
		&= \frac{1}{2} \bigl(\tilde{1}+x(1-\Delta^+)\bigr) \ \bigl(\triangleq t(x;\Delta^+) \bigr)
	\end{align*}
	
	Consider the minimization of $t(x;\Delta^+)$ under conditions \eqref{domain_x}, \eqref{x_Delta}, and \eqref{diver_cond3}.
	As
	\begin{align*}
		\frac{d }{dx} h(x;\Delta^+) = h'(x;\Delta^+) &= \log{\Delta^+} -\log{\Bigl(\frac{\tilde{1}-x\Delta^+}{1-x}\Bigr)}+(1-x)\Bigl\{\frac{-\Delta^+}{\tilde{1}-x\Delta^+}+\frac{1}{1-x}\Bigr\}\\
		&=\log{\Bigl(\frac{\Delta^+(1-x)}{\tilde{1}-x\Delta^+}\Bigr)} + \frac{\tilde{1}-\Delta^+}{\tilde{1}-x\Delta^+}\\
		&\leq \frac{\Delta^{+}-\tilde{1}}{\tilde{1}-x\Delta^{+}}+\frac{\tilde{1}-\Delta^{+}}{\tilde{1}-x\Delta^{+}} = 0 \qquad \text{($\because \log(1+x) \leq x$)},
	\end{align*}
	the minimum value of $t(x;\Delta^+)$(say, $t^*$) is attained when \eqref{diver_cond3} holds with the equation. Let $x^*$ denote the point that attains $t^*$; then, 
	\begin{equation}
		\label{del_by_t_x}
		\Delta^+ = (\tilde{1}-2t^*)/x^* +  1.
	\end{equation}
	Inserting \eqref{del_by_t_x} into the left-hand side of \eqref{diver_cond3} and equating it with $-\delta$ gives 
	\[
	x^*\log{\Bigl(\frac{\tilde{1}-2t^*}{x^*} + 1\Bigr)}+(1-x^*) \log{\Bigl(\frac{2t^*-1}{1-x^*}+1\Bigr)} = -\delta,
	\]
	while, from \eqref{domain_x}, \eqref{x_Delta}, and \eqref{del_by_t_x}, 
	\[
	0 < x^* < 2t^* < \tilde{1}.
	\]
	Let us define the region $\tilde{A}(\delta)$  by
	\[
	\tilde{A}(\delta) \triangleq \Bigl\{(x^*,t^*)\, \Big| \, x^*\log{\Bigl(\frac{\tilde{1}-2t^*}{x^*} + 1\Bigr)}+(1-x^*) \log{\Bigl(\frac{2t^*-1}{1-x^*}+1\Bigr)} = -\delta,\quad 0 < x^* < 2t^* < \tilde{1}.\Bigr\}
	\]
	Then, 
	\[
	\frac{1}{2m}\sum_{i=1}^m \min (1, \Delta_i) = t(x;\Delta^+) \geq  \min{\{t^* \,|\, (x^*,t^*) \in \tilde{A}(\delta)\}}.
	\]
	Taking the limit operation for both sides as the partition becomes finer gives  the result.
	\subsection{Proof of \eqref{bias_Ce}}
	\label{proof_bias_Ce} 
	For simplicity, the notation $\doteqdot$ is used when the terms $o(n^{-1)}$ or $o_ p(n^{-1})$ are ignored.
	The expansion of $\log{g(x_t;\hat{\theta})}$ around $\theta_*$ is given by
	\begin{equation}
		\label{expan_log_g}
		\begin{split}
			\log{g(x_t;\hat{\theta})} &\doteqdot \log{g(x_t;\theta_*)} + \sum_{i=1}^p \frac{\partial\ }{\partial \theta^i} \log{g(x_t;\theta)}\bigg |_{\theta=\theta_*} (\hat{\theta}^i- \theta_*^i) \\
			& \quad +\frac{1}{2} \sum_{1\leq i, j \leq p} \frac{\partial^2\ \  }{\partial \theta_i \partial \theta_j} \log{g(x_t;\theta)}\bigg |_{\theta=\theta_*} (\hat{\theta}^i- \theta_*^i)(\hat{\theta}^j- \theta_*^j)\\
			&= \log{g(x_t;\theta_*)} + \sum_{i=1}^p (\xi_i(x_t)-\eta_i^*)(\hat{\theta}^i- \theta_*^i) \\
			& \quad -\frac{1}{2} \sum_{1\leq i, j \leq p} (\Ddot{\Psi})_{ij}(\hat{\theta}^i- \theta_*^i)(\hat{\theta}^j- \theta_*^j).\\
		\end{split}
	\end{equation}
	Meanwhile, $\hat{\theta}^i = \theta^i(\hat\eta)$ can be expanded around $\eta^*(=\eta(\theta_*))$, as
	\begin{equation}
		\label{exapn_theta}
		\theta^i(\hat{\eta}) \doteqdot \theta^i(\eta^*) + \sum_{j=1}^p \frac{\partial\theta^i}{\partial \eta_j} (\hat{\eta}_j- \eta_j^*) = \theta^i(\eta^*) + \sum_{j=1}^p(\Ddot\Phi)_{ij}(\hat{\eta}_j- \eta_j^*),
	\end{equation}
	where $\Phi(\eta)$ is the conjugate convex function of $\Psi$, which satisfies the following relations:
	\begin{align}
		&\frac{\partial \Phi }{\partial \eta_i} (\eta) = \theta^i, \quad i=1,\ldots,p\\
		&\Ddot\Phi  \triangleq \biggl(\frac{\partial^2 \Phi}{\partial \eta_i \partial \eta_j}\biggr) = {\Ddot\Psi}^{-1}. \label{Phi_Psi}
	\end{align}
	Inserting \eqref{exapn_theta} and \eqref{Phi_Psi} into \eqref{expan_log_g} gives
	\begin{equation}
		\label{log_g_expan_m}
		\begin{split}
			\log{g(x_t;\hat{\theta})} &\doteqdot \log{g(x_t;\theta_*)} + \sum_{1\leq i, j \leq p}\bigl(\Ddot\Psi^{-1}\bigr)_{ij} (\xi_i(x_t)-\eta_i^*)(\hat{\eta}_j- \eta^*_j)\\
			& \quad -\frac{1}{2} \sum_{1\leq i, j \leq p} (\Ddot{\Psi})_{ij}(\hat{\theta}^i- \theta_*^i)(\hat{\theta}^j- \theta_*^j).
		\end{split}
	\end{equation}
	Taking the expectation for both sides gives
	\begin{equation}
		\label{E_log_g_expan}
		\begin{split}
			E[\log{g(X_t;\hat{\theta})}] &\doteqdot E\log{g(X_t;\theta_*)}] + \sum_{1\leq i, j \leq p}\bigl(\Ddot\Psi^{-1}\bigr)_{ij}E[ (\xi_i(X_t)-\eta_i^*)(\hat{\eta}_j- \eta^*_j)]\\
			& \quad -\frac{1}{2} \sum_{1\leq i, j \leq p} (\Ddot{\Psi})_{ij}E[(\hat{\theta}^i- \theta_*^i)(\hat{\theta}^j- \theta_*^j)]
		\end{split}
	\end{equation}
	Note that 
	\begin{equation}
		\label{E_xi_eta}
		\begin{split}
			E[ (\xi_i(X_t)-\eta_i^*)(\hat{\eta}_j- \eta^*_j)] &= n^{-1}\sum_{s=1}^n E[ (\xi_i(X_t)-\eta_i^*) (\xi_j(X_s)-\eta^*_j) ]\\
			& = n^{-1}E[ (\xi_i(X_t)-\eta_i^*) (\xi_j(X_t)-\eta^*_j)]\\
			& = n^{-1} (G)_{ij}
		\end{split}
	\end{equation}
	as, for $ s \ne t$, 
	\[
	E[ (\xi_i(X_t)-\eta_i^*) (\xi_j(X_s)-\eta^*_j) = E[ \xi_i(X_t)-\eta_i^*]E[\xi_j(X_s)-\eta^*_j]=0.
	\]
	From (80) in Section \ref{Proof_Theo} of the Appendix, 
	\begin{equation}
		\label{E_theta_ij}
		\begin{split}
			&E[(\hat{\theta}^i-\theta_*^i)(\hat{\theta}^j-\theta_*^j)]\\
			&\doteqdot n^{-1}\sum_{1\leq l,m \leq p}\tilde{g}^{il}\tilde{g}^{jm}g_{lm} \\
			&= n^{-1}\bigl(\tilde{G}^{-1} G \tilde{G}^{-1}\bigr)_{ij}.
		\end{split}
	\end{equation}
	Inserting \eqref{E_xi_eta} and \eqref{E_theta_ij} into \eqref{E_log_g_expan} and using the fact that 
	$\tilde{G} = \Ddot\Psi$ gives
	\begin{align*}
		E[\log{g(X_t;\hat{\theta})}] &\doteqdot E[\log{g(X_t;\theta_*)}] + \frac{1}{n}\mathrm{tr}\bigl(\tilde{G}^{-1} G \bigr).
		-\frac{1}{2n}\mathrm{tr}\bigl(\tilde{G}\tilde{G}^{-1} G \tilde{G}^{-1}\bigr)\\
		& = E[\log{g(X;\theta_*)}] + \frac{1}{2n}\mathrm{tr}\bigl({\tilde{G}}^{-1} G \bigr),
	\end{align*}
	and
	\begin{align*}
		E[\widehat{Ce(M)}] - Ce(M) &= -\frac{1}{n} \sum_{t=1}^n E[\log{g(X_t;\hat{\theta})}] + E[\log{g(X;\theta_*)}]\\
		&= -\frac{1}{n} \sum_{t=1}^n \Bigl\{ E[\log{g(X_t;\hat{\theta})}] -  E[\log{g(X;\theta_*)}]\Bigr\}\\
		&= - \frac{1}{2n} \mathrm{tr}\bigl({\tilde{G}}^{-1} G \bigr).
	\end{align*}
\end{document}